\documentclass[12pt]{amsart}

\usepackage[margin=1in,marginparwidth=0.8in, marginparsep=0.1in]{geometry}

\pdfoutput=1

\usepackage{soul}
\usepackage{amsfonts, amsthm, amssymb, latexsym, amsmath, amscd, graphicx, bbm}
\usepackage{MnSymbol}

\usepackage[pagebackref, bookmarks=true, bookmarksopen=true,%
bookmarksdepth=3,bookmarksopenlevel=2,%
colorlinks=true,%
linkcolor=blue,%
citecolor=blue,%
filecolor=blue,%
menucolor=blue,%
urlcolor=blue]{hyperref}

\usepackage[all, knot, poly]{xy}
\xyoption{knot}

\usepackage{times}
\usepackage{stackrel}
\usepackage{tikz}
\usetikzlibrary{matrix,shapes,arrows,calc,topaths,intersections,positioning,decorations.pathreplacing,decorations.pathmorphing,fit}

\newtheorem{lemma}{Lemma}[section]
\newtheorem{proposition}[lemma]{Proposition}
\newtheorem{definition/proposition}[lemma]{Definition/Proposition}
\newtheorem{corollary}[lemma]{Corollary}
\newtheorem{theorem}[lemma]{Theorem}

\theoremstyle{definition}
\newtheorem{definition}[lemma]{Definition}

\newtheorem{convention}[lemma]{Convention}
\newtheorem{example}[lemma]{Example}

\newtheorem*{warning*}{Warning}
\newtheorem{problem}[lemma]{Problem}

\theoremstyle{remark} 
\newtheorem{remark}[lemma]{Remark}

\newcommand{\C}{\mathbb{C}}
\newcommand{\D}{\mathbb{D}}

\renewcommand{\L}{\mathbb{L}}

\renewcommand{\P}{\mathbb{P}}

\newcommand{\R}{\mathbb{R}}

\newcommand{\T}{\mathbb{T}}

\newcommand{\Z}{\mathbb{Z}}

\newcommand{\cA}{\mathcal{A}}

\newcommand{\cC}{\mathcal{C}}
\newcommand{\cD}{\mathcal{D}}
\newcommand{\cE}{\mathcal{E}}
\newcommand{\cF}{\mathcal{F}}
\newcommand{\cG}{\mathcal{G}}

\newcommand{\cK}{\mathcal{K}}
\newcommand{\cL}{\mathcal{L}}
\newcommand{\cM}{\mathcal{M}}

\newcommand{\cP}{\mathcal{P}}

\newcommand{\cX}{\mathcal{X}}

\newcommand{\bR}{\mathbb{R}}

\newcommand{\bZ}{\mathbb{Z}}

\newcommand\into{\hookrightarrow}

\newcommand*{\hrlen}{10}
\newcommand*{\hramp}{3}

\tikzset{
asdstyle/.style={blue,thick},
righthairs/.style={postaction={decorate,draw,decoration={border,amplitude=\hramp,segment length=\hrlen,angle=-90,pre=moveto,pre length=\hrlen/2}}},
lefthairs/.style={postaction={decorate,draw,decoration={border,amplitude=\hramp,segment length=\hrlen,angle=90,pre=moveto,pre length=\hrlen/2}}},
righthairsnogap/.style={postaction={decorate,draw,decoration={border,amplitude=\hramp,segment length=\hrlen,angle=-90}}},
lefthairsnogap/.style={postaction={decorate,draw,decoration={border,amplitude=\hramp,segment length=\hrlen,angle=90}}},
graphstyle/.style={thick},
arrowstyle/.style={thick,decorate,decoration={snake,amplitude=1.7,segment length=10pt,post length=.5mm,pre length=0}},
genmapstyle/.style={thick,-stealth'},
arrhdstyle/.style={thick},
exceptarcstyle/.style={red, ultra thick},
dualquiverstyle/.style={thick,->}
}

\DeclareMathOperator{\Spec}{Spec}

\DeclareMathOperator{\Hom}{Hom}

\DeclareMathOperator{\Id}{Id}

\newcommand{\Fuk}{\mathit{Fuk}}

\newcommand{\ghat}{\scalebox{1.5}{$\upmodels$}}
\newcommand{\mutghat}{\scalebox{1.5}{$\downmodels$}}
\newcommand{\TC}{T^+_{\cC}\cL}
\newcommand{\TCk}{T^+_{\cC_{(k)}}\cL_{(k)}}
\newcommand{\coeffs}{\mathbbm{k}}
\newcommand{\muloc}{\mu \mathit{loc}}
\newcommand{\muLoc}{\mu \mathit{Loc}}
\newcommand{\Mut}{\mathit{Mut}}
\newcommand{\loc}{\mathit{loc}}
\newcommand{\Loc}{\mathit{Loc}}
\newcommand{\congto}{\xrightarrow{\sim}}
\newcommand\dmod{\textrm{-mod}}
\newcommand{\sh}{\mathit{sh}}

\makeatletter
\newcommand{\bigperp}{%
  \mathop{\mathpalette\bigp@rp\relax}%
  \displaylimits
}

\newcommand{\bigp@rp}[2]{%
  \vcenter{
    \m@th\hbox{\scalebox{\ifx#1\displaystyle2.1\else1.5\fi}{$#1\perp$}}
  }%
}
\makeatother

\newcommand{\eeT}{{\bigperp}}

\newcommand{\lcirc}{\rotatebox[origin=c]{90}{$\circlearrowleft$}}

\author{Vivek Shende, David Treumann, and Harold Williams}
\title{On the  combinatorics of exact Lagrangian surfaces}
\date{}

\begin{document}

\begin{abstract}
We study Weinstein 4-manifolds which admit Lagrangian skeleta given by attaching 
disks to a surface along a collection of simple closed curves.  In terms of the curves
describing one such skeleton, we describe surgeries that preserve the ambient Weinstein manifold, but
change the skeleton.  The surgeries can be iterated to produce more
such skeleta  --- in many cases, infinitely many more. 

Each skeleton is built around a Lagrangian surface.  Passing to the Fukaya category, 
the skeletal surgeries induce cluster transformations on the spaces of rank one local systems on these surfaces,
and noncommutative analogues of cluster transformations on the spaces of higher rank local systems. 
In particular, the problem of producing and distinguishing such Lagrangians 
maps to a combination of combinatorial-geometric questions about 
curve configurations on surfaces and algebraic questions
about exchange graphs of cluster algebras. 

Conversely, this expands the dictionary relating the cluster theory of character varieties, positroid strata, and related spaces to the symplectic geometry of Lagrangian fillings of Legendrian knots, by incorporating cluster charts more general than those associated to bicolored surface graphs. 
\end{abstract}

\pagenumbering{Alph}
\maketitle
\thispagestyle{empty}
\pagenumbering{arabic}

\vspace{10mm}

\begin{center}
\includegraphics{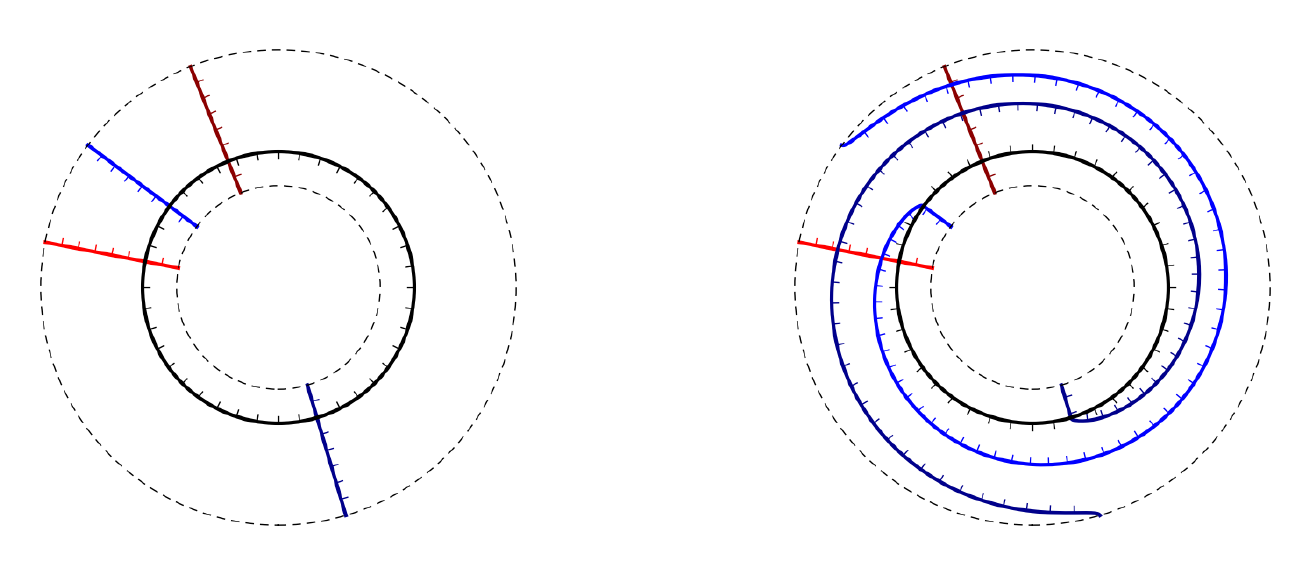}
\end{center}

\newpage



\vspace{10mm}


\thispagestyle{empty}

\newpage

\section{Introduction}\label{sec:intro}

The zero section of a cotangent bundle is the prototypical example of 
an exact Lagrangian embedded in an exact symplectic manifold.  Producing and distinguishing such Lagrangians is a 
basic problem 
in symplectic geometry.  One benchmark is the standing conjecture of Arnol'd that the zero section gives the only compact exact Lagrangian in a cotangent bundle, up to Hamiltonian isotopy.  
Together with the Weinstein neighborhood theorem, the Arnol'd conjecture 
suggests that exact Lagrangians have a discrete
nature:  up to Hamiltonian isotopy, they should have no moduli.

Our purpose here is to introduce a mechanism 
for producing and distinguishing large collections of exact Lagrangian surfaces. 
We assume we are given one exact Lagrangian surface, $\cL$, to which a collection of 
Lagrangian disks are attached along smooth circles, forming a singular Lagrangian $\L$.  
We work in a neighborhood $W$ of the skeleton $\L$. 
After collapsing one of the attached disks so that $\cL$ acquires a singularity, 
there are two choices of Lagrangian surgery \cite{LS,Pol} ---  one returning the original
surface $\cL$, and one 
yielding a new exact Lagrangian surface $\cL'$ which is smoothly, but not Hamiltonianly,
isotopic to the original surface.  The transition $\cL \rightsquigarrow \cL'$ is the Lagrangian disk surgery of M.-L. Yau 
\cite{Y}.  

The basic geometric contribution of this paper is to explain how the entire skeleton may be carried 
through this transition, giving a new skeleton $\L'$ extending the new Lagrangian surface.
This skeletal surgery $\L \rightsquigarrow \L'$ in turn has a combinatorial description in terms of operations on configurations of curves in $\cL$ --- the projections of the attaching Legendrians used to build the skeleta. 
For constructing
exact Lagrangians, the point is that this procedure can now be iterated: we can use the disks in $\L'$ 
to perform surgeries on $\cL'$, and get more exact Lagrangians.

This branching 
production of Lagrangians by a local surgery procedure --- potential sequences of surgeries are indexed by an $n$-ary tree, if $n$ disks were attached --- geometrizes the notion of quiver mutation. Indeed, $\L$ defines a quiver: its vertices index the curves along which disks are attached, and its arrows record the intersection numbers of their projections to $\cL$. 
This quiver undergoes a mutation when we perform a skeletal surgery. 

Associated to a quiver and its mutations is a cluster algebra \cite{FZ} --- the coordinate ring of a space built from open algebraic tori (cluster charts) labeling the vertices of the $n$-ary tree, glued along certain birational maps (cluster transformations). 
These arise in various contexts such as canonical bases and total positivity in Lie theory \cite{Fom,GLS} and character varieties of punctured surfaces \cite{FG,GSV2}. 

In our setting this structure appears as follows: $\L$ carries a sheaf of categories $\muloc$;
we will be interested in the global sections $\muloc(\L)$, which we term microlocal sheaves on $\L$. 
We will show that 
skeletal surgery $\L \rightsquigarrow \L'$ induces an equivalence $\muloc(\L) \cong \muloc(\L')$.
On the other hand, there is a natural inclusion of the category of local systems on the original surface into the category
of microlocal sheaves on the skeleton.  
Of particular interest are the rank one local systems, the category of which we write as $Loc_1(\cL)$; 
note these are parameterized by an algebraic torus.
Thus the skeletal surgery induces a comparison of algebraic tori: 
\begin{equation}\label{eq:clustertransintro}
Loc_1(\cL) \subset \muloc(\L) \cong \muloc(\L') \supset Loc_1(\cL').
\end{equation}
We show this to be the cluster $\cX$-transformation associated to the 
quiver mutation described above. 
More generally, the corresponding comparison on higher rank local systems is given by a nonabelian version of a cluster transformation.  

In particular, it follows that the images of $Loc_1(\cL)$ and $Loc_1(\cL')$ in $\muloc(\L)$  are different.
This fact holds geometric significance. 
Indeed, 
according to a conjecture of Kontsevich \cite{K}, perfect modules over the wrapped Fukaya category 
of $W$ are the global sections of a certain sheaf of categories over $\L$.  This is known in the case of cotangent bundles \cite{NZ,  N1, FSS}.  More generally,  the expected sheaf can be described explicitly in terms of the microlocalization 
theory of Kashiwara and Schapira \cite{KS}; it is our sheaf $\muloc$. 
The inclusion $loc(\cL) \to \muloc(\L)$  corresponds to 
the pullback of perfect module categories
along the Viterbo restriction functor \cite{AS} for the inclusion $T^*\cL  \subset W$. 
For disk surgery on a torus, 
Equation \ref{eq:clustertransintro} corresponds to the  wall-crossing transformation computed  in \cite{Aur}, 
and expressed explicitly as a cluster transformation in \cite[Prop. 11.8]{Sei2}.

Accepting this conjectural package, it follows that  $\cL$ and $\cL'$ cannot be Hamiltonian isotopic: otherwise the images of $Loc_1(\cL)$ and $Loc_1(\cL')$ in the Fukaya category would necessarily coincide.
Moreover, since we have established that skeletal surgeries induce cluster transformations, 
we can employ cluster algebra to compute --- and, in particular, distinguish --- 
the algebraic tori $Loc_1(\cL)$ and $Loc_1(\cL'')$, even when $\cL$ and $\cL''$ 
are related by a longer sequence of surgeries.  The (conjectural) Hamiltonian isotopy invariance of 
the cluster chart associated to a Lagrangian implies that solving the algebraic/combinatorial problem
of distinguishing cluster charts in fact solves the symplecto-geometric problem of distinguishing the Lagrangians.
A general cluster variety has infinitely many distinct cluster charts \cite{FZ2}, so  the
cluster chart associated to a Lagrangian is a strong enough invariant to distinguish infinitely many 
Lagrangians. 

However, for a given $\L$ it may not be possible to lift an arbitrary sequence of quiver mutations to a corresponding sequence of skeletal surgeries.
 This is because 
of the following subtlety: while the surgery $\L \rightsquigarrow \L'$ always results in a skeleton
which can be built from a surface by attaching handles along Legendrian lifts of curves, the surgery can create self-intersections in these curves. Our surgery does not apply to 
disks attached along curves with self-intersections, so this results in an obstruction to subsequent mutations.  
We will show that when $L$ is a torus and 
the curve collection is geodesic, this issue never arises and arbitrary mutations can be performed; 
this is related to the constructions of \cite{Sym, Via}.

\vspace{4mm}
In another direction, the present construction extends the reach of the dictionary established in
\cite{STWZ}.  There, we gave a symplectic interpretation of the relation between cluster algebras
and bicolored graphs on surfaces \cite{Pos, FG, GK}.  
This went as follows: a surface $\Sigma$ and 
a Legendrian knot $\Lambda$ in the contact boundary of $T^*\Sigma$ 
together determine a certain moduli space; an exact Lagrangian
filling of $\Lambda$ determines a toric chart; and a bicolored graph $\Gamma$ on $\Sigma$ 
determines a Legendrian knot 
$\Lambda$ together with a canonical filling $\cL$. 
It is well known that, for the resulting cluster structures, not all cluster charts 
can be realized by bicolored graphs.  In particular, the vertices of the quiver associated to a bicolored graph are 
named by the faces of the graph; one can perform an abstract quiver mutation at any of these, but the new cluster comes from another bicolored graph only when the face is a square. 

In \cite{STWZ}, we raised the possibility 
that the remaining charts come from Lagrangian fillings which do not arise from bicolored graphs. 
The present technology allows us to construct such fillings. 
The first step is to change perspective from that of the surface $\Sigma$ containing the bicolored graph, to that of the associated Lagrangian filling $\cL$.
We recover $T^*\Sigma$ as the result of attaching handles to $T^* \cL$; the associated Lagrangian disks are exactly the faces of the bicolored graph. 
The filling $\cL$ has boundary $\Lambda$, but our constructions still make sense in this context.  
When $\cL$ has boundary,  $\muloc(\L)$ should correspond to the ``partially wrapped'' Fukaya category
where  $\partial \cL \subset \partial W$ serves as a Legendrian
stopper.

In particular, in the framework of the present paper we can perform a skeletal surgery on any face, square or not, giving a symplecto-geometric description of the chart resulting from the corresponding mutation.
Though beyond the scope of bicolored graphs, the resulting theory is still combinatorially explicit in the sense of being
completely encoded in configurations of curves on $\cL$. Moreover, by allowing Weinstein
manifolds more general than cotangent bundles, the present framework 
captures cluster structures (i.e., equivalence classes
of quivers) more general than those realized by bicolored graphs.

\vspace{10mm}

\subsection{Main results}

We outline more formally the main definitions and results of the article.  Our basic data-set is a configuration
of curves on a surface: 

\begin{definition} \label{def:configuration}
Let $\cL$ be a topological surface.  A {\em curve configuration} on $\cL$ will mean a set 
of properly immersed, co-oriented, pairwise transverse curves on $\cL$. Here a co-orientation of a curve is a choice of one of the two orientations of its conormal bundle.  
If $\cL$ has boundary, then we allow the curves to end on the boundary of $\cL$. 
\end{definition}

This data encodes a Lagrangian skeleton of a Weinstein 4-manifold.  The underlying topological
space of the skeleton is the following: 

\begin{definition} \label{def:L}
Let $\cC$ be a curve configuration on $\cL$.  We write $\L$ for the topological space formed by gluing 
one disk to $\cL$ along each curve in $\cC$. In the case that $\cL$ has boundary along which
a curve $C$ ends, we glue a half-disk to that $C$.
\end{definition}

The 4-manifold is formed by attaching Weinstein handles \cite{W} -- note that a co-orientation of an immersed curve is the same data as a lift to a Legendrian in the contact boundary $T^\infty \cL$ of $T^*\cL$.

\begin{definition}\label{def:W}
For a curve configuration $\cC$ on a surface $\cL$, we write $W := W_\mathcal{C}$ for the Weinstein
4-manifold formed by attaching Weinstein handles to $T^* \cL$ along the Legendrian lifts of closed curves in $\cC$
to their co-orientations. 
\end{definition}

We write $\TC \subset T^*\cL$ for the zero section together with the cones over the Legendrian lifts of the $C_i$, and realize $\L$ inside $W$ as the union of $\TC$ with the cores of the handles. It is a Lagrangian skeleton of $W$ in the sense that it is 
the complement in $W$ of the locus escaping to infinity under a natural Liouville flow; in particular, 
it is a retract of $W$.
In case $\cL$ has boundary, then $\partial \L$ has boundary coming from the union
 of the boundary of $\cL$ with the boundaries of the half-disks attached along 
any curves which end along $\partial \cL$.  This is naturally viewed as a singular Legendrian
in $\partial W$.  See \cite{N2} for more discussion of skeleta in this context.

We refer to $\L$ as a seed skeleton, as it is a singular Lagrangian incarnation of a seed of a cluster algebra. 
We follow Fock and Goncharov, for whom this denotes a collection of elements 
in a lattice equipped with a skew-symmetric form \cite{FG}. 
Often one encodes this data as a quiver without oriented 2-cycles, whose vertices are the given lattice elements and whose arrows record the pairings between them. 
The configuration $\cC$ naturally gives rise to a seed: 
the lattice is $H_1(\cL;\Z)$ with its intersection pairing, and the elements are the classes of the $C_i$. 
Note that this seed and the Weinstein manifold $W$ depend only on the Legendrian isotopy classes of the lifts of the $C_i$.
With this in mind we have:

\begin{definition}
An isotopy of a curve configuration is an isotopy of its constituent curves that arises from a Legendrian isotopy of their lifts. Concretely this means whenever two curves become tangent during the isotopy, their co-orientations at any point of tangency are opposite. 
\end{definition}

\begin{figure}
\begin{center}
\includegraphics{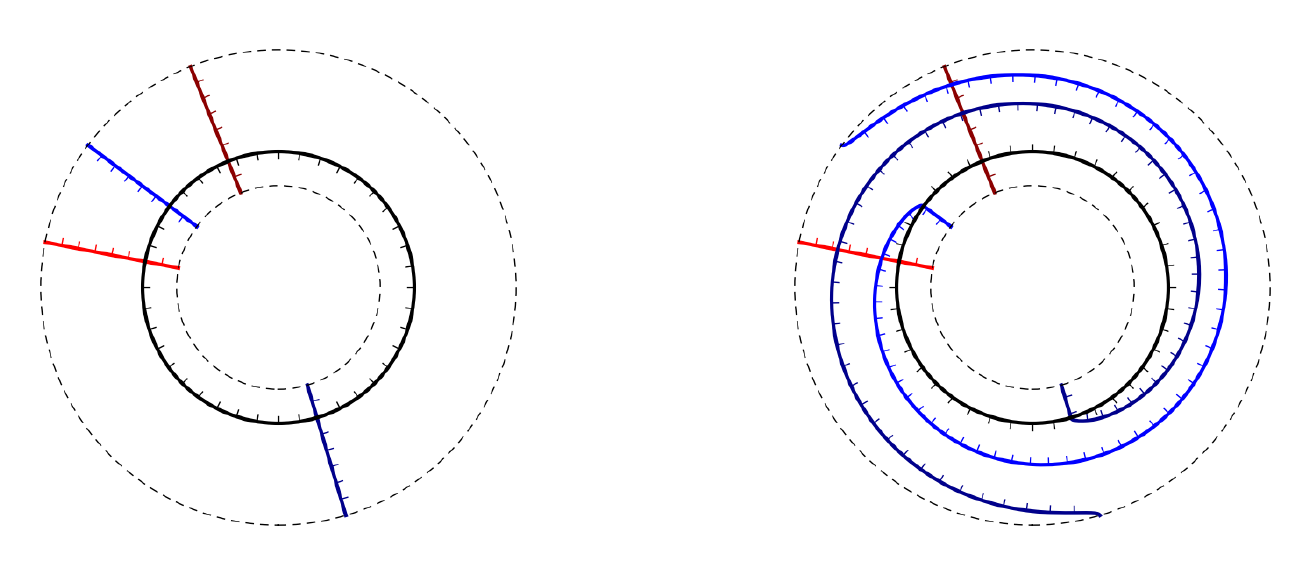}
\end{center}
\caption{Local pictures before (left) and after (right) mutation at an embedded curve (here, the black circular curve). We represent the co-orientations of the curves as hairs pointing to one side. By convention we say, for example, that on the left the blue curve intersects the black circle positively. Mutation twists the curves around the black circle wherever they intersect it positively, and leaves them alone wherever they intersect it negatively.}\label{fig:mutateintro}
\end{figure}

In cluster algebra, there is fundamental operation on seeds called {\em mutation}, determined by the choice
of one of the lattice elements determining the seed, or equivalently, a vertex of the quiver. 
We lift this to an operation of mutation at any simple closed curve $C_k$ in $\cC$; that is, the image of an embedding $S^1 \to \cL$. 
The result is a new curve configuration $\mu_k(\cC) = \{C_i' \}$ obtained from $\cC$ by twisting the other curves around $C_k$ according to the orientations of their intersections with $C_k$; see Definition \ref{def:mutation} and Figure \ref{fig:mutateintro}. 
We write $\cC'$ for $\mu_k(\cC)$ when $k$ is understood, and for clarity denote the surface on which the new configuration sits as $\cL'$.
Writing $\L' := \mu_k(\L)$, and $W' := \mu_k(W)$  for the 
seed skeleton and Weinstein manifold associated to $\cC'$ as above, we show: 

\begin{theorem} \label{thm:symplectomorphismintro} 
There is a symplectomorphism $W \cong W'$  
such that the preimage of $\cL' \subset W'$ 
is related to $\cL$
by Lagrangian disk surgery along the disk $D_k$ attached to $C_k$. 
\end{theorem}

In Section \ref{sec:KSseed} we define a sheaf $\muloc$ of dg categories on $\L$.  It is glued together from 
sheaves of categories on conical models of local pieces of $\L$; these local 
sheaves of categories are themselves microlocalizations of constructible sheaf categories 
as in \cite[Chap. 6]{KS}, \cite{N2, N3, N4}, \cite{Gui}. 
Following the discussion above we expect that the global section category $\muloc(\L)$ 
captures some appropriate version of  the Fukaya category of $W$.  Given such a result, we could 
deduce from
Theorem \ref{thm:symplectomorphism} that $\muloc(\L) \cong Fuk(W) \cong Fuk(W') \cong \muloc(\L')$.

Instead, we construct such a composition directly using sheaf-theoretic results of \cite{GKS}.  We refer to the resulting equivalence as a mutation functor.

\begin{theorem}\label{thm:invariance}
If $C_k$ is a simple closed curve in the curve collection $\cC$, 
there is an equivalence
$\Mut_k: \muloc (\L) \cong \muloc (\L')$.  
\end{theorem}
\begin{remark}
The equivalence is induced by a local construction in a neighborhood of the surgery, which does not depend
on the remaining geometry of the skeleton.  Implicit in the above statement, and explicit in the proof of
the theorem, is the freedom to pass between different conical models used to locally describe $\muloc$. 
\end{remark}

Let $\loc(\cL)$ be the category of local systems on $\cL$.  We show: 

\begin{lemma}
There is a fully faithful
inclusion $\loc(\cL) \hookrightarrow \muloc(\L)$.  
\end{lemma}

The relation between $\muloc(\L)$ and cluster algebra
comes from the comparison
$$\loc(\cL) \subset \muloc(\L) \cong \muloc (\L') \supset \loc(\cL')$$
Let $\Loc_1(\cL) \subset \loc(\cL)$ denote the full subcategory of local systems whose stalks are free of rank one and concentrated in cohomological degree zero, similarly for $Loc_1(\cL')$.  The objects of $\Loc_1(\cL)$ are determined by their holonomies, hence are parametrized by an algebraic torus.  

Since $\cL'$ comes with a homeomorphism to $\cL$, it makes sense to ask how the holonomies of a rank one local system transform under surgery.

\begin{theorem} \label{thm:cluster}
The image of an object of $\Loc_1(\cL)$ under $\Mut_k$ is an object of $\Loc_1(\cL')$ if and only if their holonomies differ by the (signed) cluster $\cX$-transformation at $[C_k]$.
\end{theorem}

\begin{remark}
We reduce the general proof of Theorem \ref{thm:cluster} to a special case proved in \cite{BezKap}; see Section \ref{sec:BezKap} for a discussion. 
The signs indicated depend on a choice made in defining $\muloc$, which we 
have omitted from the notation.  The choice is 
classified by an element of $H^2(\L, \bZ/2\bZ)$ which vanishes upon restriction to $\cL$. 
Depending on the choice made one obtains a mix of the usual positive formulae or similar ones with minus signs.  
Typically in the cluster literature one assumes the lattice elements in a seed are linearly 
independent; in this case the sign choice is cosmetic.  
We expect that this choice can be identified with the corresponding one made in defining $\Fuk(W)$ (for
which see e.g. \cite[Sec. 12]{Sei-pic} or \cite{FOOO}). 
\end{remark}

We let $\cM_n(\L)$ denote the closure in the moduli space of $\muloc(\L)$ of the locus of 
objects whose stalks on $\cL$ have cohomology of rank $n$ concentrated in degree zero. 
The precise meaning of moduli space here is clarified in Section \ref{sec:moduli}.  
Translated to a statement about spaces, Theorem \ref{thm:cluster} becomes:

\begin{theorem}
The rank one moduli space $\cM_1(\L)$ has a partial cluster $\cX$-structure with initial seed $(H_1(\cL; \Z), \{[C_i]\})$.   
\end{theorem}

\begin{remark}
For $n > 1$, the moduli space $\cM_n(\L)$ carries likewise a ``nonabelian cluster structure'' whose charts are spaces of rank $n$ local systems on $\cL$.  Their transition functions are determined by the same computation as in Theorem \ref{thm:cluster}, see Section \ref{sec:ncclusters} and Section \ref{sec:cluster} for discussion. 
\end{remark}

\begin{figure}
\centering
\includegraphics{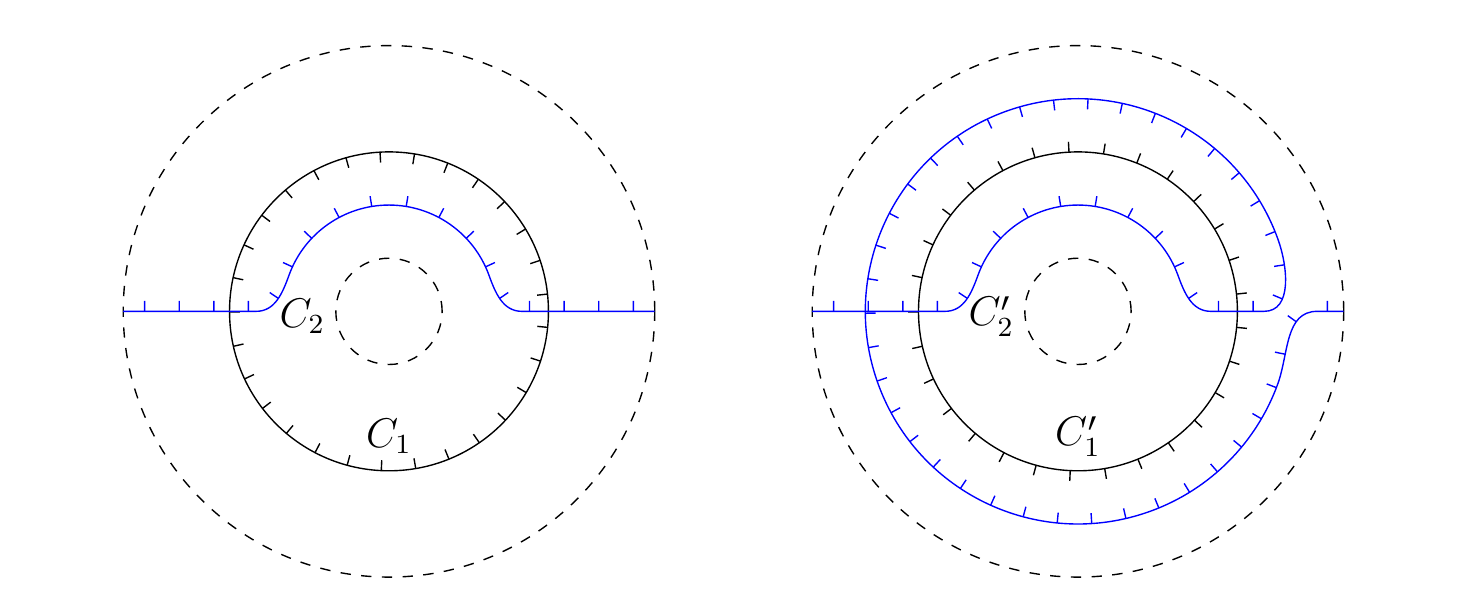}
\caption{On the left are two curves with intersections of opposite signs. When we mutate at $C_2$, 
the curve $C_1$ is twisted into the self-intersecting curve $C'_1$.}\label{fig:selfintintro} 
\end{figure}

The notion of cluster structure we must consider is more general than that usually encountered in the literature. We have seen that choices made in defining $\muloc$ can lead to the appearance of certain signs. But even making choices to avoid such signs, the precise notion of cluster $\cX$-structure introduced in \cite{FG2} applies only when the classes of the $C_i$ are a basis for $H_1(\cL,\Z)$. In general, each toric chart on $\cM_1(\L)$ has a map to a corresponding chart on the usual $\cX$-variety, dual to the natural homomorphism  of character lattices (see Section \ref{sec:cluster} for discussion).

A more fundamental issue is that while a cluster structure on a space consists of a system of toric charts related by all possible sequences of mutations, in the present context we are forced to also consider partial cluster structures; that is, structures involving only the tori obtained from a subset of possible mutation sequences. 
This is because we only have a sensible notion of mutation at a simple closed curve, while a mutation at one curve may create a self-intersection in another. 
This  occurs exactly when the two curves have intersections of opposite signs as in Figure \ref{fig:selfintintro}.  
Since this condition is \emph{not} invariant under (Legendrian) isotopies of the curve configuration, it is possible that by choosing suitable isotopies between mutations one can realize arbitrary sequences of surgeries:

\begin{definition}
A configuration $\cC$ of embedded co-oriented curves is nondegenerate if it admits arbitrary sequences of mutations. That is, for every list  $(i_1,\dotsc,i_n)$ of curve indices, there is a sequence $(\iota_1,\dotsc,\iota_{n-1})$, where $\iota_k$ is an isotopy from $\mu_{i_k} \circ \iota_{k-1} \circ \mu_{i_{k-1}} \cdots \iota_1 \circ \mu_{i_1}(\cC)$ to a configuration such that $\mu_{i_{k+1}} \circ \iota_k \circ \mu_{i_k} \cdots \mu_{i_1}$ consists entirely of simple closed curves.
\end{definition}

It is easy to find curve configurations which are clearly degenerate, and in general we do not know how to determine whether a given configuration is nondegenerate. However, we have one important case when the existence of a complete cluster structure is guaranteed:

\begin{theorem}\label{thm:nondegintro}
If the surface $\cL$ is a torus and all curves in $\cC$ are geodesics in a flat metric, then $\cC$ is nondegenerate. In particular, $\cM_1(\L)$ carries a complete cluster $\cX$-structure. 
\end{theorem}

\begin{figure}
\centering
\includegraphics{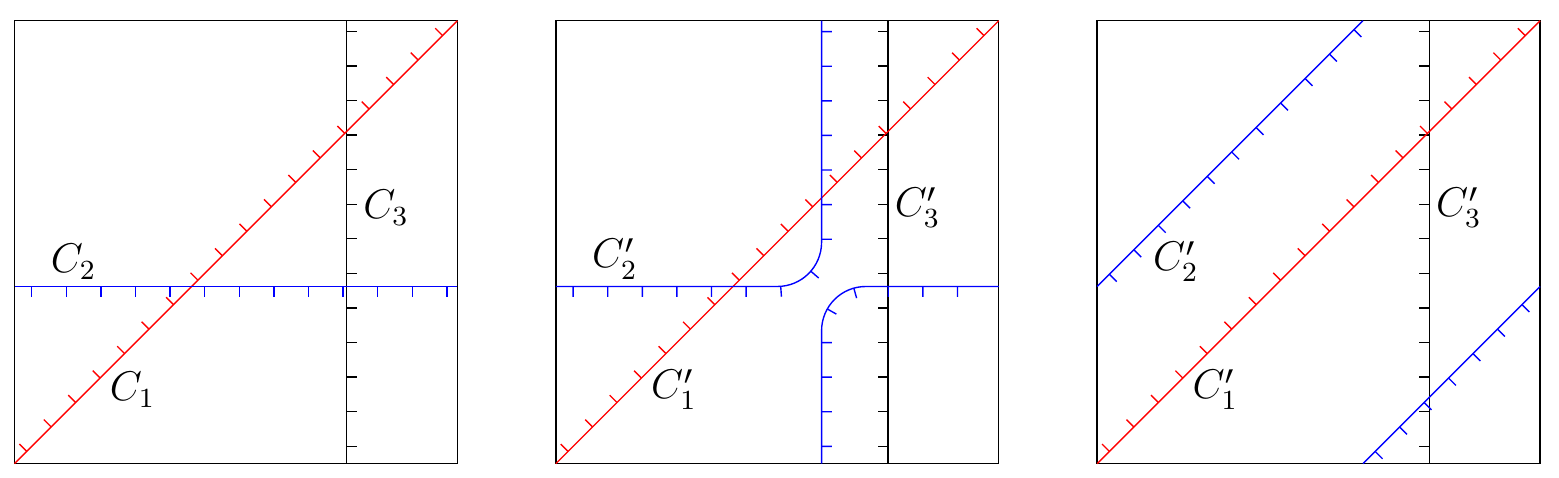}
\caption{The left frame shows a configuration of three geodesics on $T^2$. The middle frame shows the result of mutating at $C_3$. This twists $C_2$ into a new curve $C'_2$, which has two intersections with $C'_1 = C_1$ of opposite orientations; mutating at $C'_2$ would thus create a self-intersection in $C'_1$. But there is an isotopy of the Legendrian lifts which pulls $C'_2$ straight, allowing mutation at $C'_2$ without obstructing further mutation.}\label{fig:straightening}
\end{figure}

Such examples arise from almost toric fibrations, and include 
the complements of anti-canonical divisors in complex toric surfaces; 
we review their theory in Section \ref{sec:almosttoric}.

\newpage

\subsection{Further directions}

\subsubsection{Degeneracy and potentials} \label{sec:potential}
The theory of curve configurations described here appears to be related to the theory of quivers with potential \cite{DWZ}. A potential on a quiver is the extra data of a formal sum of oriented cycles. It encodes an algebra, the Jacobian algebra, which is the quotient of the path algebra by the cyclic derivates of the potential. One is generally only interested in potentials up to automorphisms of the path algebra, as the isomorphism class of the Jacobian algebra is invariant under such transformations.

One can mutate a quiver with potential at a vertex not meeting oriented 2-cycles, but this may create 2-cycles elsewhere in the quiver. The potential dictates when to erase these new 2-cycles and mutate further:  exactly when it can be done without changing the Jacobian algebra.

In our case, to a configuration $\cC$ of simple closed curves we associate a quiver $Q_\cC$. Its vertices are the curves in $\cC$ and its arrows are their intersections. The orientations of the arrows correspond to the orientations of the intersections; if we erase all 2-cycles we get the 2-acyclic quiver obtained from the the seed associated to $\cC$. 

The condition that mutation at $C_k$ does not produce a self-intersection in $C_j$ is precisely the condition that $C_k$ and $C_j$ are not the vertices of an oriented 2-cycle in $Q_\cC$. The notion of equivalence
we have defined allows us to erase all such 2-cycles exactly when there is an isotopy of $\cC$ that cancels out intersections of $C_k$ and $C_j$ until those remaining are of the same orientation. Thus the data of the curve configuration is formally analogous to that of a potential on $Q_\cC$, insofar as it controls the erasing of 2-cycles needed to iterate the mutation process.

With this in mind, we have borrowed the terminology of nondegenerate configurations from the analogous notion for quivers with potential. 
We expect this relation can be made precise: 

\begin{problem}
For any configuration $\cC$ construct a potential on $Q_\cC$ combinatorially --- e.g. by counting
some polygons in $\cL$ with edges on the curves and corners at their intersections.  
Show that Legendrian isotopies of curve configurations induce equivalences
of potentials, and that mutations of curve configurations induce mutations of potentials (up to equivalence). 
\end{problem}

Following \cite{Smi}, such a potential should have a geometric origin. To the 4-manifold $W$ one should associate a Calabi-Yau 6-manifold fibered over $W$, with the disks in $\L$ being the images of an associated a collection of Lagrangian 3-spheres. The quiver $Q_\cC$ then records the intersections of these 3-spheres, and the potential should record Floer-theoretic relations among them:

\begin{problem}
For any configuration $\cC$ construct a potential on $Q_\cC$ geometrically,
using the Fukaya category of a Calabi-Yau 6-manifold which fibers over $W$. 
\end{problem}

A solution to either of the above problems should help address the more general question of how curve configurations can fail to be nondegenerate:

\begin{problem}
Given a curve configuration $\cC$, compute which sequences of mutations can be performed on it without creating self-intersections. Generalize Theorem \ref{thm:nondegintro} by describing other explicit classes on nondegenerate configurations. For example, show that $\cC$ is nondegenerate when $Q_\cC$ has no directed cycles.
\end{problem}

\subsubsection{Cluster algebras from bicolored graphs}  
There is a large existing literature studying cluster structures coming from graphs on 
 surfaces, or equivalent geometric-combinatorial data. 
Examples of such cluster structures include those on the positroid strata of the Grassmannian \cite{Pos} and (tame and wild) character varieties of punctured surfaces \cite{GSV, FG, FST}. 
The general pattern is to begin with a bicolored graph $\Gamma$ embedded in a surface $\Sigma$ such that its complement has contractible connected components. 
The quiver is the dual graph to $\Gamma$; the associated 
cluster torus can be identified with the space of rank one local systems on $\Gamma$; 
its map to the relevant space  
is defined combinatorially as a sum over flows or perfect matchings on $\Gamma$. 

\begin{warning*}
The surface $\Sigma$ {\bf does not} play the same role as our surface $\cL$.
\end{warning*}

There is however a precise relation between the two.  It factors through the 
symplectic interpretation of the combinatorics of bicolored surface graphs, given in our previous work \cite{STWZ}.
In that account, the first step is to replace the bicolored graph $\Gamma$ 
by the equivalent data of an 
alternating strand diagram, which is then lifted to a 
Legendrian knot $\Lambda$ in the contact boundary of $T^*\Sigma$. 
We then studied the category 
$Sh_\Lambda(\Sigma)$ of constructible sheaves 
 with microsupport in this knot.  Of particular relevance are
the rank one objects, called {\em simple sheaves}
in \cite{KS}, studied
in \cite{STZ}, and identified with objects of the augmentation category of Legendrian contact homology
in \cite{NRSSZ}; we denote the moduli space of such sheaves by $\cM_1(\Lambda)$.

Following \cite{GK}, the bicolored graph defines (and is a deformation retract of) a so-called
{\em conjugate surface} $\cL$.   In \cite{STWZ}, we showed that $\cL$ can be embedded inside $T^* \Sigma$ as
 an exact Lagrangian filling of $\Lambda$.  We can choose this embedding so  its intersection with the zero section is  the graph $\Gamma$.  

The Nadler-Zaslow correspondence \cite{NZ, N1} then determines a map 
 $$Loc_1(\Gamma) = Loc_1(\cL) \hookrightarrow \cM_1(\Lambda),$$
which we showed is a chart in a cluster structure whose type is that of the dual quiver to $\Gamma$. Moreover, for certain choices of graph, $\Lambda$ can be isotoped so that $\cM_1(\Lambda)$ is manifestly isomorphic to a character variety or other space of interest.

The present paper tells the story from the perspective of $\cL$
rather than from the perspective of $\Sigma$.  The dictionary is as follows.  
We are interested in the Weinstein manifold $W = T^* \Sigma$, or more precisely in the Weinstein pair $(T^* \Sigma, \Lambda)$. 
The approach of \cite{STWZ} is essentially to study it in terms of the Lagrangian skeleton $\Sigma \cup \R_+ \Lambda$. 
In the present account, we are interested instead in the skeleton $\Sigma \cup \cL$,
or more precisely, in a slight perturbation of it in which the disks $\Sigma \setminus \Gamma$
are attached to $\cL$ such that their boundaries are pairwise transverse as in Definition \ref{def:W}.
The resulting skeleton is our $\L$.  To make  precise this relation, 
one should do the following: 

\begin{problem}
Beginning with a bicolored graph $\Gamma \subset \Sigma$ and applying the construction of
\cite{STWZ} to obtain a Legendrian $\Lambda$ with Lagrangian filling $\cL$ such that $\Gamma = \cL \cap \Sigma$, specify 
a perturbation of the face cycles of $\Gamma$ inside $\cL$, and show: 
\begin{itemize}
\item Applying Definition \ref{def:W} to the resulting curve
collection results in a pair $(W, \partial \cL)$ symplectomorphic to $(T^* \Sigma, \Lambda)$.
\item Rescaling the fibers of $T^*\Sigma$ gives a noncharacteristic
(in the sense of \cite{N4}) deformation $\L \rightsquigarrow \Sigma \cup \R_+ \Lambda$, hence inducing
$\muloc(\L) \cong Sh_\Lambda(\Sigma)$.  
\end{itemize}
\end{problem}

Since the vertices of the quiver correspond to the faces of $\Gamma$, we can perform a mutation at any face.  However, the new quiver again arises from a bicolored graph only when the face
is a square; 
the resulting transformation of bicolored graphs is called the ``square move''. 
In terms of the Legendrian-and-Lagrangian picture of \cite{STWZ}, this corresponds to a Legendrian isotopy
$\Lambda \rightsquigarrow \Lambda'$ together with a family of exact Lagrangian fillings 
$\cL \rightsquigarrow \cL'$ which passes through a singular Lagrangian.  In other words, $\cL$ and $\cL'$ are related
by a Lagrangian disk surgery at the square face. 

In the formalism of the present article, however, nothing is privileged about square faces. We can perform skeletal surgery on $\L$ at the disk coming from any face of $\Gamma$, yielding an exact Lagrangian filling giving rise to the relevant cluster chart. While this gives a geometric interpretation of mutation at an arbitrary face of $\Gamma$, we do not know to what extent further mutations are possible. 

\begin{problem}
Assume that the surface $\cL$ arises as the filling described by a bicolored graph,
and the curve configuration $\cC$ is determined by the face cycles, as described above.  
Determine whether $\cC$ is nondegenerate, or if not, characterize its mutation graph. 
\end{problem}

We note that in \cite{STWZ}, it is shown character varieties, positroid varieties, etc., are {\em biregularly}
isomorphic to $\cM_1(\Lambda)$, for appropriate $\Lambda$ --- this stands in contrast to most all accounts
of cluster algebra, in which one works only birationally.  In particular, we note that nondegeneracy of certain
explicit curve collections would imply that the above mentioned varieties carry complete, rather than partial, 
cluster atlases.  It is, however, also possible that the varieties carry complete cluster atlases, but not all charts
are realized by the geometric construction we have described.

\subsubsection{Infinitely many exact Lagrangians}
The ideas of this paper suggest a method to produce and distinguish infinitely many exact Lagrangians
in certain Weinstein 4-manifolds:
\begin{enumerate}
\item The invariant 
$\cL \mapsto \mathrm{Image}(Loc_1(\cL) \subset \muloc(\L) = Fuk(W))$ should give a Hamiltonian isotopy 
invariant of the Lagrangian $\cL$.   
\item Sequences of mutations of $\L$ induce transformations of this invariant which are given explicitly by cluster 
transformations on the torus $Loc_1(\cL)$, governed by the cluster structure coming from the quiver $Q_\L$.  
\item Because there are very few quivers which give rise to cluster structures with finitely many clusters ---
these have an ADE classification
\cite{FG2} ---
it should generally be the case that the above procedure gives infinitely many Hamiltonian non-isotopic exact
Lagrangians. 
\end{enumerate}

There are some steps which remain in carrying out the above programme.  

\vspace{2mm}Regarding the first point, 
either Kontsevich's conjecture that 
$\muloc(\L) = Fuk(W)$ has to be proven, or one has to show by another route that the Lagrangians which we construct
here are related in the Fukaya category by cluster transformations.  In fact, both can be done: the first
will appear in \cite{PS} and the second in \cite{ESS} (using and extending techniques from \cite{BEE}).

Regarding the second and third points, there are two subtleties.  The first is that 
we must understand when quiver mutations give rise to curve configuration mutations which preserve
the property that the curves are non-self-intersecting.  Ideally this will happen via a theory of 
potentials as suggested in Section \ref{sec:potential}; but other combinatorial approaches, especially in
examples, are possible as well.  For instance, we have shown directly that {\em all} mutations are allowed
for geodesic curve configurations on $\T^2$.  This setting includes that of recent works such as \cite{Kea, Via-one,
Via, Pas}. 

The second is that to completely understand which Lagrangians are distinguished this way, one would need a complete understanding of the exchange graph of the relevant cluster algebra. That is, though cluster tori are a priori labeled by the vertices of an infinite tree, the composition of cluster transformations associated to the path between two vertices can be biregular. The tori at these vertices thus yield the same open subset of the cluster variety; the exchange graph is the quotient of the infinite tree by this equivalence relation. 
For example, if two vertices in the initial quiver are connected by a single arrow, there is an associated 5-cycle in the exchange graph. To our knowledge it is a (well-known) open question of whether all cycles are generated by those arising from single arrows in this way

Moreover, we have seen that $\cM_1(\L)$ is in general not quite the same (even birationally) as the usual cluster $\cX$-variety. For example, when the classes of the curves in $\cC$ generate $H_1(\cL,\Z)$ but are not linearly independent, $\cM_1(\L)$ is instead birational to a positive-codimension subspace of the cluster $\cX$-variety. 
To distinguish Lagrangians, one must understand the classification of distinct cluster charts on such spaces, rather than charts on the usual cluster $\cA$- or $\cX$-varieties \cite{FG2}:

\begin{problem}
Let $s = (N,\{e_i\})$ be a seed as defined in Section \ref{sec:cluster}, where the $e_i$ are not necessarily a basis (for example, $(H_1(\cL,\Z),\{[C_i]\})$). Relate the classification of cluster tori in the associated $\cX$-variety (for example, $\cM_1(\L)$, if $\cC$ is nondegenerate), to that of clusters in the associated ordinary cluster algebra --- that is, the cluster algebra associated to the seed $(\Z^n = \Z\{e_i\},\{e_i\})$ with skew-symmetric form pulled back from $N$.
\end{problem}

It is also natural to conjecture that the symplectic 4-manifolds which correspond
to the finite-type cluster algebras have only finitely many Hamiltonian isotopy classes of 
Lagrangians which are topologically isotopic to $\Sigma$.  More generally, it would be interesting
to understand the symplectic meaning of the cluster algebras of finite mutation type \cite{FeST}, as well as general growth rates of cluster algebras. In the other direction, an invariant for distinguishing Lagrangians related by disk surgery was introduce in \cite{Y}; it might be interesting to understand this from the perspective of cluster algebras. 
A final natural question is:

\begin{problem}
Give a more intrinsic characterization of the Weinstein 4-manifolds which can be obtained from the construction of Definition \ref{def:W}.
\end{problem}

\subsubsection{Noncommutative Cluster Theory}\label{sec:ncclusters}
While the rank one moduli space $\cM_1(\L)$ carries a cluster $\cX$-structure in the usual sense (or part of one, if $\cC$ is degenerate), we have observed that for $n > 1$ the space $\cM_n(\L)$ carries a nonabelian generalization of one. Its charts are spaces of rank $n$ local systems on $\cL$. On the other hand, cluster-theoretic structures involving noncommutative variables have recently been studied by a number of authors \cite{K-nc,BRet,dFK}. We expect the spaces $\cM_n(\L)$ are geometric incarnations of this developing theory:

\begin{problem}
Make precise the relationship between the space $\cM_n(\L)$ and the noncommutative cluster theory developed in e.g. \cite{BRet}.
\end{problem}

\subsubsection{Relation to SYZ Mirror Symmetry}
Several works have explored the relationship between 
SYZ mirror symmetry --- i.e., the study of mirror symmetry via Lagrangian torus fibrations --- and cluster algebra, 
for example \cite{Aur, Sei2, AKO,GHK-mir,Pas,Via}, and most notably 
\cite{GHKK} which used ideas from mirror symmetry to resolve the longstanding positivity conjecture of \cite{FZ}.  

By comparison, in the present work we view a cluster variety $X$ as a space of Lagrangian surfaces, which may have positive genus, in a symplectic 4-manifold, which may be of lower dimension than $X$.  
We view $\cM_1(\L)$ as a type of mirror to $W$, albeit not an SYZ mirror.

Thus ideas from mirror symmetry can be applied in the present context.  
In particular, the zeroeth Hochschild 
homology of $\muloc(\L)$ --- expected to be equivalent to a certain piece of the symplectic homology of $W$ 
 \cite{K-homms, Sei-def, Sei-hoc} --- 
gives rise to functions on the moduli space of objects in $\muloc(\L)$, hence to functions on $\cM_1(\L)$.  

\begin{problem}
Give an explicit expression for the Hochschild homology of $\muloc(\L)$ in terms of the skeleton $\L$, or 
equivalently, the curve configuration $\cC$. 
\end{problem}
\begin{remark}
Some difficulties with carrying this out are described in \cite{Dyc}. 
\end{remark}

\begin{problem}
Give an explicit expression for the symplectic homology of $W$ in terms of the curve configuration $\cC$. 
\end{problem}
\begin{remark}
One approach would be to develop an analogue for cotangent bundles of Ng's resolution procedure \cite{Ng} 
so that the Legendrian DGA can
be computed by counting polygons on $\cL$ bounded by the curve collection $\cC$; recall that \cite{BEE} 
explains how the Legendrian DGA determines symplectic homology of $W$.  Note that we also expect
such counting to be implicated in the construction of the potential as discussed in Section \ref{sec:potential}. 
\end{remark}

\subsubsection{Symplectic and Poisson structures;  and quantization of cluster varieties} 
One appealing feature of the fact that we realize cluster varieties as moduli of objects in the Fukaya categories
of {\em 4-manifolds} is that this provides a modular explanation for the existence of 
symplectic or Poisson structures on cluster varieties, as in \cite{GSV,FG2}:  such structures are to be expected of moduli spaces of objects in 2-Calabi-Yau categories \cite{PTVV}.  

Moreover, since we are studying moduli of a sheaf of categories over $\L$, these interpretations 
should localize over $\L$. 
In fact we will construct in \cite{ST} such symplectic or Poisson structures, 
as a special case of a general construction of such 
structures on moduli spaces of microlocal sheaves.  Along the lines of \cite{CPTVV},
this may give rise to deformation quantizations.  It is natural to ask whether these can be integrated 
via the methods of  \cite{BBJ} and \cite{AFT} and related to the
quantizations of \cite{FG2, FG3, FG4}.

\subsubsection{Microlocal Sheaves on Nodal Curves}\label{sec:BezKap}
When $\cC$ consists of pairwise nonintersecting simple closed curves, 
the seed skeleton $\L$ is the arborealization of a nodal surface \cite{N3,N4}.  The homotopy category of $\muloc(\L)$ contains the category of ``microlocal perverse sheaves'' associated to the nodal surface by Bezrukavnikov and Kapranov \cite{BezKap}. 
They used the conical model of the nodal singularity given by the union of $\R^2$ and the cotangent fiber over the origin.  
Their moduli coincide with our $\cM_1(\L)$ and its higher-rank, framed versions, recovering the multiplicative quiver varieties of \cite{CBS,Yam}. 

The avatar of disk surgery from this perspective is the action of the Fourier-Sato transform on perverse sheaves. That its action on local systems is a cluster transformation is computed in \cite{BezKap}. 

\vspace{-2mm}
\subsection*{Acknowledgements}  
We thank Mohammad Abouzaid, Nate Bottman, Tom Bridgeland, 
Sean Keel, Maxim Kontsevich, Emmy Murphy, David Nadler, Theo Johnson-Freyd, Denis Auroux,
Nick Sheridan, Nicol\'o Sibilla, and Ivan Smith for helpful conversations. 
Section 6 owes its existence to discussions with Jonny Evans, Ailsa Keating, 
James Pascaleff, Laura Starkston, Dmitry Tonkonog, and Renato Vianna.  
We especially thank Paul Seidel and Eric Zaslow. 
V. S. is partially supported by the NSF grant DMS-1406871 and a Sloan fellowship, and 
has enjoyed the hospitality of the Mittag-Leffler institute during the preparation of this article. H.W. is supported by an NSF Postdoctoral Research Fellowship DMS-1502845.  

\vspace{-2mm}
\subsection*{Conventions and abuses}

We refer use dg categories throughout and refer to \cite{Kel,Toe3} for generalities. Used in the context of dg categories, equivalence always means quasi-equivalence, fully faithful means quasi-fully faithful, etc. We may say category rather than dg category when this is clear from the context. We will work with homotopy sheaves of dg categories, which should be understood in the model structure whose weak equivalences are the quasi-equivalences \cite{Tab2}; a sheaf of dg categories always means a homotopy sheaf in this sense. 

Throughout $\coeffs$ denotes a commutative coefficient ring, generally omitted from the notation.
Given a quiver $Q$, we write $\coeffs Q\dmod$ for the dg derived category of $Q$-representations. 
 Given a manifold $M$, we let $\sh(M)$ denote the dg derived category of constructible sheaves of $\coeffs_M$-modules on $M$. That is, it is the quotient \cite{Kel3, Dri} of the dg category of complexes of sheaves with constructible cohomology by the acyclic complexes; we simply refer to an object of $\sh(M)$ as a sheaf. We write $\loc(M) \subset \sh(M)$ for the full subcategory whose cohomology sheaves are local systems. We write $\Loc_n(M) \subset \loc(M)$ for the full subcategory whose cohomology is concentrated in degree zero and whose stalks are of rank $n$. Given a conical Lagrangian $L \subset T^*M$, $\sh_L(M) \subset \sh(M)$ is the full subcategory of sheaves whose singular support is contained in $L$, see \cite{KS}, our use
of this notion in contexts close to the present one in \cite{STZ, STWZ}, or the very brief summary in Appendix \ref{app:sheaves}. 

\section{Geometry of Mutation}\label{sec:geometry}

This section contains the detailed description of our skeletal surgery. 
We proceed on two levels: first in terms of a combinatorial operation on curve configurations, then in terms of a geometric operation on Lagrangian skeleta.  We then sketch an argument for why the ambient symplectic manifold $W$ should be preserved by skeletal surgery.

\subsection{Mutation of curve configurations} \label{sec:halfDehn}

Fix a surface $\cL$ and a collection $\cC$ of co-oriented, immersed curves, as in Definition \ref{def:W}. 
Choose a simple closed curve 
$C_k$ among them.  
We  give here a combinatorial description of a new collection $\cC'$ of curves on $\cL$.  

\begin{convention}\label{con:intersectionsigns}
If $\alpha$ and $\beta$ are co-oriented curves on a surface $\Sigma$ that meet transversely at $x \in \Sigma$, the ordered pair $(T_{\alpha}^+ \Sigma$, $T_{\beta}^+\Sigma)$ of rays in $T_x^* \Sigma$ determines an orientation of $T_x^* \Sigma$. 
If $\Sigma$ is oriented, we write $\langle \alpha,\beta\rangle_+$ and $\langle \alpha,\beta\rangle_-$  for the number of intersections that agree, resp. disagree with the orientation.  
\end{convention}
\begin{remark}
Note that $\langle \beta, \alpha \rangle_+ = \langle \alpha, \beta \rangle_-$, and  that the algebraic intersection number $\langle \alpha, \beta \rangle$ is equal to $\langle \alpha, \beta \rangle_+ - \langle \alpha, \beta \rangle_-$. 
\end{remark}

We write $\cL_{(k)}$ for a neighborhood of $C_k$ in $\cL$, and $\cC_{(k)}$ for the intersection of the curve configuration with this cylinder. 
More precisely, $\cC_{(k)}$ consists of one closed curve $C_k$ 
in the center of an annulus, and a number of pairwise noncrossing curves $B_i$ running from one end of the annulus to the other. 

We may assume each $B_i$ intersects $C_k$ only once.  

\begin{warning*} 
Multiple $B_i$ can come from the same curve $C_j$ in $\cC$: if $C_j$ intersects $C_k$ $n$ times, then $C_j \cap \cL_{(k)}$ will have $n$ components.
\end{warning*}

\begin{definition}\label{def:mutation} 
Let $tw: \cL_{(k)} \to \cL_{(k)}$ be a positive Dehn twist in a very small collar neighborhood of a translation of $C_k$ a short distance in the direction opposite its co-orientation. 
Let $\cC_+ := \{B_i \, |\, \langle B_i, C_k\rangle = 1\}$ be the subcollection of curves which intersect $C_k$ positively, and let $\cC_-  := \{B_i \, |\, \langle B_i, C_k\rangle = -1\}$ be the subcollection which intersect it negatively.  Let $C_k'$ be obtained from $C_k$ by reversing the co-orientation. 

We define the mutated curve collection $\cC_{(k)}'$ on $\cL_{(k)}$ by: 
$$\cC_{(k)}' := tw(\cC_+) \cup C_k' \cup \cC_-$$ 
\end{definition}

The definition is best understood by staring at Figure \ref{fig:mutate2},
which should be interpreted
according to the following convention.

\begin{convention} (Drawing hairs.) 
Let $X$ be a manifold.  Fix a submanifold $V \subset X$.  
Choosing a metric on $X$, we can identify the conormal bundle, normal
bundle, and a tubular neighborhood of $V$.  
Thus we can describe subvarieties of the conormal bundle of $V$ in terms of subvarieties
of the tubular neighborhood. 
In the present case, $X$ is always two dimensional.  Local pictures of $X$ are drawn on
the piece of paper, which gives a choice of local metric.  We indicate a conical
subvariety of the conormal bundle to a manifold 
by drawing ``hairs'' in the tubular neighborhood.  
Generally we will draw the hairs in the same
color as the submanifold. 
Note that when $V$ is codimension one, drawing the hairs on one side or the other
is the same as choosing a co-orientation of $V$. 
\end{convention}

\begin{remark}
Mutating at $C_k$, and then mutating at $C'_k$
results in a curve configuration whose Legendrian lifts are 
Legendrian isotopic (relative the boundary)
to the configuration which would result from applying a 
single Dehn twist.  
\end{remark}

\begin{figure}
\begin{center}
\includegraphics{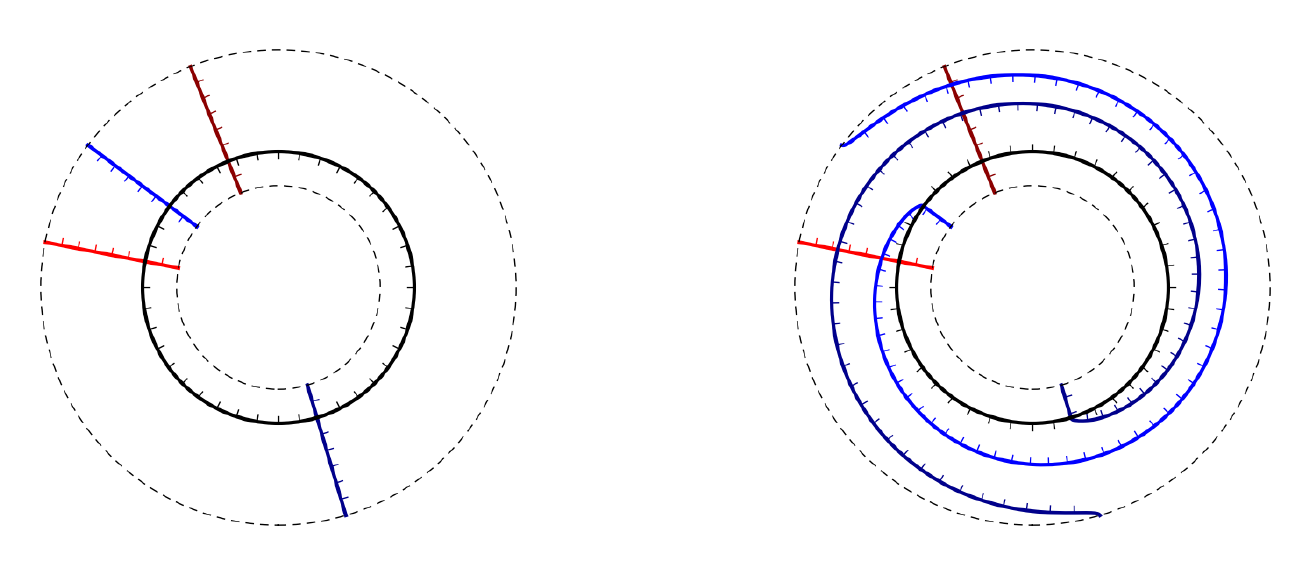}
\end{center}
\caption{\label{fig:mutate2} Before (left) and after (right) pictures of a mutation at a curve met positively by two lines
and negatively by two lines.  Here $\Sigma_{(k)}$, $\Sigma'_{(k)}$ are the left and right annuli whose boundaries are the dashed circles, and the curves $C_k$, $C_k'$ are the black circles. Convention \ref{con:intersectionsigns} says that in the left picture $\langle \text{black}, \text{\color{red} red} \rangle_+ = 1$.}
\end{figure}

\begin{proposition} \label{prop:localcurvemutation}
Let $B_1, B_2$ be segments, and $C_k$ the central curve in the collection
$\cC_{(k)}$; let 
$B_1', B_2'$ and $C_k'$ be their counterparts in 
$\cC'_{(k)}$.  Then 
\begin{eqnarray*}
\langle B'_i, C'_k\rangle_{\pm} & = & \langle B_i, C_k\rangle_{\mp} \\
\langle B_1', B_2' \rangle_+ & = & 
\langle B_1, C_k \rangle_+ \langle C_k, B_2 \rangle_+  
\end{eqnarray*}
\end{proposition}
\begin{proof}
By inspection of Figure \ref{fig:mutate2}.  

Or, in symbols: 
the statement regarding intersections with $C_k$ holds because we are reversing $C_k$
to get $C_k'$. 
For the second statement, there are different cases according as the values of
$\langle B_1, C_k \rangle_+$ and $\langle B_2, C_k \rangle_+$. 

When $B_1, B_2$ have the same orientation,
i.e. $ \langle B_1, C_k \rangle_+ = \langle B_2, C_k\rangle_+ $, we have
$\langle B_1', B_2' \rangle_+ = 0$ and also 
one of $\langle B_1, C_k \rangle_+$,  $\langle C_k, B_2 \rangle_+ $ must be zero,
giving the desired equality. 

In case $\langle B_1, C_k \rangle_+ = 1 = \langle C_k, B_2 \rangle_+ $, we have 
to show that $\langle B_1', B_2' \rangle_+ = 1$. 
In homology $B_1' = B + C_k$ and $B_2' = B_2$, so
$ \langle B_1', B_2' \rangle_+ = \langle B_1 + C_k, B_2 \rangle_+
= \langle B_1, B_2 \rangle_+ + \langle C_k, B_2\rangle_+ = 1$. 

Finally, in case $\langle B_1, C_k \rangle_+ = 0 = \langle C_k, B_2 \rangle_+ $
we have 
to show that $\langle B_1', B_2' \rangle_+ = 0$.  Here $B_1' = B_1$ and 
$B_2' = B_2 + C_k$, and so 
$ \langle B_1', B_2' \rangle_+ = \langle B_1 , B_2 + C_k\rangle_+
= \langle B_1, B_2 \rangle_+ + \langle B_1, C_k\rangle_+ = 0$.
\end{proof}

We define the mutation globally by gluing the mutation on $\cL_{(k)}$ to the identity
elsewhere.  That is, writing $\cL^{(k)}$ for an $\epsilon$-neighborhood of the closure of the complement of $\cL_{(k)}$.

\begin{definition}
Let $\cC$ be a curve collection on a surface $\cL$, and $C_k \in \cC$ a simple closed curve.  
We define the mutation of $\cC$ at $C_k$ to be the curve configuration which coincides with 
$\cC$ inside of $\cL^{(k)}$ and coincides with $\cC_{(k)}'$ inside $\cL_{(k)}$. 
We denote this mutation by $\cC'$. 
\end{definition}

\begin{remark}
To ensure that $\cC_{(k)}'$ and $\cC_{(k)}$ agree in a collar sufficiently close to the boundary of $\cL_{(k)}$, 
we  choose $\epsilon$ small enough that $\cL_{(k)} \cap \cL^{(k)}$ lies in this collar.
\end{remark}

The following globalizes Proposition \ref{prop:localcurvemutation}:

\begin{proposition}\label{prop:globalcurvemutation}
For any curve $C \ne C_k$, we have 
$\langle C', C'_k \rangle_+ =  \langle C_k, C \rangle_+ $
and $\langle C'_k, C' \rangle_+ =  \langle C, C_k \rangle_+ $.
For $C_i, C_j \ne C_k$, 
$$\langle C'_i, C'_j \rangle_+ = \langle C_i , C_j \rangle_+ + \left( \langle C_i,  C_k \rangle_+ 
\cdot \langle C_k , C_j \rangle_+ \right)$$ 
\end{proposition}
\begin{proof}
The change of intersections happens in the collar $\cL_{(k)}$,
so we restrict attention here; the result now follows from Proposition \ref{prop:localcurvemutation}. 
\end{proof}

\subsection{Iteration of curve mutation}\label{sec:iteration}

Mutating a curve configuration as in Definition \ref{def:mutation} returns another curve configuration. 
However, even when the original curve configuration contain only embedded curves, this may no longer 
be the case in the mutated configuration; see Figure \ref{fig:selfintintro}.  Since we only mutate at embedded
curves, this constrains the possibility of iterating the mutation procedure.  
We introduce notation to name the difficulty: 

\begin{definition}
Let $\cC$ be a curve configuration.  Its  {\em intersection quiver} $Q_\cC$  has for vertices  the curves
the curves of $\cC$.  Arrows from $C_i$ to $C_j$ are geometric intersections contributing to  $\langle C_i, C_j\rangle_+$. 
The {\em algebraic intersection quiver} $Q_{[\cC]}$ is the quiver whose vertices are the curves, and which has $\mathrm{max}( \langle C_i, C_j \rangle, 0)$ arrows from $C_i$ to $C_j$.
\end{definition}

The quiver $Q_\cC$ can have loops, corresponding to self-intersections of the curves.  It can also
have oriented two cycles, when the absolute value of the algebraic intersection number of the curves is smaller than the number of geometric intersections. 
If we take $Q_\cC$ and erase all self-loops and cancel out 2-cycles until none remain, we obtain $Q_{[\cC]}$. 
While the notion of quiver mutation is usually formulated for quivers without self-loops or oriented 2-cycles \cite{FZ}, consideration of quivers such as $Q_\cC$ leads to the following generalization:

\begin{definition}
Let $Q$ be any quiver and $v_k$ a vertex with no self-loops. The mutation $\Mut_k(Q)$ of $Q$ at $v_k$ has the same vertices as $Q$ and
\begin{itemize}
\item an arrow $a: v_i \to v_j$ for each such arrow of $Q$ with $v_i, v_j \neq v_k$,
\item an arrow $a^{op}: v_j \to v_i$ for each arrow $a: v_i \to v_j$ of $Q$ with either $v_i = v_k$ or $v_j = v_k$,
\item an arrow $[ab]: v_i \to v_j$ for each pair of arrows $a: v_i \to v_k$, $b: v_k \to v_j$ in $Q$.
\end{itemize}
\end{definition}

\begin{warning*}
  Given a quiver without self-loops or oriented 2-cycles, the above notion of mutation does \emph{not} agree with the usual notion of quiver mutation, which is defined only for quivers of this kind. Rather, if we take such a quiver, perform the above mutation operation, then erase all 2-cycles created, we obtain the result of the standard notion of mutation of a 2-acyclic quiver. Note the general mutation is exactly what mutation of a quiver \emph{with potential} does to the underlying quiver \cite{DWZ}. In practice no ambiguity will result: we mean the above notion when we refer to mutation of a quiver such as $Q_\cC$ that may in principle have 2-cycles, we mean the standard notion when we refer to mutation of a quiver such as $Q_{[\cC]}$ that by definition cannot. 
\end{warning*}

\begin{proposition}
If $C_k$ is a simple closed curve in $\cC$, $Q_{\Mut_k(\cC)}$ (resp. $Q_{[\Mut_k(\cC)]}$) is the mutation of $Q_{\cC}$ (resp. $Q_{[\cC]}$), at $C_k$. 
\end{proposition}
\begin{proof}
This is a restatement of Proposition \ref{prop:globalcurvemutation}. 
\end{proof}

Note in particular that vertex $v_i$ which participates in a 2-cycle with $v_j$  creates a self-loop at $v_j$.  Mutating at the corresponding curve $C_i$ creates a self-intersection in the resulting $C_j'$.  While this is well defined, 
it is not desirable, since we do not then know how to mutate at $C_j'$.

\begin{definition}
A curve configuration  $\cC$ is {\em simple} if $Q_\cC$ has no loops or oriented two-cycles.  In other words, 
if all curves are embedded, and the algebraic and geometric intersection numbers agree up to sign. 
\end{definition}

However, in some cases we can make use of the freedom that, while for definiteness 
we have defined the curve configuration as co-oriented immersed curves in a surface, in fact we only care about the 
Legendrian lifts of these curves, up to Legendrian isotopy.  In terms of the curves in the surface, we may
isotope them past each other, as long as we do not in the process pass through a tangency of curves 
{\em with the same co-orientation}.  

In some cases, it is possible to Legendrian isotope the curve configuration to cancel a pair of oppositely 
oriented intersections between curves.   For example, a mutation at $C_k$ followed by 
a mutation
$C_k'$ generally creates pairs of intersections which can be cancelled by Legendrian isotopy. 

\begin{definition}
A simple configuration $\cC$ of embedded co-oriented curves is {\em nondegenerate} if it admits arbitrary sequences of mutations. That is, for every list  $(i_1,\dotsc,i_n)$ of indices of circles, there is a sequence of Legendrian
isotopies $(\iota_1,\dotsc,\iota_{n-1})$ such that 
$\iota_k \circ \mu_{i_k} \circ \iota_{k-1} \circ \mu_{i_{k-1}} \cdots \iota_1 \circ \mu_{i_1}(\cC)$ 
is a simple curve configuration, for all $k$. 
\end{definition}

We know one family of such examples. 

\begin{theorem}
A configuration $\cC = \{C_i\}$ of co-oriented geodesics on $T^2 = \bR^2/\bZ^2$, equipped with its standard Euclidean metric, is nondegenerate.
\end{theorem}

\begin{proof}
First note that for any $i$, $j$ all intersections between $C_i$ and $C_j$ are of the same sign; that is, the number of intersections between $C_i$ and $C_j$ is exactly the absolute value of $\langle C_i, C_j\rangle$. Indeed, the sign of all such intersections is determined by the slopes of any lifts of $C_i$ and $C_j$ to $\bR^2$. Thus the quiver $Q_\cC$ does not itself contain any 2-cycles.

We claim that the configuration obtained by mutating at some $C_k$ is again equivalent to one consist entirely of geodesics. Inductively applying the observation of the previous paragraph, it will then follow that $\cC$ is nondegenerate. 

Choosing coordinates appropriately on the universal cover $\R^2$, we may assume 
$C_k$ lifts to a rightwardly co-oriented vertical line. 
The Dehn twist around $C_k$ lifts to a homeomorphism from $\bR^2$ to itself which is isotopic to the linear homeomorphism given by the matrix $$\begin{pmatrix} 1 & 0\\ 1 & 1 \end{pmatrix}.$$ On $T^2$ this isotopy descends to one which simultaneously straightens the curves $C'_j$ which were twisted by mutation at $C_k$ (that is, the $C'_j$ for which $\langle C_j, C_k \rangle > 0$).

It remains to argue that this isotopy of the twisted curves lifts to an isotopy of the Legendrian lift of the entire configuration. This happens provided that whenever a curve being isotoped becomes tangent to a curve which is not moved, their co-orientations are opposite. The stationary $C'_j$  are those $C'_j$ for which $\langle C_j, C_k \rangle \leq 0$, in which case $C'_j = C_j$. These have the property that any lift to $\bR^2$ is a straight line which is either vertical or whose co-orientation points downward. On the other hand, for the moving $C'_j$, the geodesic at the end of the isotopy lifts to a straight line which is co-oriented upward. The twisted curve $C'_j$ does not lift to a straight line, but nonetheless can be chosen so that its lift is upwardly co-oriented away from its vertical tangents. The straightening isotopy can be chosen to preserve this property, thus only creating tangencies  between a downward co-oriented curve and an upward co-oriented curve (see Figure \ref{fig:straightening}). 
\end{proof}

\subsection{Mutation as seen by the disk}
\label{subsec:con-mod}

We have drawn the mutation from the point of view of curves on the surface $\cL$.  While our discussion made it
seem as if it was a discrete process, in fact there is a natural interpolation between the before and after 
configurations.  However, it cannot be seen from the point of view of the surface.  Instead, we describe it from
the point of view of the disk being attached to the curve $C_k$ at which we are mutating. 

First we describe the neighborhood of the disk $D_k$ inside the skeleton $\L$, supposing no other curves met
$C_k$.  In this case, the skeleton looks locally like the union of a cylinder --- $\cL_{(k)}$ from the previous
discussion --- with a disk $D_k$, glued in along $C_k$.  Observe that while this cannot be drawn conically
inside the $T^* \cL_{(k)}$, it can be drawn conically inside the cotangent bundle an $\R^2$ which
contains $D_k$ as the unit disk --- we take $\cL_{(k)}$ to be the conormal bundle of the boundary of 
the disk.  

Another picture we shall use, and denote by $\ghat$, is the union of the zero section of $T^*\R^2$ with the 
``inward'' conormals to the disk.  Topologically, this is again a disk glued to the cylinder, although it is not
diffeomorphic to the previous one.  (There is no reason it should be: a skeleton is the union of downward
flows of a Morse function, and the natural relation between them as the Morse function varies is not diffeomorphism.)  

Under the disk surgery \cite{Y}, we should see the disk shrink and then regrow in another way.  The key insight
in this section is that the movie of disk surgery can be seen as the cone over a certain Legendrian isotopy, at least
in the conical model $\ghat$.   This fact will ultimately allow us to define a mutation functor on sheaf categories,
using \cite{GKS}.  

\begin{figure}
\centering
%
%
\includegraphics{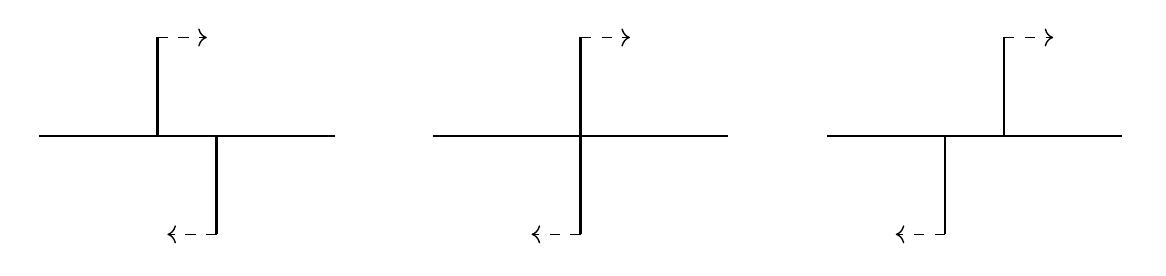}
\caption{The disk surgery as the cone over a Legendrian isotopy}
\label{fig:1ds}
\end{figure}

\begin{figure}[h]
\centering
\includegraphics{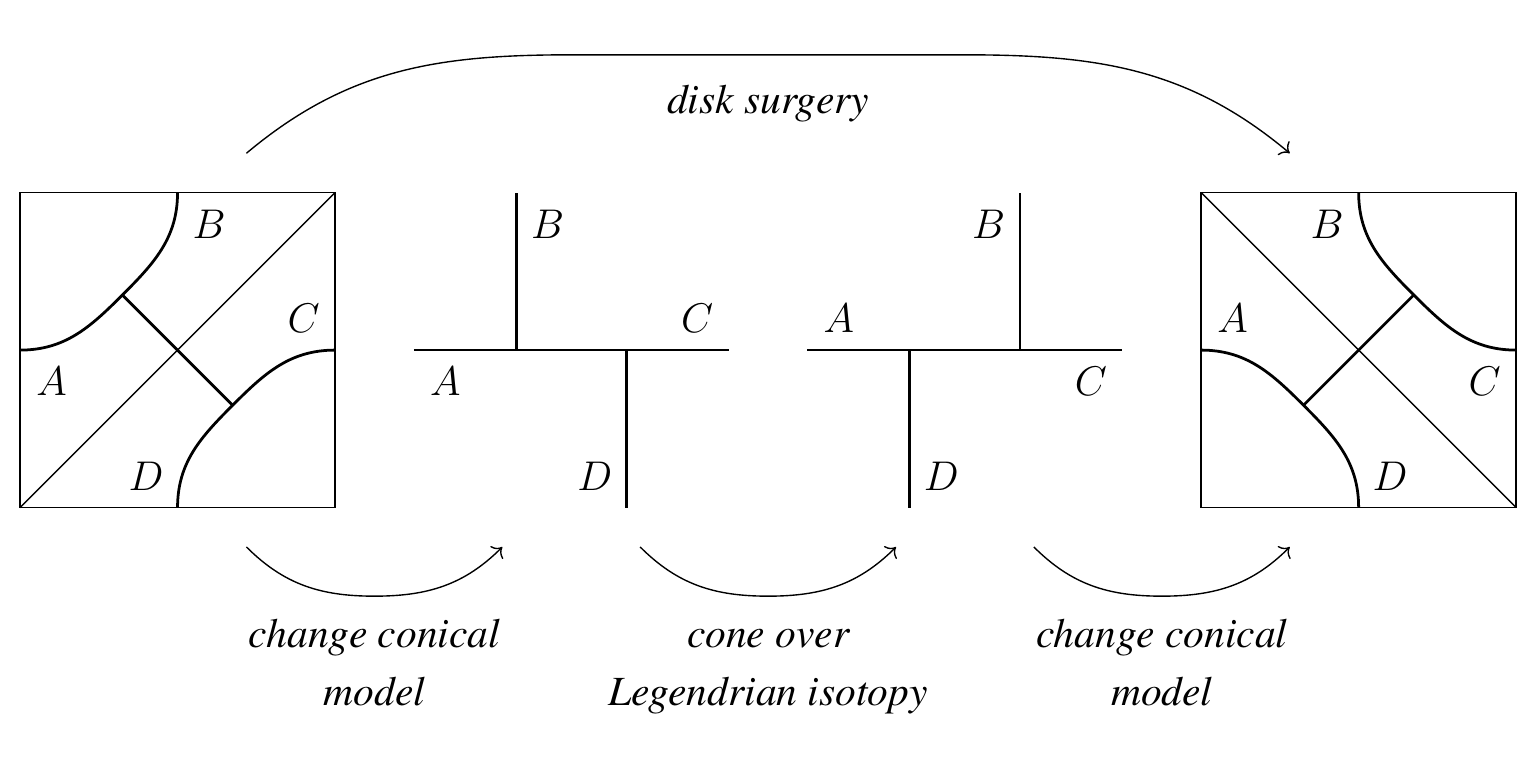}
\caption{Disk surgery in Weinstein 2-manifolds as a change of triangulation. In 4-dimensions we describe a skeletal surgery by passing through certain local conical models, whose analogues are pictured here.}
\label{fig:triangle}
\end{figure}

When no other curves meet $C_k$, we can take the isotopy to be the geodesic
flow (i.e. negative Reeb flow) --- i.e., flowing each point of the cosphere bundle
in the dual tangent direction along the base, while preserving the cotangent vector.  The curve $C_k$ lies 
at infinity along the inward conormal to the disk; after time 1 geodesic flow, 
it becomes the cocircle over the origin, and after time 2, it becomes the outward conormal to the disk.  Meanwhile, 
the disk bounded by the image of $C_k$ contracts and then expands. 
This makes sense in any dimension; we give the one-dimensional picture in \ref{fig:1ds}.  We denote the union
of the outward conormal to the disk and the zero section as $\mutghat$.  

\begin{remark}
For 1-dimensional skeleta of 2-dimensional Weinstein manifolds, the disk surgery is dual
to flipping the triangles in a triangulation. See Figure \ref{fig:triangle}, and see
\cite{DyK} for a corresponding account of the Fukaya category.
\end{remark}

We will now give a similar, albeit more complicated, account of the general situation
--- when the curve $C_k$ may meet other curves. 

\begin{definition} \label{def:ghat}
We define $\ghat \subset T^* \R^2$ as the union of the zero section
and the cone over a singular Lagrangian $\ghat^\infty$.  
We write $\partial \ghat$ for the intersection of $\ghat$ with a large $S^3$. 

We specify $\ghat^\infty$ by giving its front projection with co-orientation.   
This is the disjoint union of an inwardly co-oriented unit circle, $\gamma_k$,  with pairwise disjoint open intervals $\{\gamma_p\}_{p \in I_k}$, where $I_k$ is set of intersections of $C_k$ with other curves in $\cC$.  We fix a homeomorphism of $\gamma_k$ and $C_k$ and use this to freely identify $p \in I_k$ with points of $\gamma_k$.  The position and co-orientation of each $\gamma_p$ are chosen so that $p \in \gamma_k$ is contained in the closure of $\gamma_p$ and so that the Legendrian lift $\ghat^\infty$ is connected.  
\end{definition}

The result of mutation also has a conical model.   
Again we define this space as the union of the zero section in $T^* \R^2$ with the
conormal to a certain Legendrian knot $\mutghat^{\infty}$, which, in turn, we define
as the result of applying a cut-off geodesic flow.

Explicitly, let $(x, y)$ be the usual coordinates on $\R^2$,
and $\theta$ an additional angular coordinate for the co-circle bundle. 
Let $f:\R^2 \to [0,1]$ be a smooth function that is vanishes 
outside of a disk of large radius and is identically $1$ on a slightly 
smaller disk.  We define 
\begin{equation}
\label{eq:damp_flow}
F_t(x,y,\theta) = (x+f(x,y) t\cos(\theta),y+f(x,y)t\sin(\theta),\theta)
\end{equation}
The flow $F_t(\ghat^\infty)$ is illustrated in Figure \ref{fig:flow}.

\begin{definition} \label{def:mutghat}
We define $\mutghat^\infty := F_{2}(\ghat^\infty)$, and $\mutghat$ as the union of the cone 
over $\mutghat^\infty$ with the zero section. 
\end{definition} 

Recall that we write ${\cL_{(k)}}'$  and $\cL^{(k)} {}'$
 {\em not} for a neighborhood of $C'_k$ and its complement, but instead for the parts of 
 $\cL'$ which are the images of $\cL_{(k)}$ and $\cL^{(k)}$ under the fixed identification 
 $\cL \cong \cL'$ used in defining the mutation of curve configurations.   In particular, recall that 
 the restricted curve configuration ${\cC_{(k)}}'$ will generally contain intersections amongst the curves
 ending on the boundary of ${\cL_{(k)}}'$, whereas a neighborhood of $C'_k$ would not.  

\begin{remark}
The ``neighborhood of $C'_k$'' would be naturally denoted by  
${\cL'}_{(k)}$, similarly for ${\cL'}^{(k)}, {\cC'}_{(k)}$.  Compare the ordering of the prime and the $k$ with the above.
However, we never use these subsets in this paper: any occurences below of the these symbols
are misprints for the other ordering. 
\end{remark}

\begin{definition}
Let $\L_{(k)}$, resp. ${\L_{(k)}}'$,
be the skeleton resulting from applying Definition \ref{def:L} to $(\cL_{(k)}, \cC_{(k)})$, resp. $({\cL_{(k)}}',  {\cC_{(k)}}')$. 
\end{definition}

\begin{proposition}

There are homeomorphisms, respecting the obvious identifications at the boundary, $\L_{(k)} \cong \ghat$
and $\L_{(k)} \cong \mutghat$.
\end{proposition}
\begin{proof}
The real content here is the assertion that the flow $F_t$ creates the correct intersections in the projections 
of $\mutghat^\infty$.  This can be seen by inspection of Figure \ref{fig:flow}. 
\end{proof}

\begin{figure}
\begin{center}
\includegraphics[scale = .35]{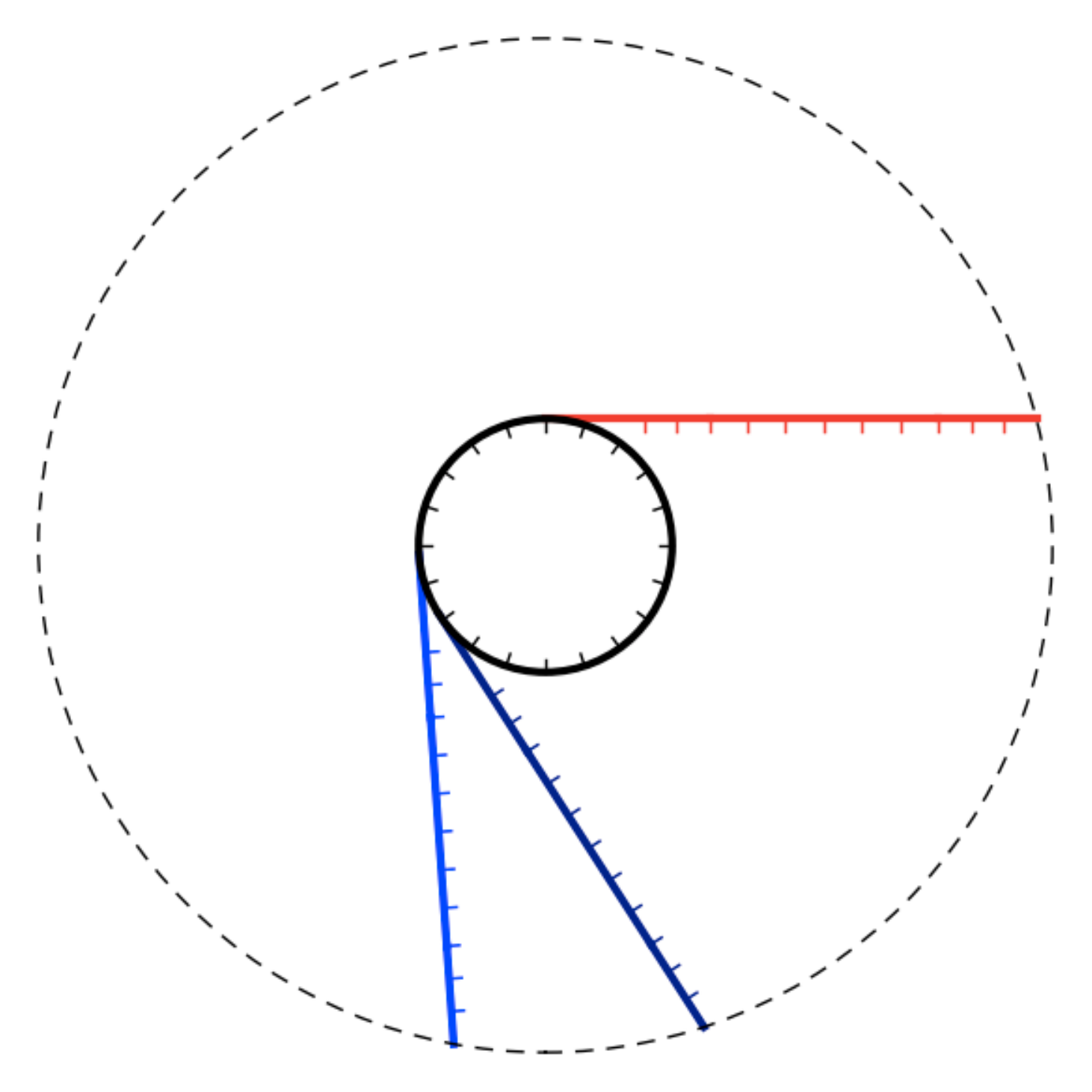} \includegraphics[scale = .35]{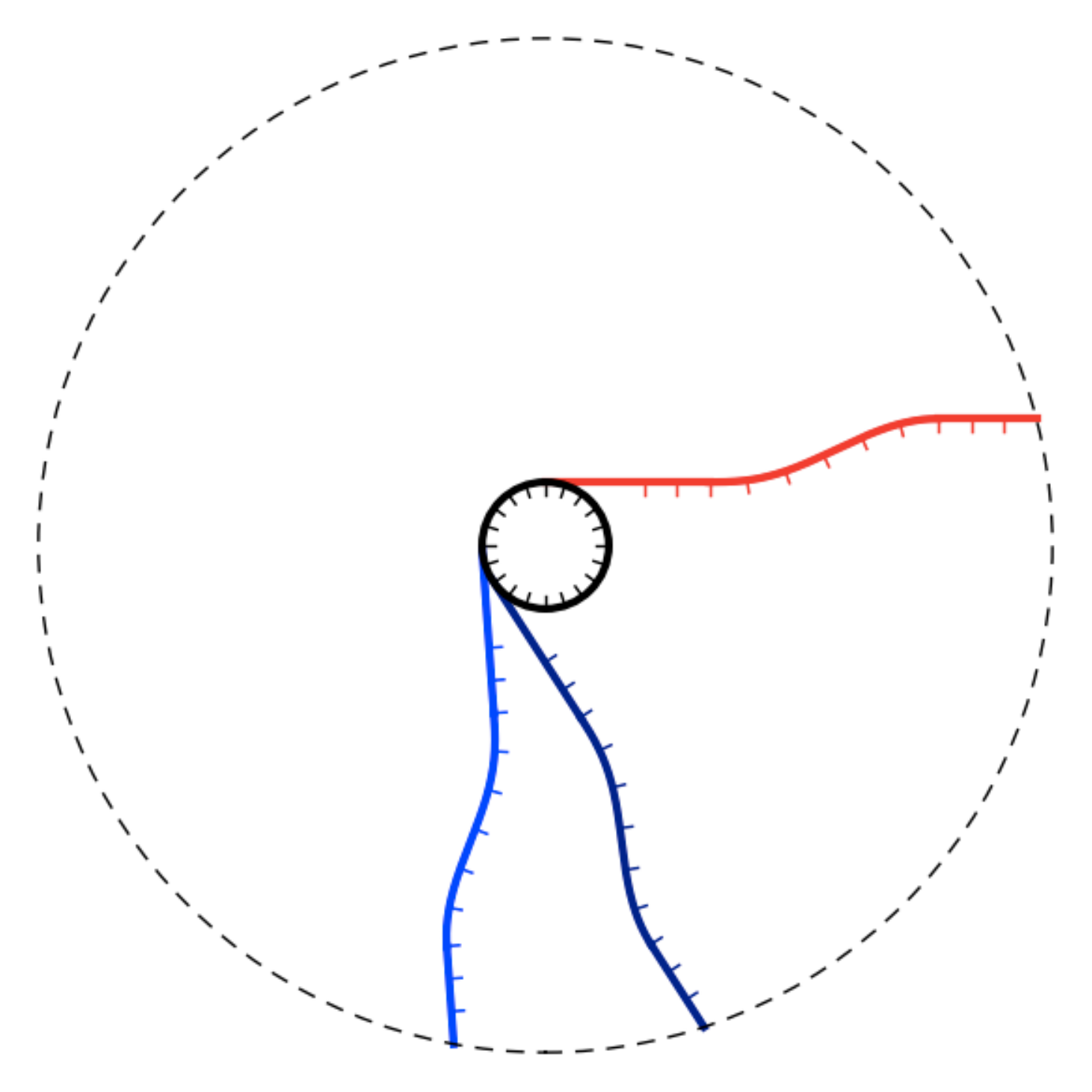} 
\includegraphics[scale = .35]{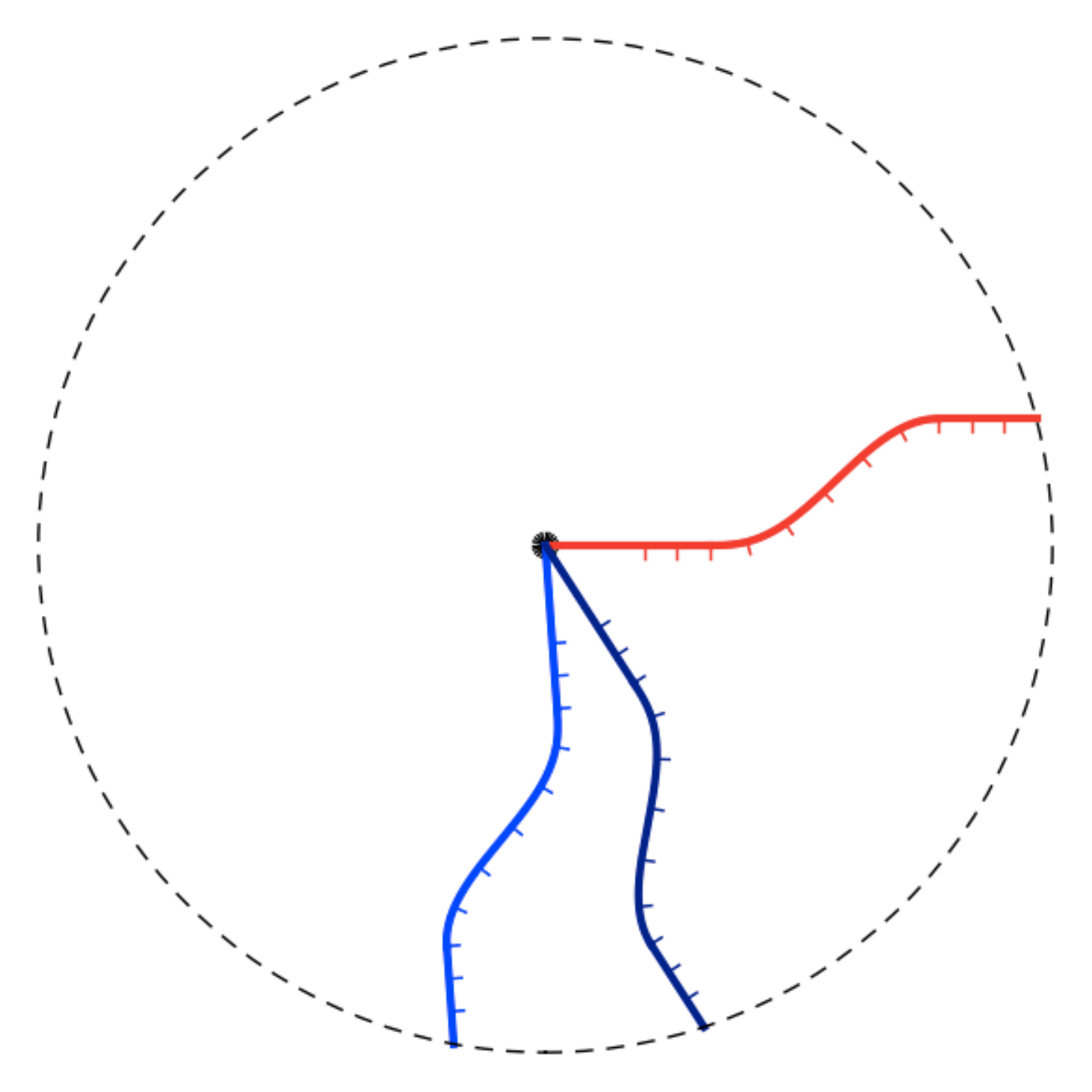} 
\end{center}
\begin{center}
\includegraphics[scale = .35]{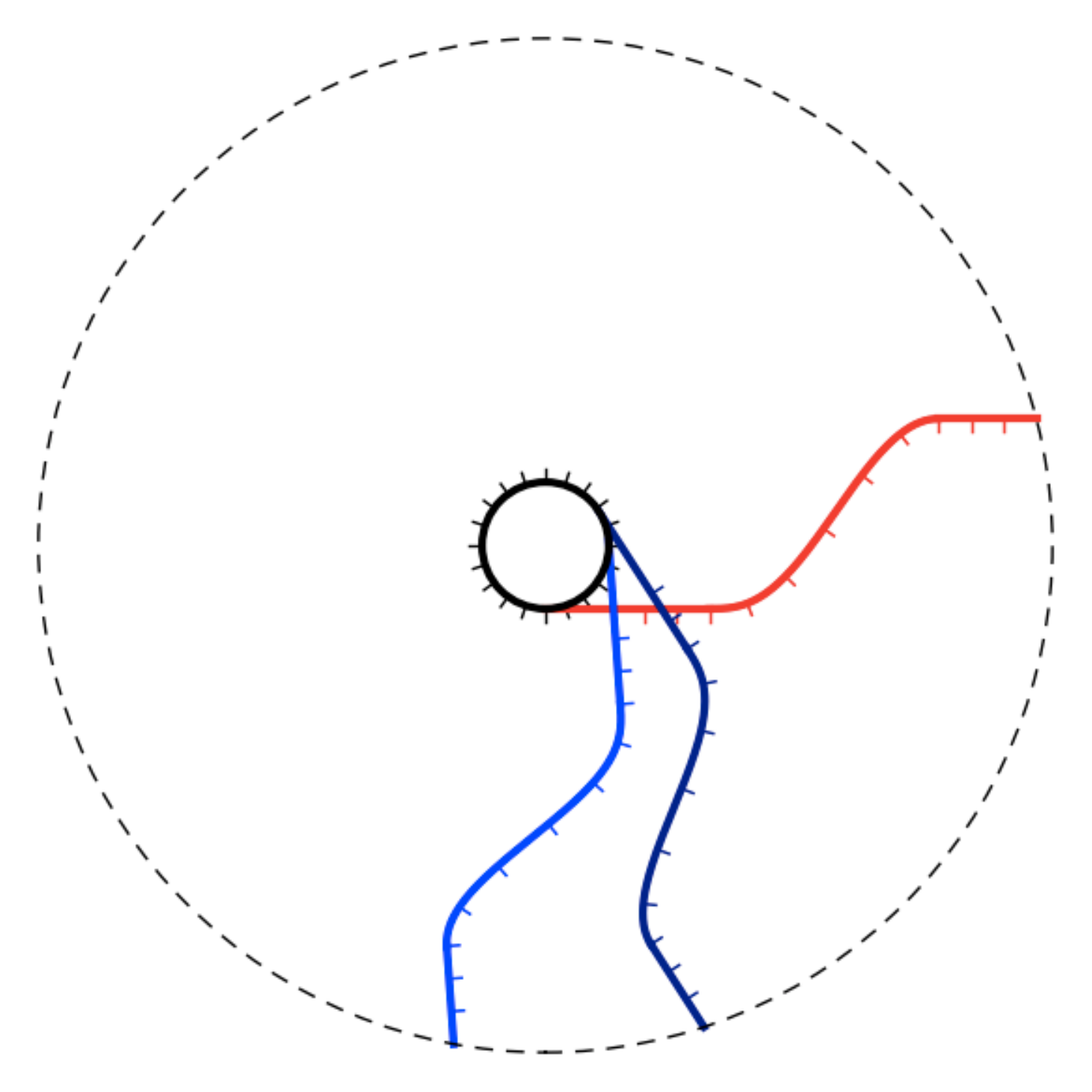} \includegraphics[scale = .35]{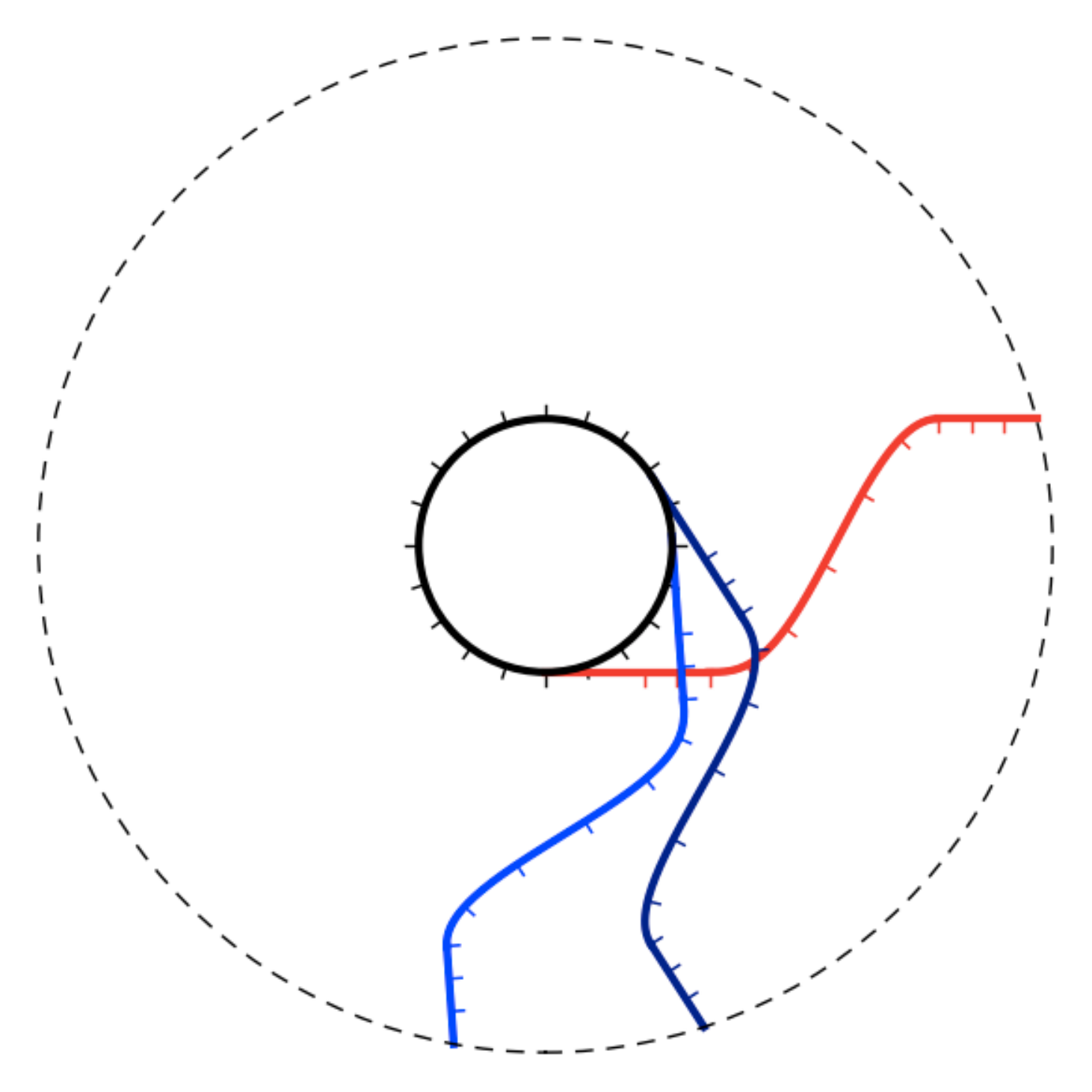}
\end{center}
\caption{The flow $F_t$ from $\protect\ghat^\infty$ (first frame) to $\protect\mutghat^\infty$ (last frame).  Note the distinction from Figure \ref{fig:mutate2}: in the first frame $\cL_{(k)}$ does not lie entirely in the page, but rather is the union of the annulus outside $C_k$ (the black circle) and the conormal to $C_k$.  The disk inside $C_k$ as drawn on the page is the disk $D_k$.  The colored strands are other $C_i$ which intersect $C_k$.}\label{fig:flow}
\end{figure}

\subsection{Mutation inside the 4-manifold}\label{sec:mut4}

A curve configuration $\cC$ and its mutation $\cC'$ give rise to two symplectic manifolds $W$ and $W'$ via the construction of Definition \ref{def:W}.  In this section we consider the relation between $W$ and $W'$.

As we have constructed $W$ by Weinstein handle attachment, it is an exact symplectic manifold 
--- i.e. the symplectic form has a primitive, $\omega = d \theta$ --- with 
convex contact boundary --- i.e., along the boundary, the vector $X$ with $\omega(X, \cdot) = \theta$
points outward.  Actually, we shall prefer to attach an infinite conical end along this boundary, which we 
do without changing the notation for $W$. 
More generally, our conventions regarding exact symplectic manifolds follow \cite{Sei}.  
Our notion of isomorphism of such manifolds is Liouville isomorphism: a symplectomorphism
$f: W \to W'$ such that moreover $\theta - f^* \theta'$ is exact and compactly supported.

Previously, we have viewed both $\cC$ and $\cC'$ as living on the same surface.  However
for our present purposes, we take the view that $\cC'$ sits on a surface $\cL'$ 
which is different from but diffeomorphic to the original surface $\cL$, by a diffeomorphism which we 
have chosen in the process of defining it.  Similarly for the local curve configuration 
$\cC_{(k)}, {\cC'}_{(k)}$.  This choice could have been made in other ways; but a best choice
will be ultimately dictated by Theorem \ref{prop:clustertrans}. 

\begin{theorem}\label{thm:symplectomorphism}
There is a Liouville isomorphism of $W$ and $W'$ which identifies $\L$ and $\L'$ outside a neighborhood of the disk $D_k$.  The preimage of $\cL'$ under this symplectomorphism is obtained from $\cL$ by Lagrangian disk surgery on  $D_k$.  
\end{theorem}

The proof of this theorem occupies the remainder of this section. 
As noted in \cite[Lem. 2.2]{Sei}, Moser's lemma implies that in any isotopy of Liouville
structures preserves the Liouville isomorphism type. 
We construct such an isotopy by describing an isotopy
of skeleta; which implicitly names an isotopy of the corresponding flows.   
We will be somewhat informal, 
omitting in particular all rounding-of-corners arguments.  We refer 
to the treatise \cite{CE} for many methods of manipulating such manifolds. 

\begin{remark}
Strictly speaking, the theorems asserted in the remainder of the paper concern constructible sheaf categories
on skeleta, and do not depend on the results of this subsection. 
\end{remark}

Recall that we write $\cL_{(k)}$ for a small neighborhood in $\cL$ of $C_k$, and $\cL^{(k)}$ for 
a complementary chart; we abusively write $\partial \cL_{(k)} = \partial \cL^{(k)}$ for the overlap
of these charts; equivalently a collar neighborhood of either one of their boundaries.  We likewise 
denote by $\cC_{(k)}, \cC^{(k)}, \partial \cC_{(k)} = \partial \cC^{(k)}$ restriction of the curve configuration
to these spaces.    

Evidently we can apply Definition \ref{def:W} to any of these; we denote the resulting manifolds
by $W_{(k)}, W^{(k)}, \partial W_{(k)} = \partial W^{(k)}$.  Evidently
$$W = W_{(k)} \stackbin[\partial W_{(k)}]{\bigcup}{} W^{(k)}$$

Let us describe $W_{(k)}$ more explicitly.    First we make the co-disk bundle 
$D^* \cL_{(k)} = D^2 \times \cL_{(k)}$.  This is a manifold with
corners: it has boundary components given by the cocircle bundle
$S^* \cL_{(k)} = S^1 \times \cL_{(k)}$, and $D^*\cL|_{\partial \cL_{(k)}} = 
\partial \cL_{(k)} \times D^2 = S^1 \times S^0 \times D^2$.  These intersect along
$S^* \cL|_{\partial \cL_{(k)}} = S^1 \times S^0 \times S^1$.  We now attach
a Weinstein handle along the lift $\Lambda_k$ of $C_k$. 

We write the resulting manifold-with-corners as $W_{(k)}$.  
The intersection of the corresponding Lagrangian skeleton $\L_{(k)}$
with $\partial W_{(k)}$ is as follows.  In each component of 
$D^*\Sigma|_{\partial \Sigma_{(k)}}$, the intersection is a circle of $\partial\Sigma_{(k)}$, 
emanating radial spokes for corresponding to the disk fragments being attached along the 
$C_i \cap \Sigma_{(k)}$ for $i \ne k$.  This disk fragments also intersect the remaining
boundary component, so in all $\L_{(k)} \cap \partial W_{(k)}$ is two circles, joined
by several lines. 

We smooth the corners of $W_{(k)}$ to get a space $\widetilde{W}_{(k)}$; alternatively it 
might be taken as an $\epsilon-$neighborhood of the disk $D_k$. 
The space $\widetilde{W}_{(k)}$ is symplectically a ball.  

We now observe that our conical models can be glued in place of $\widetilde{W}_{(k)}$
and $\widetilde{W}'_{(k)}$.

\begin{proposition} \label{prop:whatisghat}
There is a neighborhood $U$ of $\ghat$ and a symplectomorphism respecting
the boundary
$$(U, \partial U, \ghat \cap \partial U) \cong (\widetilde{W}_{(k)}, \partial \widetilde{W}_{(k)},
\L \cap \partial \widetilde{W}_{(k)})$$ 

There is a neighborhood $U'$ of $\mutghat$ and a symplectomorphism respecting
the boundary
$$(U', \partial U', \mutghat \cap \partial U') \cong (\widetilde{W}_{(k)}', \partial \widetilde{W}_{(k)}',
\L' \cap \partial \widetilde{W}_{(k)}')$$ 
\end{proposition}

\begin{remark}
The above symplectomorphism certainly does not identify $\ghat$ and $\L_{(k)}$: no diffeomorphism
can, since these spaces have different singularities.  However, one should not expect them to be identified: 
they implicitly name different Morse functions, and the appropriate relation between such functions is one of 
isotopy.  Correspondingly, it can be shown that 
there are deformations $\ghat \sim \L_{(k)}$ and $\mutghat \sim {\L_{(k)}}'$
which are  noncharacteristic in the sense of \cite{N4}.
\end{remark}

Finally, we can describe the desired symplectomorphism: 

\begin{definition} \label{def:symp}
There is a symplectomorphism $\mu_k: W \to W'$ restricting to the evident
identification on
$W \setminus \widetilde{W}_{(k)} = W' \setminus \widetilde{W}'_{(k)}$. 
Identifying $\widetilde{W}_{(k)}$ and $\widetilde{W}'_{(k)}$ to standard balls via
Propositions  \ref{prop:whatisghat}, 
we define the rest of the map on 
$\widetilde{W}_{(k)} \to \widetilde{W}'_{(k)}$ to be the identity 
sufficiently far from the boundary, and a movie of contact isotopy of 
Definition \ref{def:mutghat} near the boundary. 
\end{definition}

For a picture of what is meant, one dimension down, and beginning and ending
with the conormal-to-disk conical model rather than $\ghat$ and $\mutghat$, 
see Figures \ref{fig:lines}, \ref{fig:swirlz}, and \ref{fig:surgery1d}. 

Finally, we have a Lagrangian surface named $\cL$ inside $W$ and a Lagrangian 
surface named
$\cL'$ inside $W'$.  To see that they are related
by the Lagrangian attaching disk surgery of \cite{Y}, recall that 
the local model of the disk surgery is the passage between the hyperbolas
$xy = -\epsilon$ and $xy = \epsilon$ for $\epsilon \in \R$.   A conical model of this
transition is given by the collapsing and re-expanding of the disk in the base in the transition 
of Definition \ref{def:mutghat}.  

\newpage

\begin{figure}[h]
\centering
%
%
%
%
%
\includegraphics[scale=0.9]{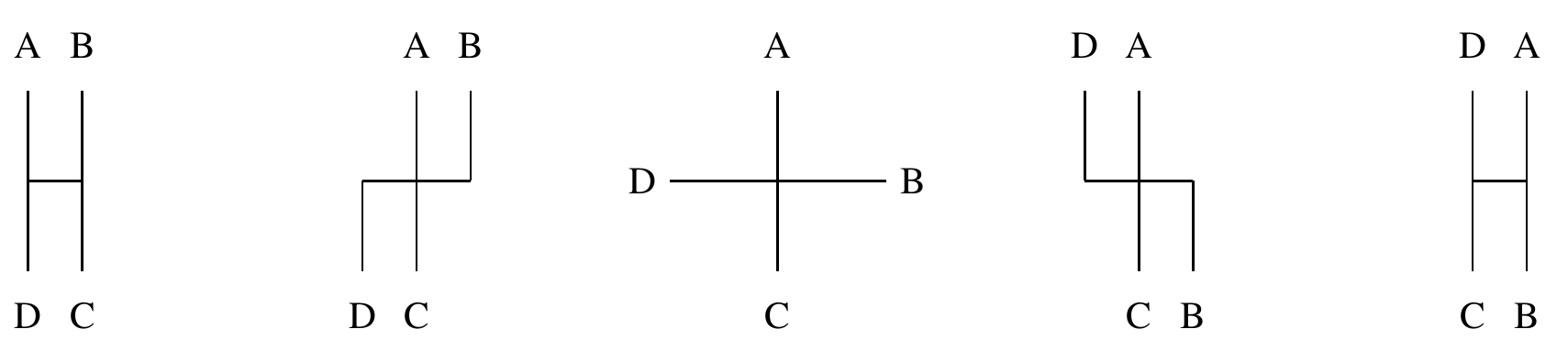}
\caption{The disk surgery, realized as a cone over a Legendrian isotopy. In this 1-dimensional picture the Legendrian is 4 points at the boundary of $T^*\R$.}
\label{fig:lines}
\end{figure}

\begin{figure}[h]
  \centering
%
\includegraphics[scale=0.8]{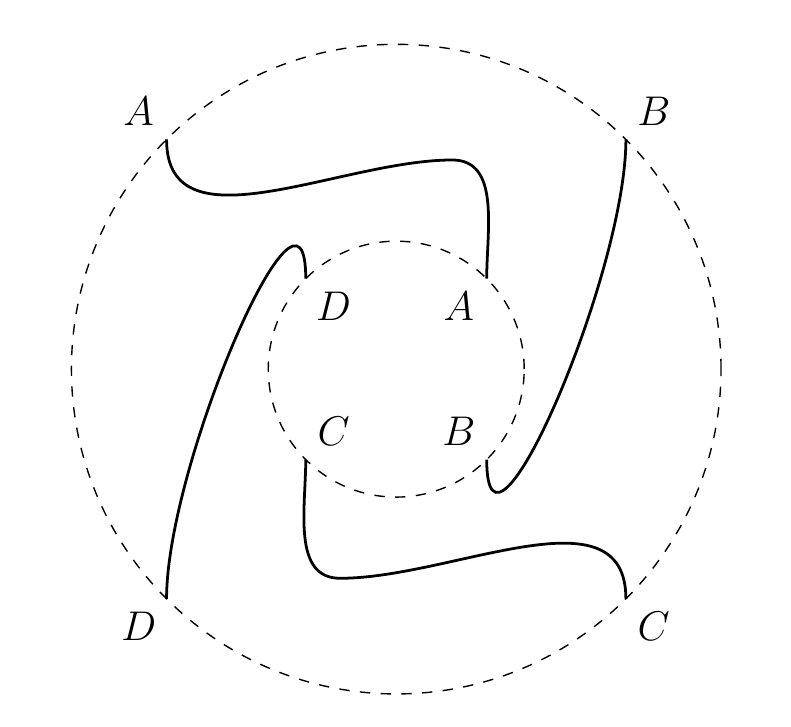}
  \caption{The movie of the Legendrian isotopy from Figure \ref{fig:lines}.}
  \label{fig:swirlz}
\end{figure}

\begin{figure}[h]
\centering
\includegraphics[scale=0.8]{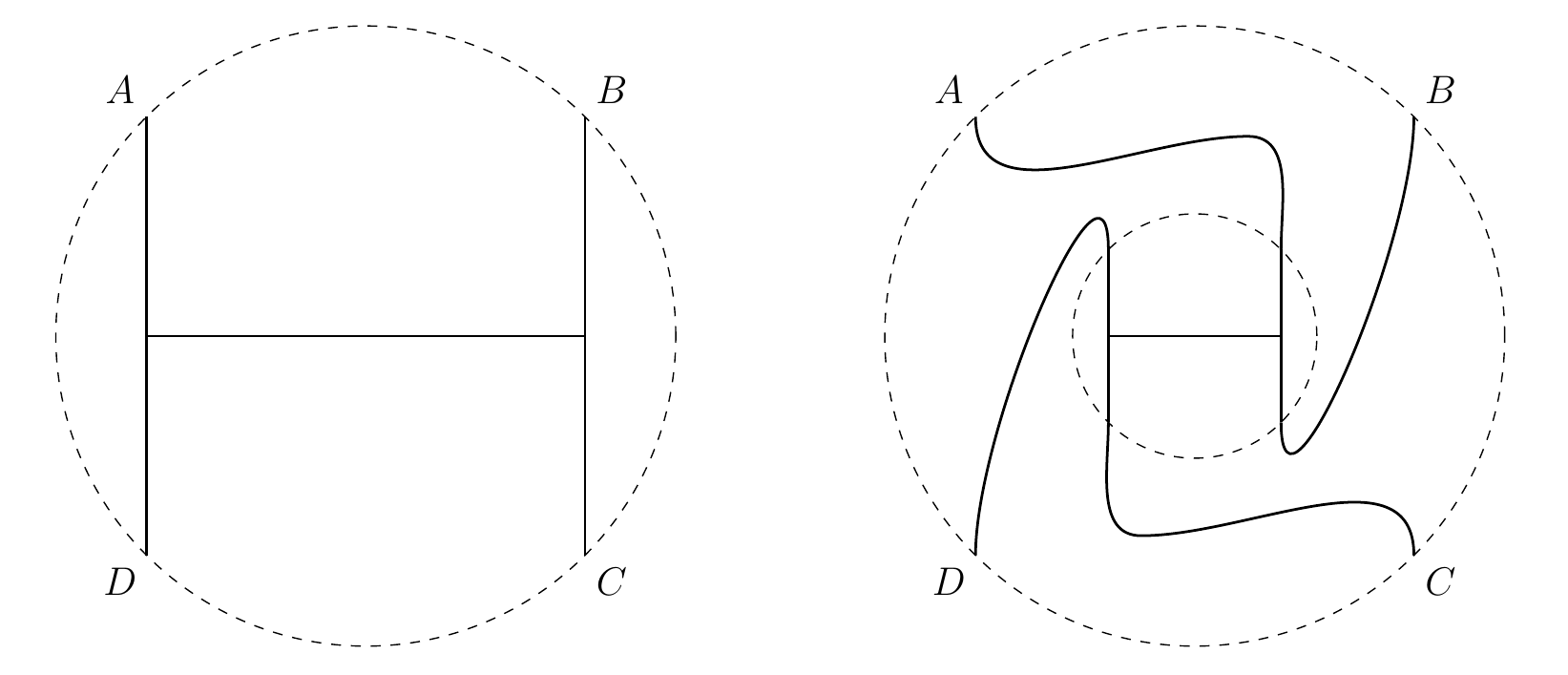}
\caption{The disk surgery, realized by gluing a transformed conical model into the
movie of the contact isotopy realizing the transformation, as in Definition \ref{def:symp}.}
\label{fig:surgery1d}
\end{figure}

\newpage

\section{Microlocal sheaves} 

In this section we define the category $\muloc(\L)$ of microlocal sheaves on the Lagrangian seed skeleton $\L$ of Definition \ref{def:W}. 
We expect this to be equivalent to a Fukaya category of $W$, following a proposal of \cite{K} and the results of \cite{NZ,N1} in the case of cotangent bundles. We glue this category together from 
Kashiwara-Schapira sheaf of conical local models. 

\subsection{Microlocal sheaves on locally conical Lagrangians}\label{sec:locallyconic}

\begin{figure}
\centering
%
%
%
%
\includegraphics{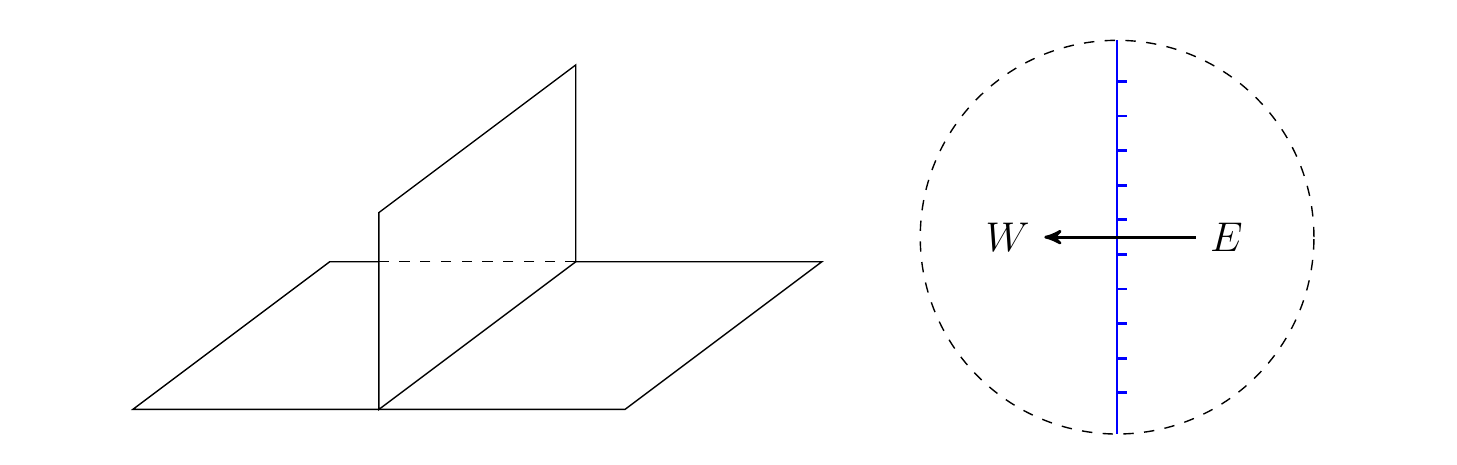}
\caption{A conic Lagrangian $L \subset T^*\R^2$ such that $\sh_L(\R^2) \cong \coeffs A_2\dmod$. On the right we indicate $L$ as the cone over the Legendrian lift of the $y$-axis, co-oriented rightward. Note the non-isomorphic restriction maps of a sheaf with stalks $W$ and $E$ go ``against the grain'' of the co-oriented curve. The topology of $L$ is pictured on the left: a plane with a single fin attached to it.}
\label{fig:A2}
\end{figure}

A Lagrangian in a cotangent bundle is conical if it is invariant under the fiberwise rescaling action.  
Recall that such a Lagrangian $\eeT \subset T^*M$ defines a category 
$sh_\eeT(M)$ of sheaves on $M$ with microsupport contained in $\eeT$ \cite{KS}.  
We very briefly review the notion of microsupport and definition of $\sh_\eeT(M)$ 
in Appendix \ref{app:sheaves} (see also the account of our conventions at the end of Section \ref{sec:intro}). 
In fact, the only geometries required for our purposes are the following three local examples, which admit simple algebraic descriptions in terms of quiver representations. 

\begin{example}
Locally constant sheaves are characterized by having singular support equal to the zero section.  
Thus, regarding $M$ as a conical Lagrangian in its cotangent bundle, $sh_M(M)$ is just $loc(M)$, the 
category of locally constant sheaves. 
\end{example}

\begin{figure}
\centering
%
%
\includegraphics{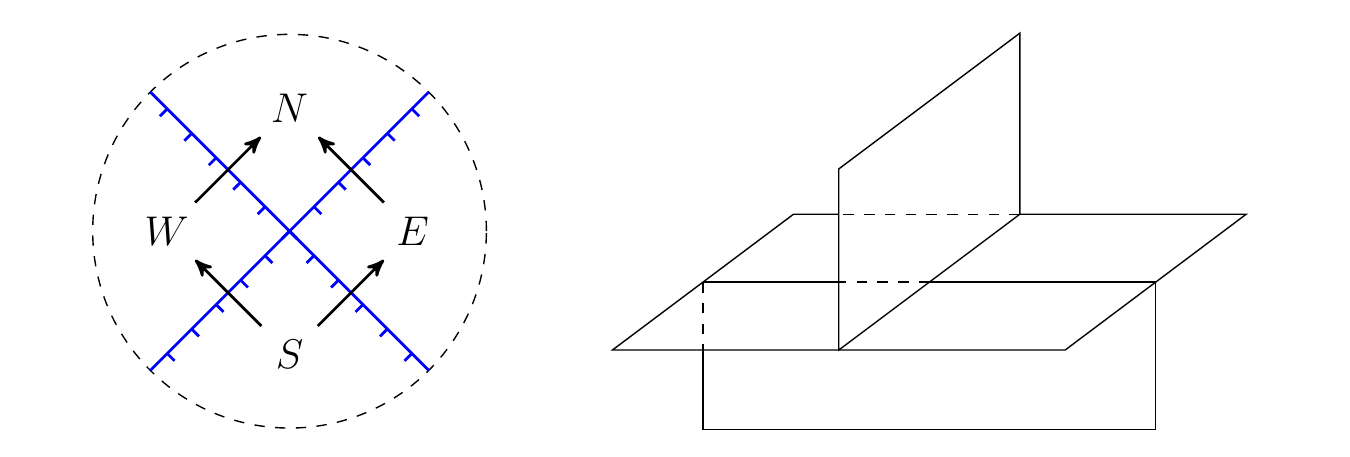}
\caption{\label{fig:A3} A conic Lagrangian $L \subset T^*\R^2$ such that 
 $\sh_{L}(\R^2) \cong \coeffs A_3\dmod$. 
On the left we indicate $L$ as the cone over the Legendrian lift of a pair of co-oriented curves. 
The topology of $L$ is pictured on the right: a plane with a pair of fins attached to it.}
\end{figure}

\begin{example}\label{ex:A2}
Let $\eeT \subset T^*\R^2$ be the union of the zero section with the cone over $\Lambda = dx|_{\{x=0\}}$, the Legendrian whose front projection is the $y$-axis, cooriented to the right (see Figure \ref{fig:A2}). 
Then $sh_\eeT(\R^2)$ is equivalent to $\coeffs A_2\dmod$, the dg derived category of representations of the $A_2$-quiver, as follows. 
We write $W$ and $E$ for the stalks of a sheaf in the open left half-plane $\{x< 0\}$ and closed right half-disk $\{x \geq 0\}$, respectively (all stalks in either region are canonically isomorphic up to homotopy). 
There is a generization map $E \to W$ given by restricting from a neighborhood of a point on the $y$-axis to a smaller open set lying entirely to the left of the $y$-axis. 
We also have the microlocal stalk at a point of $L$ lying off the zero section, hence above the $y$-axis: it is the cone over this map.  
\end{example}

\begin{example}\label{ex:A3}
Let $\eeT \subset T^*\R^2$ be the union of the zero section with the cone over $\Lambda = (dx - dy)|_{\{x=y\}} \cup (-dx-dy)|_{\{x=-y\}}$, the Legendrian whose front projection is the union of the lines  $x=y$ and $x=-y$, co-oriented downwards (see Figure \ref{fig:A3}). 
Then $\sh_\eeT(\R^2)$ can be described in terms of the dg category of quadruples $N$, $W$, $E$, $S$ of perfect complexes of $\coeffs$-modules, with a commuting square of maps as pictured.  Such data gives rise to an object of $\sh_\eeT(\R^2)$ 
under the following crossing condition: the total complex $S\to W\oplus E\to N$ must be acyclic \cite[Theorem 3.12]{STZ}.  

The restrictions $\sh_\eeT(\R^2) \to \sh_\eeT(\R^2 \cap \{y>\epsilon\})$, $\sh_\eeT(\R^2) \to \sh_\eeT(\R^2 \cap \{y<-\epsilon\})$ to the regions above and below the $x$-axis are equivalences.  These categories are equivalent to the representation categories of two different orientations of the $A_3$ quiver, which we obtain by forgetting $S$, $N$, respectively.  The induced equivalence $\sh_\eeT(\R^2 \cap \{y>\epsilon\}) \cong \sh_\eeT(\R^2 \cap \{y<-\epsilon\})$ is a reflection functor.
\end{example}

To glue together sheaf categories associated to conical Lagrangians, it is crucial that these categories localize 
over the entire Lagrangian --- not merely the base manifold.  That is, we want to realize these sheaf categories as the global
sections of a (homotopy) sheaf of dg categories over $\eeT$.  
While working with sheaves of dg categories requires the homotopical 
foundations of, for example, \cite{Lur1,Toe,Tab}, the treatise \cite{KS} 
predates these texts by a number of years. 
We briefly sketch how the relevant geometric results of \cite{KS}  
may be adapted to the differential graded setting at hand (see also \cite[Sec. 2.1]{SiTZ} and \cite[Sec. 5.2]{N2}).

Kashiwara and Schapira study $T^*M$ in the conic topology, i.e., the open sets are $\R_+$-invariant.  
For each conic open subset $U \subset T^* M$, they 
define a category --- we denote it here by $\mathrm{KS}^{pre}(U)$ --- as the quotient of 
the category of sheaves on $M$ by the category of sheaves with microsupport not meeting $U$. 
This is a presheaf of dg categories on $T^* M$: if we write $\mathrm{dgCat}$ for the (appropriately
homotopical) 
category of dg categories, 
then $U \mapsto \mathrm{KS}^{pre}(U)$ is a contravariant functor from the category of open subsets of $T^* M$ to $\mathrm{dgCat}$.

Each object of $\mathrm{KS}^{pre}(U)$ has a well-defined (conic, co-isotropic) support in $U$. Thus 
for a conic Lagrangian subset $\eeT \subset T^*M$, there is a presheaf of full subcategories 
$\mathrm{KS}^{pre}_L(U) \subset \mathrm{KS}^{pre}(U)$ of objects supported on $L$, which by definition vanishes when $U \cap L$ is empty.  
The sheafification  $\mathrm{KS}_\eeT$ of $\mathrm{KS}_\eeT^{pre}$ is therefore supported on $\eeT$. 

\begin{definition}
For a conic Lagrangian $\eeT \subset T^*M$, we write $\muloc$ for the sheaf of dg categories on $\eeT$ given by
the restriction of $\mathrm{KS}_\eeT$ to $\eeT$.  
\end{definition}

The stalks of $\mathrm{KS}_{all}$, hence of $KS_\eeT$ and $\muloc$, are determined in \cite[Thm. 6.1.2]{KS}.  
In particular, 
\begin{itemize}
\item If $U \subset T^*M$ is a conic open set of the form $T^*\pi(U)$, then 
$\muloc(\eeT \cap U) \cong sh_{\eeT \cap U}(\pi(U))$. 
\item If $U \subset T^*M$ is a conic open set that does not meet the zero section and is sufficiently small, 
$\muloc(\eeT\cap U)$ is the quotient of $sh_{\pi(U) \cup (\eeT \cap U)}(\pi(U))$ by $\loc(\pi(U))$. 
\end{itemize}

These results allow us to conclude that $\muloc$ is indeed the desired localization: 

\begin{proposition}  
For a conical Lagrangian $\eeT \subset T^*M$, the global section category
$\muloc(\eeT)$ is equivalent to $sh_\eeT(T^*M)$.
\end{proposition}
\begin{proof}
The global sections of $\muloc$ can be calculated by first pushing forward along 
$\eeT \hookrightarrow T^* M \to M$ and then pushing forward to a point.  But the 
pushforward to $M$ admits a map from the ``sheaf of sheaves'' on $M$ 
--- i.e. the sheaf whose value over $U \subset M$ is the derived category of sheaves on $U$.  
It follows from the above results that this map is an isomorphism at stalks, hence an isomorphism
of sheaves, hence induces an isomorphism of global sections. 
\end{proof}

We are interested in gluing conical Lagrangians along open neighborhoods of their boundary components --- both components in the base manifold $M$, and  components at contact infinity $T^\infty M$. 
As $\muloc$ is constant in the fibre direction, we are not careful with the difference between $\Lambda$ and a collar
neighborhood of $\Lambda$.  

It will be useful to have the following trivialization of the restrction of $\muloc$ to $\Lambda$.  

\begin{lemma} \label{lem:mumon}
Let $\eeT \subset T^*M$ be a conical Lagrangian such that $\Lambda = \eeT \cap T^\infty M$ is 
smooth and $\pi: \Lambda \to M$ is an immersion.  Let $\loc$ denote the sheaf of categories on 
$\Lambda$ whose sections give categories 
of local systems.   On $\Lambda$, there is a canonical equivalence 
$\muloc|_{\Lambda} \cong \loc$. 
\end{lemma}
\begin{proof} 
Since $\pi$ restricts to an immersion on $\Lambda$, it restricts to an 
embedding on any sufficiently small open ball $B \subset \Lambda$.  Let $f$ be a real-valued smooth function on a neighborhood of $\pi(B)$ such that $f^{-1}(0) = \pi(B)$ and $df|_{\pi(B)}$ is nonzero and contained in the cone over $\Lambda$.  To an object $F$ of $\muloc(B)$ we assign the local system on $B$ whose global sections are $Cone (F(f^{-1}(-\infty,\epsilon)) \to F(f^{-1}(-\infty,-\epsilon)))$.  
This defines an equivalence $\muloc(B) \congto \loc(B)$ \cite[Chap. 3]{KS}, compatible with restrictions hence inducing 
an isomorphism of sheaves. 
\end{proof}

\begin{remark}
When $\Sigma$ is a surface and $\Lambda$ a Legendrian knot such that $\Lambda \to \Sigma$ is not necessarily
an immersion, the rotation number of $\Lambda$ measures the failure of $\muloc|_\Lambda$ to be trivializable \cite{STZ}. 
See \cite{Gui} for some considerations in the general case. 
\end{remark}

We are now ready to consider microlocal sheaves on spaces which are only locally conical:

\begin{definition}\label{def:KSsheafgeneral} 
A Kashiwara-Schapira (KS) sheaf on a topological space $\T$ is a sheaf $\muloc$ of dg categories such that 
there exists 
\begin{itemize}
\item an open cover $\T = \bigcup \T_i$
\item embeddings $\iota_i: \T_i \into T^* M_i$ whose images are conical Lagrangians
\item equivalences $\muloc|_{\T_i} \cong \muloc|_{\iota_i(\T_i)}$.  
\end{itemize}
\end{definition}

We call a category of the form $\muloc(\T)$ for some KS sheaf $\muloc$ ``a category of microlocal sheaves on $\T$''.

\begin{remark}
In the above assertion $\muloc|_{\T_i} \cong \muloc(\iota_i(\T_i))$, the first $\muloc$ is the
given sheaf of categories on $\T$, and the second $\muloc$ is the Kashiwara-Schapira sheaf in the local model. 
\end{remark}

A space $\T$ generally does not have a unique KS sheaf --- but when $\T$ is presented as the skeleton of a Weinstein manifold $W$, it should be possible to specify a choice of $\muloc$ through some further trivializations of topological structures on $W$, or equivalently $\T$. 
We will instead content ourselves with constructing an explicit choice in the case at hand: 
given a cover $\T = \bigcup \T_i$ and conic Lagrangian embeddings $\iota_i: \T_i \hookrightarrow T^* M_i$, one can glue together a KS sheaf from the categories $sh_{\iota_i(\T_i)}(M_i)$. 
This requires specifying descent data, in particular equivalences $\muloc|_{\iota_i(\T_i \cap \T_j)} \cong 
\muloc|_{\iota_j(\T_i \cap \T_j)}$.

For a fixed $\muloc$, there will also generally be many other choices of open cover and conical embeddings witnessing the fact that $\muloc$ is a KS sheaf. 
We use the freedom to pass between such models in an essential way. 

\subsection{Microlocal sheaves on seed skeleta}\label{sec:KSseed}

We return to the setting of Section \ref{sec:geometry}, letting $\cC$ be a curve configuration on a surface $\cL$, and $\L \subset W$ the Lagrangian skeleton of the associated 4-manifold.  By construction $\L$ is equipped with an open cover by conic Lagrangians, namely the positive conormal bundle $\TC \subset T^*\cL$ of the curve configuration and, for
the closed curves $C_i$, cotangent bundles of open disks $D_i \subset T^*D_i$. 

\begin{convention} We write $\cC^\circ \subset \cC$ to denote the subset of closed curves. \end{convention}

\begin{definition}\label{def:KSsheaf} 
We define a KS sheaf on $\L$, by gluing the KS sheaves of the conical Lagrangians $\TC \subset T^*\cL$ and 
$D_i \subset T^* D_i$ along the annuli $\TC \cap D_i$.  

This requires gluing data on the annuli (and no more, since there are no multiple overlaps).
We identify each restriction with $\loc|_{\TC \cap D_i}$: for $\muloc|_{D_i} = \loc|_{D_i}$ this is immediate. 
For $\muloc|_{\TC}$ we use the Morse trivialization of Lemma \ref{lem:mumon}, composed with the autoequivalence of 
$\loc|_{\TC \cap D_i}$ given multiplying the monodromy by $\sigma(C_i)$, for some fixed choice $\sigma: \cC^\circ \to \Z/2\Z$. 

We denote the resulting sheaf as $\muloc_\sigma$, and generally omit the $\sigma$. 
\end{definition} 

\begin{remark}
The abstract construction of Definition \ref{def:KSsheaf} can be made concrete in practice: in the examples
of interest, $\muloc(\L)$ is built from categories of quiver representations by taking homotopy limits, 
and these can be computed explicitly using path categories \cite{Tab}.  We work out a
crucial example in great detail in Section \ref{sec:L0}.
\end{remark}

There is a natural identification between functions from $\cC^\circ$ to $\Z/2\Z$ and $H^2(\L, \cL; \Z/2\Z)$.  Note the long exact sequence
of cohomology 
$$H^1(\cL; \Z/2\Z) \to H^2(\L, \cL; \Z/2\Z) \to H^2(\L; \Z/2\Z) \to H^2(\cL; \Z/2\Z)$$

\begin{proposition} \label{prop:sigmambiguity}
The category $\muloc^\sigma(\L)$ only depends on the image of $\sigma$ in $H^2(\L; \Z/2\Z)$. 
\end{proposition}
\begin{proof}
If $\sigma_1$ and $\sigma_2$ differ by an element of $H^1(\cL,\Z/2\Z)$, tensoring with the associated local system is an autoequivalence of $\muloc(\TC)$ inducing an equivalence $\muloc_{\sigma_1}(\L) \cong \muloc_{\sigma_2}(\L)$. 
\end{proof}

\begin{remark}
The Fukaya category depends on a class in $H^2(W; \Z/2\Z) =  H^2(\L; \Z/2\Z)$; 
for $\cL$ to support any branes in this category, the restriction of the class to $\cL$ must vanish \cite[Sec. 12]{Sei}. 
\end{remark}

\begin{remark}\label{rmk:infrank}
On occasion it is also convenient to consider a larger category $\muloc^\infty$ of microlocal sheaves with cohomologically bounded but infinite-rank stalks. This is defined the same way as $\muloc$, but replacing the local models $\sh_{\TC}(\cL)$ and $\loc(D_i)$ by their weakly constructible versions $\sh_{\TC}^\infty(\cL)$ and $\loc^\infty(D_i)$.
\end{remark}

\begin{definition} For $p \in \L$, we write $\cF_p$ for the image of $\cF$ under the functor to the stalk category
$\muloc(\L) \to \muloc(\L)_p$. 
The {\em support} of $\cF$ is the set of points  $p \in \L$ at which $\cF_p$ is nonzero. 
\end{definition}

\begin{remark}
Note that discussing the precise value of the stalk will generally require choosing trivializations of $\cF_p$, but
discussing when it is zero does not. 
\end{remark}

 Up to equivalence $\muloc(\L)_p$ is classified by the local models discussed in Section \ref{sec:locallyconic}: when $p$ is a smooth point, $\muloc(\L)_p$ is (non-canonically) equivalent to the dg derived category of perfect complexes of $\coeffs$-modules \cite[Chap. 6]{KS}; when $p$ lies on one of the $C_i$, but not at a crossing, $\muloc(\L)_p$ is equivalent to $\coeffs A_2\dmod$; when $p$ lies at a crossing of the $C_i$, $\muloc(\L)_p$ is equivalent to $\coeffs A_3\dmod$.

\begin{proposition}\label{prop:fullyfaithful}
There is a fully faithful functor $\loc(\cL) \into \muloc(\L)$ whose essential image is the full subcategory of microlocal sheaves supported on $\cL \subset \L$. 
\end{proposition}
\begin{proof}
We have $\loc(\cL) \subset \muloc(\cL \cup \bigcup_i T^+_{C_i} \cL)$. 
Local systems have vanishing microlocal stalks away from the zero section, so the restriction
morphism $\loc(\cL) \to \loc(\coprod \Lambda_i)$ is the zero morphism.  The
only element of $\coprod D_i$ which restricts to $0$ is the zero sheaf.  There is a well-defined functor $\loc(\cL) \to \muloc(\L)$ given by
$$\loc(\cL) = \loc(\cL) \times_0 0 \subset \muloc( \cL \cup N^+ C_i) \times_{loc (\coprod \Lambda_i )} \loc(D_i).$$
It is clear from this the morphism is fully faithful. 
\end{proof} 

With this in mind we regard $\loc(\cL)$ as a full subcategory of $\muloc(\L)$ from now on without comment.
More generally, by the same sort of argument, 

\begin{proposition} \label{prop:inclusion} If $\cD$ is any subset of the curve configuration $\cC$, and 
$\iota: \L_{\cD} \subset \L_{\cC}$ is the associated inclusion of skeleta, then there is a canonical,
locally fully faithful functor $\iota_* \muloc \to \muloc$ whose global sections are a fully faithful
inclusion $\muloc(\L_{\cD}) \to \muloc(\L_{\cC})$ with image exactly equal the microlocal sheaves
supported on $\L_{\cD}$. 
\end{proposition}

We wish to discuss  objects  in $\muloc(\L)$ coming from sheaves in the underived sense, i.e., with
cohomology concentrated in degree zero.  More precisely, consider the restriction morphism 
$\muloc(\L) \to \muloc(\TC) = \sh_{\TC}(\cL)$.  We define $\muLoc(\L)$ to be the full subcategory
whose objects are the preimage of objects in $\sh_{\TC}(\cL)$ whose cohomology is concentrated in
degree zero.  

\subsection{Moduli Spaces}\label{sec:moduli}

The moduli theory of objects in dg categories such as $\muloc(\L)$ was developed in \cite{TV}. The setting is that of derived algebraic geometry; for general background we refer to \cite{Lur3,Toe2,TV2}. Our main interest here is in moduli spaces of Lagrangian branes, especially in the subspace parametrizing branes supported on a fixed exact Lagrangian: we can conclude two Lagrangians are not Hamiltonian isotopic if these subspaces do not coincide. However, as this question can be decided by considering only the truncations of the moduli spaces  involved, it is essentially one of ordinary algebraic geometry. With this in mind the reader will lose little in bypassing the discussion of derived moduli spaces and proceding with Definition \ref{def:classical} in mind.

Recall that $\muloc(\L)$ is a homotopy limit of dg categories of $A_n$-quiver representations. There are locally geometric derived stacks parametrizing such representations \cite{TV}, and taking homotopy limits of these yields the moduli stack of objects in $\muloc(\L)$:

\begin{definition} \label{def:moduli}
If $\L$ is the Lagrangian skeleton associated to a curve configuration $\cC$, we write $\R\cM(\L)$ for the locally geometric derived stack of objects in $\muloc(\L)$.
\end{definition}

In derived geometry, the infinitesimal study of derived moduli spaces can be more accessible than that of ordinary moduli spaces. For example, writing $\R Loc(\cL)$ for the derived moduli stack of local systems on $\cL$, we have the following consequence of Proposition \ref{prop:fullyfaithful}: 

\begin{proposition}
The inclusion $\loc(\cL) \into \muloc(\L)$ associated to the embedding $\cL \subset \L$ induces an open inclusion $\R Loc(\cL) \into \R \cM(\L)$.
\end{proposition}
\begin{proof}
This follows formally from the fact that $\loc(\cL) \into \muloc(\L)$ is a faithful inclusion of dg categories. Indeed, it follows from this 
that the morphism is injective on points, and since 
the tangent complexes to the moduli spaces are given by self-ext algebras \cite[Thm 0.2]{TV}, it follows that the map is \'etale. 
\end{proof}

With this in mind we will primarily restrict our attention to the following objects, where $\R Loc_n(\cL) \subset \R Loc(\cL)$ denotes the substack of rank $n$ local systems:

\begin{definition}
The moduli space $\R \cM_n(\L)$ of rank $n$ microlocal sheaves on $\L$ is the component of $\R \cM(\L)$ containing the image of $\R Loc_n(\cL)$ under extension by zero.
\end{definition}

As noted before, our present interest in the moduli space $\R \cM(\L)$ is largely in that organizes subspaces of the form $\R Loc(\cL')$ for Lagrangians $\cL'$ obtained by iterated surgery on $\cL$. While the higher and derived structures on $\R \cM(\L)$ are important for many purposes, the question of distinguishing these subspaces can studied at the level of truncations without losing any information:

\begin{proposition}\label{def:classical}
We let $\cM_n(\L)$ denote the substack of the truncation $t_0 \R \cM_n(\L)$ parametrizing objects without negative self-extensions. It is an Artin stack in the classical sense, and extension by zero induces an open map $\Loc_n(\cL) \into \cM_n(\L)$ from the classical moduli stack of rank $n$ local systems on $\cL$.
\end{proposition}
\begin{proof}
That the locus in $t_0 \R \cM_n(\L)$ without negative self-extensions is an Artin 1-stack follows from \cite[Sec. 3.4]{TV}. That we have an open map of ordinary stacks follows from the fact that local systems on $\cL$ do not have negative self-extenstions, and that the truncation of an \'etale map is \'etale \cite[Sec. 2.2.4]{TV2}.
\end{proof}


\section{Mutation functors}\label{sec:mutationfunctors}

Given a skeletal surgery $\L \rightsquigarrow \L'$ at a disk $D_k$, we now construct an equivalence of categories $\Mut_k: \muloc(\L) \congto \muloc(\L')$. 
Just as the surgery $\L \rightsquigarrow \L'$ is local to the disk $D_k$, so too is the mutation functor $\Mut_k$. 

The mutation functor  should be the microlocal counterpart of the equivalence $\Fuk(W) \cong \Fuk(W')$ associated to the symplectomorphism $W \cong W'$ of Theorem \ref{thm:symplectomorphism}, under 
the expected equivalence $\muloc(\L) \cong \Fuk(W)$. 
We compute how 
$\loc(\cL) \subset \muloc(\L)$ transforms under surgery; our calculation is modeled on 
a computation in \cite{BezKap} of the effect of the Fourier transform on perverse sheaves. 
The result will be interpreted in Section \ref{sec:cluster} as showing this transformation to be a 
cluster $\cX$-transformation or a nonabelian version thereof. 

\subsection{Sheaf category equivalences from contact isotopy}

The basic tool in our construction of equivalences between categories of microlocal sheaves is the 
``sheaf quantization'' 
theorem of Guillermou, Kashiwara, and Schapira \cite{GKS}.  Informally, this asserts that a contact isotopy of
$T^\infty M$ induces an autoequivalence of the constructible sheaf category, respecting the action of
contact isotopy on the microsupport at infinity.  

\begin{remark} 
Under the equivalence between sheaf categories and Fukaya categories \cite{NZ, N1}, this amounts to a similar
statement about the {\em infinitesimally wrapped} Fukaya category, on which contact isotopy acts almost tautologically. 
Note that, unlike for the wrapped category, such an action is generally nontrivial.  
Compatibility of our construction with the (conjectural) equivalence 
$\muloc(\L) \cong Fuk(W)$ amounts to the assertion that the Nadler-Zaslow equivalence 
interwines the action of contact isotopy on $Fuk(W)$ with the Guillermou-Kashiwara-Schapira construction. 
\end{remark}

The authors of \cite{GKS} prefer the notion of homogenous symplectomorphism of the cotangent bundle minus the zero section --- i.e. a symplectomorphism which commutes with fiberwise rescaling --- to that of contactomorphism; these
are equivalent after e.g. choosing a metric and identifying $T^\infty M$ with the unit cosphere bundle.  In any case, 
such a symplectomorphism transforms one conic Lagrangian into another.  The main result of \cite{GKS} is that 
a 1-parameter family $\{\varphi_t| t\in [0,1]\}$ of homogeneous 
symplectomorphisms induces an equivalence of microlocal sheaf categories, respecting microsupports. 

\begin{theorem}\label{thm:GKSsec} \cite{GKS} 
Let $\{\varphi_t| t \in [0,1]\}$ be a 1-parameter family of homogeneous symplectomorphisms. 
For any $t$ and any conical Lagrangian $\eeT$ containing the zero section, there is a canonical equivalence $\sh_{\varphi_0(\eeT)}(M) \cong \muloc_{\varphi_t(\eeT)}(M)$. 
It is the restriction of an autoequivalence of $\sh(M)$ given by convolution with an explicit kernel $K_{\varphi_t}$.
\end{theorem}

As noted above, a Legendrian isotopy of the boundary of $\eeT$ extends to an isotopy of homogeneous 
symplectomorphisms.  Recall that $\mutghat^\infty$ is, by Definition \ref{def:mutghat}, 
the result of applying such as isotopy  to $\ghat^\infty$; recall that 
$\ghat$ and  $\mutghat$ as the defined respectively as the union of the zero section and the cones
over $\ghat^\infty$ and $\mutghat^\infty$.  Thus:  

\begin{corollary}
\label{cor:bytuene-mershe-ant-averil}
The defining flow $F_t$ induces an equivalence $\sh_{\ghat}(\R^2) \congto \sh_{\mutghat}(\R^2)$. 
\end{corollary}

\begin{proposition} \label{prop:hatsheafmutation}
There is a commutative diagram: 
$$
\begin{CD}
 \muloc(\ghat) @>\cong>> \muloc(\mutghat) 
\\
@VVV @VVV  \\
 \muloc(\partial \ghat) @>\cong>> \muloc(\partial \mutghat) 
\end{CD}
$$
\end{proposition}
\begin{proof}
There is nothing to check at the boundary at infinity in $\R^2$: since the contact isotopy 
was compactly supported, it does not change the sheaf categories here.  As for the existence
of the isomorphism at contact infinity, 
this is ultimately because contact transformations respect microlocalization.  More precisely,
it follows by checking the criteria of \cite[Thm 7.2.1]{KS} for the kernel constructed in \cite{GKS}
that the above morphism induces, away from the zero section, an isomorphism of Kashiwara-Schapira
sheaves. 
\end{proof}

\begin{remark}
Note however that the above isomorphism {\em does not} respect the ``Morse trivialization'' 
procedure of Lemma \ref{lem:mumon}, even when it is defined (i.e., when $C_k$ is the only
curve).  This is because at one point during the isotopy 
--- when $C_k$ is collapsed to a point --- the front projection of the Legendrian 
will fail to be an immersion.  In fact the result is that there is a cohomological shift in degree
with respect to this trivialization.  A proper discussion of this requires the ``inertia index'' 
(aka Maslov index) of \cite{KS}.  We do not carry this out here, because we can absorb the 
resulting ambiguity in a canonical way. 
\end{remark}

\subsection{Construction of the mutation functor}
We now use the local isomorphism $\muloc(\ghat) \cong \muloc(\mutghat)$ to build an equivalence $\muloc(\L) \cong \muloc(\L')$. 

Fix once more the data of a curve configuration $\cC$ on a surface $\cL$, with  $\L \subset W$ the associated Lagrangian skeleton and 4-manifold.  We let $\cC'$, $\cL'$, and $\L'$ denote their counterparts under mutation at 
a fixed embedded curve $C_k \in \cC$. 

Preserving the notation of Section \ref{sec:geometry}, we write $\cL = \cL_{(k)} \cup \cL^{(k)}$ for
an open cover consisting of a neighborhood $\cL_{(k)}$ of $C_k$ and a complementary open set. 
Similarly,  $\L = \L_{(k)} \cup \L^{(k)}$. 
We also somewhat abusively write $\partial \L_{(k)}$ for the intersection $\L_{(k)} \cap \L^{(k)} \subset \L$. 
This retracts onto the evident interpretation of $\partial \L_{(k)}$, which is two circles 
(the inside and outside translates of $C_k$ on $\cL$) connected by several 
segments (from other disks and half-disks whose boundaries meet $C_k$).

Recall however  ${\cL_{(k)}}'$  and $\cL^{(k)} {}'$
 {\em do not} mean a neighborhood of $C'_k$ and its complement, but instead for the parts of 
 $\cL'$ which are the images of $\cL_{(k)}$ and $\cL^{(k)}$ under the fixed identification 
 $\cL \cong \cL'$ used in defining the mutation of curve configurations.   In particular, recall that 
 the restricted curve configuration ${\cC_{(k)}}'$ will generally contain intersections amongst the curves
 ending on the boundary of ${\cL_{(k)}}'$, whereas a neighborhood of $C'_k$ would not.

The canonical identification $\L^{(k)} \cong {\L^{(k)}}'$ induces a canonical equivalence 
$\muloc(\L^{(k)}) \cong \muloc({\L^{(k)}}')$, compatible with restriction to $\partial \L_{(k)}
\cong \partial {\L_{(k)}}'$. 
Thus to construct $\Mut_k$ it suffices to construct an 
equivalence $\muloc(\L_{(k)}) \cong \muloc({\L_{(k)}}')$, respecting the restriction to the boundary. 

We already have such an equivalence for the conical models $\ghat$ and $\mutghat$.  
It remains to construct  equivalences $\muloc(\L_{(k)}) \congto \muloc(\ghat)$. 
One should expect such an equivalence, since one expects both these categories to model the Fukaya categories of $B^4$ whose objects have the same prescribed asymptotics in $S^3$.  
On the sheaf side, we construct the equivalence just by computing directly. 

\begin{theorem} \label{thm:equiv}
There is an equivalence $\Mut_k: \muloc(\L_{(k)}) \congto \muloc({\L_{(k)}}')$ such that the following diagram commutes up to isomorphism:
\[
\begin{CD} 
\muloc(\L_{(k)})  @>>>  \muloc( \partial \L_{(k)})  \\
@VVV @VVV\\
\muloc({\L_{(k)}}')  @>>> \muloc( \partial {\L_{(k)}'})
\end{CD}
\]
There is a global equivalence $\Mut_k: \muloc(\L) \congto \muloc(\L')$ that is intertwined with this local equivalence by the restriction functors $\muloc(\L) \to \muloc(\L_{(k)})$, $\muloc(\L') \to \muloc({\L_{(k)}}')$. 
\end{theorem}
\begin{proof}

We extract this local equivalence, together with the commutative diagram, from a larger diagram passing through the conical models $\ghat$, $\mutghat$:
\begin{equation}\label{eq:localmut}
\begin{CD}
\muloc(\L_{(k)}) @>\cong>> \muloc(\ghat) @>\cong>> \muloc(\mutghat) @>\cong>>
\muloc({\L_{(k)}}') 
\\
@VVV @VVV @VVV @VVV  \\
\muloc(\partial \L_{(k)}) @>\cong>> \muloc(\partial \ghat) @>\cong>> \muloc(\partial \mutghat) @>\cong>>
\muloc( {\partial \L_{(k)}}') 
\end{CD}
\end{equation}

While all the sheaves of categories have been indifferently denoted $\muloc$, the above identifications
are by no means tautological: each category is defined in terms of the specified conical Lagrangian.  

We already saw the central square in Proposition \ref{prop:hatsheafmutation}. 

To see the leftmost commutative square, let $D$ be the open unit disk and $D^\circ$ an open disk whose closure is contained in $D$. Recall that the front projection of $\ghat^\infty$ contains the boundary of $D$ but does not meet $D$ itself. 
By the sheaf axiom, we have
$$\muloc(\ghat) = \muloc(\ghat \setminus D^\circ) \times_{\loc(D\smallsetminus D^\circ)} \loc(D).$$ 

On the other hand, by definition
$$\muloc(\L_{(k)}) = \muloc(\TCk) \times_{\loc(D_k \smallsetminus D^\circ_k)}
\loc(D_k),$$
where $D^\circ_k$ is the core of the handle attachment above $C_k$.

We can thus specify an equivalence $\muloc(\ghat)\cong \muloc(\L_{(k)})$ by providing vertical equivalences in the following diagram, together with natural isomorphisms of functors making the left and right squares commute:
\begin{equation}
\begin{CD}
\muloc(\ghat \setminus D^\circ) @>>>  \loc(D\smallsetminus D^\circ) @<<< \loc(D)  \\
@VVV @VVV @VVV \\
\muloc(\TCk) @>>> \loc(D_k \smallsetminus D^\circ_k)  @<<< \loc(D_k) 
\end{CD}
\end{equation}

The equivalences on the middle and right, and the corresponding commuting isomorphism, are the natural ones from the homeomorphism of $D_k$ with the unit disk (recall we fixed before a homeomorphism $\ghat \cong \L_{(k)}$).  The equivalence on the left locally relates different presentations of the microlocal sheaf categories of the $A_2$ and $A_3$ arboreal singularities \cite{N3}. That is, $\TCk$ can be covered by radial sectors that either only meet $C_k$ or that meet $C_k$ and one other curve. Following Examples \ref{ex:A2} and \ref{ex:A3} the restriction of $\muloc$ to such a sector is equivalent to $\coeffs A_2\dmod$ or $\coeffs A_3\dmod$, respectively. A similar statement holds for $\ghat \setminus D^\circ$, and the equivalence $\muloc(\TCk) \cong \muloc(\ghat \setminus D^\circ)$ is locally given by reflection functors.

That these reflection functors glue together into a global equivalence follows from the absence of monodromy of the KS sheaves $\muloc|_{\TCk}$ and $\muloc|_{\ghat \setminus D^\circ}$. 
The commuting isomorphism can then be described locally in terms of quivers: in the $A_2$ case, for example, there is a natural isomorphism between the stalk of representation at a vertex before applying a reflection and the cone over the defining map of the representation after applying it. 
The rightmost square in Equation \ref{eq:localmut} is obtained similarly; note however that the Morse trivialization of $\muloc(\mutghat^\infty)$ differs from the natural trivializations of $\muloc(\ghat^\infty)$ and $\partial {\L_{(k)}}'$ by a cohomological degree shift. 

With the local equivalence in hand, we define a global equivalence $\muloc(\L) \congto \muloc(\L')$ by providing the vertical equivalences in the diagram below, along with isomorphisms making the left and right squares commute:
\begin{equation}\label{eq:localequiv}
\begin{CD}
\muloc(\L^{(k)}) @>>> \muloc( \partial \L_{(k)}) @<<< \muloc(\L_{(k)})  \\
@VVV @VVV @VVV \\
\muloc({\L^{(k)}}') @>>> \muloc( \partial {\L_{(k)}}')  @<<< \muloc({\L_{(k)}}') 
\end{CD}
\end{equation}
By construction $\muloc(\L^{(k)})$, $\muloc(\partial \L_{(k)})$ are canonically equivalent to their primed counterparts, in a way intertwining the restriction maps between them up to a canonical isomorphism. But we have just finished constructing a commuting square on the left. 
\end{proof}

\begin{remark}
Commutativity of a diagram of functors between categories means appropriate choices
of natural transformations at the squares.  We have given such above.  Note however that the entire
construction can be twisted by an autoequivalence of the identity of $\muloc(\partial {\L_{(k)}}')$. 
\end{remark}

\begin{remark}
The choice of $\sigma: \cC \to \Z/2\Z$ made in defining $\muloc$ is immaterial in the preceding theorem. While different choices may change the global category $\muloc(\L)$, the restriction of $\muloc$ to $\L_{(k)}$ is always the same
by Proposition \ref{prop:sigmambiguity}, 
since $\L_{(k)}$ is contractible. 
\end{remark}

\subsection{A disk glued to a cylinder}\label{sec:L0} 
In this section, we study in detail the category of microlocal sheaves on the skeleton 
$\L_0$  obtained from Definition \ref{def:W} by taking $\cL$ to be an annulus and $\cC$ to consist of a single noncontractible curve $C$. 

Let $A_2$ be the quiver $\bullet \to \bullet$.  We write
\begin{eqnarray*}
\mathfrak{c}: \coeffs A_2 \dmod & \to & \coeffs\dmod \\
A\to B & \mapsto & \mathrm{Cone}(A \to B) 
\end{eqnarray*}
for the functor that maps a representation the cone over its defining map. 

Recall that an isomorphism in a dg category is, by definition, a closed degree zero map, invertible in the homotopy
category.   Note we can describe $\loc(S^1)$ as  pairs $X \in \coeffs \dmod$ and an isomorphism $m: X \to X$. 
In general, we denote a pair consisting of an object $X$ and an isomorphism $m: X \to X$ as $X \lcirc m$.

\begin{proposition} \label{prop:examplecalc}
Up to equivalence, the objects of $\muloc(\L_0)$ are tuples $\{(X,m,y)\}$, where 
\begin{itemize}
\item $X \in A_2\dmod$
\item $m: X \to X$ is an isomorphism
\item $y: \mathfrak{c}(X) \to \mathfrak{c}(X)$ is a degree -1 map;  $dy = 1-(-1)^k \mathfrak{c}(m)$ on degree $k$ elements of $\mathfrak{c}(X)$. 
\end{itemize}
\end{proposition}

\begin{proof}
The starting point for the calculation is the defining decomposition
$\L_0 = T^+_C\cL \cup  D$
into local conical models.    By definition $\mu loc(\L_0)$ is the homotopy pullback
\[
\begin{tikzpicture}
\node (tl) at (0,0) {$\mu loc(\L_0)$}; \node (tr) at (4.5,0) {$\loc(D)$};
\node (bl) at (0,-1.5) {$\mu loc(T^+_C\cL)$}; \node (br) at (4.5,-1.5) {$\mu loc( T^+_C\cL \cap  D)$};
\draw[->] (tl) to (tr); \draw[->] (tl) to (bl); \draw[->] (bl) to (br); \draw[->] (tr) to (br);
\end{tikzpicture}
\]

Let us first describe combinatorially the spaces whose homotopy fiber product we plan to take. Lemma 
\ref{lem:mumon} 
identifies $\mu loc( T^+_C\cL \cap  D) \cong \loc( T^+_C\cL \cap  D) \cong \loc(S^1)$.  Identifying $ T^+_C\cL \cap  D$ with $\partial D$ defines the right hand map as the restriction of a local system on $D$ to its boundary.  

As $T^+_C\cL$ is a circle times an $A_2$ arboreal singularity, we can describe $\mu loc(T^+_C\cL)$ in terms of quiver representations \cite{N3}.  An object of $\mu loc(T^+_C\cL)$ is an object $X$ of $A_2\dmod$ equipped with an
isomorphism. 

In particular $\mathfrak{c}:A_2\dmod \to \coeffs\dmod$ extends to a functor $\mu loc(T^+_C\cL) \to \loc( S^1)$ in a straightforward way. 

In general, a homotopy fiber product  $\alpha \stackbin[\beta]{h}{\times} \gamma$ in a model category can be computed by replacing $\beta$ with a
path space and taking an ordinary pullback.  In Tabuada's model structure for dg categories \cite{Tab}, 
the path space $\cP(T)$ of a dg-category $T$ is a dg-category whose objects are isomorphisms in $T$ that are invertible up to homotopy, a morphism between two such maps being another pair of maps intertwining them up to a chosen homotopy.  More precisely,
\[
\Hom_{\cP(T)}(X\xrightarrow{f} Y, W \xrightarrow{g} Z) = \Hom_T(X,W) \oplus \Hom_T(Y,Z) \oplus \Hom_T(X,Z)[-1].
\]
Thus a degree $k$ map from $f$ to $g$ is given by a triple $m_1 \in \Hom^k_T(X,W)$, $m_2 \in \Hom^{k}_T(Y,Z)$, $h \in \Hom^{k-1}_T(X,Z)$.  The differential $d_{\cP(T)}$ is defined by 
\[
d_{\cP(T)} \begin{pmatrix} m_1 & 0 \\ h & m_2 \end{pmatrix} = \begin{pmatrix} d_Tm_1 & 0 \\ d_Th + g \circ m_1 - (-1)^k m_2 \circ f & d_tm_2 \end{pmatrix}.
\]

Thus we replace $\loc(S^1)$ with its path space, and
compute $\mu loc(\L_0)$ as the ordinary limit of the diagram
\begin{equation}\label{eq:pathlimit}
\begin{tikzpicture}
\node (bl) at (1,-1) {$\loc(S^1)$}; \node (br) at (3,-1) {$\loc(S^1)$};
\node (tl) at (0,0) {$\mu loc(T^+_C\cL)$}; \node (tr) at (4,0) {$\loc(D)$};
\node (tm) at (2,0) {$\cP(\loc(S^1))$};
\draw[->] (tl) to (bl); \draw[->] (tm) to (bl); \draw[->] (tm) to (br); \draw[->] (tr) to (br);
\end{tikzpicture}
\end{equation}
Thus an object of $\mu loc(\L_0)$ is an object of $\mu loc(T^+_C\cL)$, which we represent as a pair 
$X \lcirc m$;  an object $Y$ of $\loc(D) \cong \coeffs\dmod$ for $X \in A_2 \dmod$, and an isomorphism 
$$y: \mathfrak{c}(X \lcirc m) \to Y \lcirc 1 \in \loc(S^1)$$

The morphisms in $\mu loc(\L_0)$ can similarly be computed from the above diagram and the description of 
morphisms in $\cP(\loc(S^1))$.

We can simplify this description of $\mu loc(\L_0)$ further. 
An isomorphism $f: (M \lcirc m) \to (N \lcirc n)$  in $\loc(S^1)$ can be expressed in terms of $\coeffs$-module maps: 
a degree zero morphism 
$f_0: M \to N$ and a degree -1 morphism $f_{-1}: M \to N$ such that $df_{-1} = f_0 m \pm n f_0$. 

In particular, the above map $y: \mathfrak{c}(X \lcirc m) \to Y \lcirc 1$ decomposes as a pair $y_0$, $y_{-1}$ of degree zero and degree -1 maps from $\mathfrak{c}(X)$ to $Y$ in $\coeffs\dmod$.  The above stated condition that $y$ is an
isomorphism amounts to $dy_{-1} = y_0 \circ (1 - (-1)^k C(m))$ on the degree $k$ part of $C(X)$.  

But we may consider a full subcategory where $Y$ is $C(X)$ itself and $y_0$ is the identity.  The inclusion of this subcategory is  an equivalence, leading to the description in the claim above.
\end{proof}

Recall that we write $\muLoc(\L)$ for the full subcategory of $\muloc(\L)$ whose objects, when restricted to 
$\sh_{\TC}(\cL)$, have cohomology concentrated in degree zero. 

\begin{figure}
\includegraphics{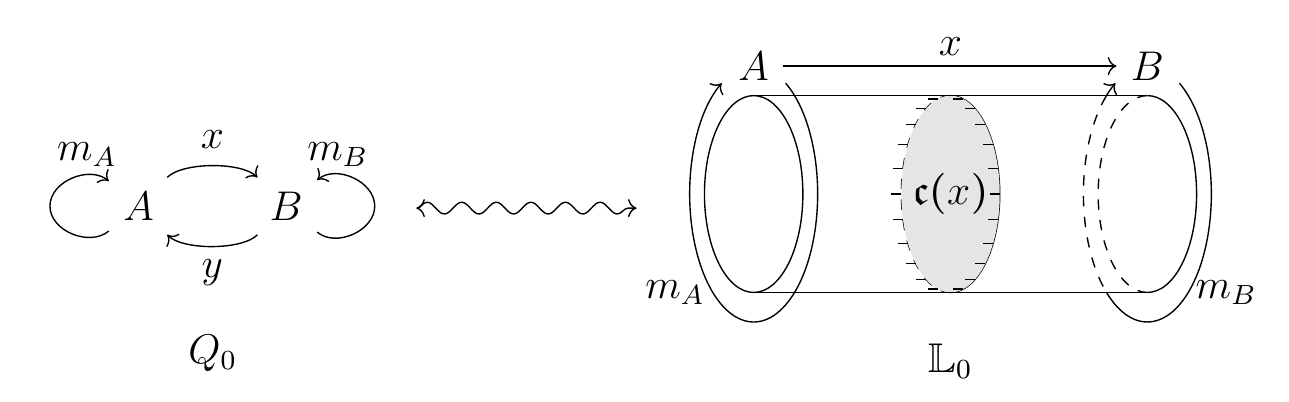}
\caption{The correspondence between representations of $Q_{0}$ and objects of $\muLoc(\L_0)$. $A$ and $B$ are identified with the stalks at either end of $\L_0$, and the cone $\mathfrak{c}(x)$ is the stalk on the disk. The map $x$ is a generazation map in $\sh_{\TC}(\cL)$ while $y$ trivializes the monodromy on $\mathfrak{c}(x)$.}
\label{fig:L0}
\end{figure}

\begin{proposition}
Let $Q_0$ denote the quiver pictured in Figure \ref{fig:L0}.   The category $\muLoc(\L_0)$ is equivalent to the full subcategory of $\coeffs Q_0\dmod$ satisfying the relations $m_B= 1-xy$, $m_A = 1-yx$, where $m_A$ and $m_B$ are invertible, and where $A$ and $B$ have cohomology concentrated in degree zero.
\end{proposition}
\begin{proof}
The characterization of $\muLoc(\L_0)$ as a full subcategory of $\coeffs Q_0\dmod$ follows from the fact that if $X \in \coeffs A_2\dmod$ has cohomology concentrated in degree zero, then the degree -1 map $y$ on the cone over $A \to B$ is a degree zero map from $B \to A$. 
\end{proof}

\begin{remark} This algebraic description 
has historically been associated to something slightly different, namely perverse sheaves on $\R^2$
constructible with respect to the stratification $0, \R^2 \setminus 0$ \cite{Bei, GGM}. 
The relation to the present
context is that such constructible sheaves are the same as those microsupported in the union
$\bigplus$ 
of the zero section and the cotangent fiber at zero; this conical Lagrangian appears at the midpoint
of the movie $\ghat \rightsquigarrow \mutghat$, hence by \cite{GKS} determines the same category of 
sheaves.  The preceding observation gives another proof of Proposition \ref{prop:examplecalc}. 
\end{remark}

\begin{example}
The trivial local system on $\L_0$ is given by the quiver 

\begin{center}
\begin{tikzpicture}
\newcommand*{\len}{1.5};
\node (A) at (0,0) {$\coeffs$}; \node (B) at (\len,0) {$\coeffs$};
\draw (A) .. controls (\len*.3,\len*.3) and (\len*.7,\len*.3) .. (B) [->,shorten <=0mm,shorten >=0mm];
\draw (A) .. controls (\len*.3,-\len*.3) and (\len*.7,-\len*.3) .. (B) [<-];
\draw (A) to[out=220,in=-90] ($(A)+(-\len*.6,0)$) to[out=90,in=140] (A) [->];
\draw (B) to[out=-40,in=-90] ($(B)+(\len*.6,0)$) to[out=90,in=40] (B) [->];
\node at ($(A)+(-\len*.35,\len*.35)$) {$1$};
\node at ($(B)+(\len*.35,\len*.35)$) {$1$};
\node at (\len*.5,\len*.45) {$1$};
\node at (\len*.5,-\len*.45) {$0$};
\end{tikzpicture}
\end{center}
\end{example}

We now assume $\coeffs$ is an algebraically closed field, and recall 
some classical facts about the representation theory of this quiver. 

\begin{lemma} \label{lem:simples}
The full subcategory of $\coeffs Q_0 \dmod$ on objects such that $m_A = 1-xy$ and $m_B = 1-yx$
and both of these are invertible is generated by the objects $S^1_A$, $S^1_B$, and $P^m_A$ of Figure \ref{fig:quiverreps}. Up to a shift they are the only simple objects in the category.
\end{lemma}
\begin{proof}
We reproduce the classical argument. 
It suffices to show that any representation has one of these as a subrepresentation.  If $yx \ne 0$, then 
choose an eigenspace $V$; we can do this since $\coeffs$ is an algebraically closed field.  
Evidently $V$ is also an eigenspace of $m_A = 1-yx$.  If $xV = 0$, then $(V, 0)$ determines
a subrepresentation on which  $m_A$ acts as the identity; it is thus a sum of
$S^1_A$'s.  

Otherwise observe $(xy)xV = x(yx)V \subset xV$.  Choose
an eigenspace $W \subset xV$ for $xy$.  If $yW = 0$, then $(0, W)$ determines a subrepresentation
on which $m_B$ acts as the identity; it is a sum of $S^1_B$'s.  

If not, then observe $xy W \subset W$, hence $(yW, W)$ is a subrepresentation, on which $yx$ and $xy$,
hence $m_A$ and $m_B$, act as scalars.  We have $y (1-m_B) W = yxyW = (1-m_A) yW$, so 
$m_A = m_B:= m$.  By assumption moreover $xy$ and $yx$ were not zero, so $m \ne 1$ and thus
$xy$ and $yx$ are nonzero scalars, hence $x$ and $y$, are invertible.  Finally, by an appropriate choice of basis, 
we can demand that $x$ is represented by an identity matrix, and $y$ by the scalar $1-m$. 
\end{proof}

\subsection{Mutation of local systems}\label{sec:localsystems} 

We want to understand the comparison 
$$loc(\cL) \subset \muloc(\L) \cong \muloc(\L') \supset loc(\cL')$$ 
induced by $\Mut_k$. 
As we have argued, because the inclusion $loc(\cL) \subset \muloc(\L)$ is fully faithful, it induces an open inclusion
of moduli spaces.  To understand an open inclusion, it suffices to understand the geometric points. 
With this in mind, we assume in this section that $\coeffs$ is an algebraically closed field. 

It suffices to work with a subcategory containing both $loc(\cL)$ and its image under mutation at $k$.  
Writing temporarily $\widetilde{\L}$ for the skeleton built from the curve collection consisting only of $C_k$, 
i.e. as a topological space $\widetilde{\L} = \cL \cup D_k$, observe that, by the local nature of the construction
of the mutation functor, the inclusion $\muloc(\widetilde{\L}) \subset \muloc(\L)$ of Proposition \ref{prop:inclusion}
intertwines mutation functors.  
In particular, the image of $loc(\cL) \subset \muloc(\widetilde{\L}) \subset  \muloc(\L)$ remains, after mutation,
inside $\muloc(\widetilde{\L})$.  It thus suffices to compute inside $\muloc(\widetilde{\L})$.  
In short, we may assume without loss of generality that the curve collection only had a single curve to begin with.
Henceforth we do this, and hence cease to distinguish between $\widetilde{\L}$ and $\L$. 

Note that our assumption that there was only a single curve amounts to $\L^{(k)} = \cL^{(k)}$. By the sheaf axiom
$
\muloc(\L) = \muloc(\L_{(k)}) \stackbin[\muloc(\partial \L_{(k)})]{h}{\times} loc(\cL^{(k)})
$. 
Thus  $\Mut_k: \muloc(\L) \cong \muloc(\L')$  is determined by
the following commutative diagram, as in Equation \ref{eq:localequiv}:

\begin{equation}\label{eq:localequivsubk}
\begin{CD}
\muloc(\cL^{(k)}) @>>> \muloc( \partial \L_{(k)}) @<<< \muloc(\L_{(k)})  \\
@VVV @VVV @VVV \\
\muloc({\cL^{(k)}}') @>>> \muloc( \partial {\L_{(k)}}')  @<<< \muloc({\L_{(k)}}') 
\end{CD}
\end{equation}

Let $\L_0$ be the cylinder-with-disk of Section \ref{sec:L0}.  
Fix a homeomorphism $\L_0 \cong \cL_{(k)} \cup D_k \subset \L$, compatible with the co-orientation of $C_k$. By construction this induces $\muloc(\L_{(k)}) \cong \muloc(\L_0)$. We fix a similar homeomorphism $\L_0 \cong {\cL_{(k)}}' \cup D'_k \subset \L'$ and equivalence $\muloc({\L_{(k)}}') \cong \muloc(\L_0)$. Note, however, that the homeomorphism $\cL_{(k)} \cong {\cL_{(k)}}'$ arising from the comparison $\cL_{(k)} \subset \L_0 \supset {\cL_{(k)}}'$ is \emph{not} the one corresponding to the homeomorphism $\cL \cong \cL'$ which we used to define the mutated curve configuration. 
Instead, since the co-orientations of $C_k$, $C'_k$ are opposite, the composition of these is an 
automorphism of $\cL_{(k)}$ that exchanges the two components of $\partial \cL_{(k)}$.

\begin{figure}
\centering
\begin{tikzpicture}
\newcommand*{\len}{1.5};
\node [matrix] (SA) at (0,0) {
\node (A) at (0,0) {$\coeffs$}; \node (B) at (\len,0) {$0$};
\draw (A) .. controls (\len*.3,\len*.3) and (\len*.7,\len*.3) .. (B) [->,shorten <=0mm,shorten >=0mm];
\draw (A) .. controls (\len*.3,-\len*.3) and (\len*.7,-\len*.3) .. (B) [<-];
\draw (A) to[out=220,in=-90] ($(A)+(-\len*.6,0)$) to[out=90,in=140] (A) [->];
\draw (B) to[out=-40,in=-90] ($(B)+(\len*.6,0)$) to[out=90,in=40] (B) [->];
\node at ($(A)+(-\len*.35,\len*.35)$) {$1$};
\node at (\len*.5,\len*.45) {$$};
\node at (\len*.5,-\len*.45) {$$};
\node at (-1.7,0) {$S_A^1$};
\node at (-1.2,0) {$=$};
\\
};
\node [matrix] (SB) at (0,-2) {
\node (A) at (0,0) {$0$}; \node (B) at (\len,0) {$\coeffs$};
\draw (A) .. controls (\len*.3,\len*.3) and (\len*.7,\len*.3) .. (B) [->,shorten <=0mm,shorten >=0mm];
\draw (A) .. controls (\len*.3,-\len*.3) and (\len*.7,-\len*.3) .. (B) [<-];
\draw (A) to[out=220,in=-90] ($(A)+(-\len*.6,0)$) to[out=90,in=140] (A) [->];
\draw (B) to[out=-40,in=-90] ($(B)+(\len*.6,0)$) to[out=90,in=40] (B) [->];
\node at ($(B)+(\len*.35,\len*.35)$) {$1$};
\node at (\len*.5,\len*.45) {$$};
\node at (\len*.5,-\len*.45) {$$};
\node at (-1.7,0) {$S_B^1$};
\node at (-1.2,0) {$=$};
\\
};
\node [matrix] (PA) at (5.2,0) {
\node (A) at (0,0) {$\coeffs$}; \node (B) at (\len,0) {$\coeffs$};
\draw (A) .. controls (\len*.3,\len*.3) and (\len*.7,\len*.3) .. (B) [->,shorten <=0mm,shorten >=0mm];
\draw (A) .. controls (\len*.3,-\len*.3) and (\len*.7,-\len*.3) .. (B) [<-];
\draw (A) to[out=220,in=-90] ($(A)+(-\len*.6,0)$) to[out=90,in=140] (A) [->];
\draw (B) to[out=-40,in=-90] ($(B)+(\len*.6,0)$) to[out=90,in=40] (B) [->];
\node at ($(A)+(-\len*.35,\len*.35)$) {$m$};
\node at ($(B)+(\len*.35,\len*.35)$) {$m$};
\node at (\len*.5,\len*.45) {$1$};
\node at (\len*.5,-\len*.45) {$1-m$};
\node at (-1.7,0) {$P_A^m$};
\node at (-1.2,0) {$=$};
\\
};
\node [matrix] (PB) at (5.2,-2) {
\node (A) at (0,0) {$\coeffs$}; \node (B) at (\len,0) {$\coeffs$};
\draw (A) .. controls (\len*.3,\len*.3) and (\len*.7,\len*.3) .. (B) [->,shorten <=0mm,shorten >=0mm];
\draw (A) .. controls (\len*.3,-\len*.3) and (\len*.7,-\len*.3) .. (B) [<-];
\draw (A) to[out=220,in=-90] ($(A)+(-\len*.6,0)$) to[out=90,in=140] (A) [->];
\draw (B) to[out=-40,in=-90] ($(B)+(\len*.6,0)$) to[out=90,in=40] (B) [->];
\node at ($(A)+(-\len*.35,\len*.35)$) {$m$};
\node at ($(B)+(\len*.35,\len*.35)$) {$m$};
\node at (\len*.5,\len*.45) {$1-m$};
\node at (\len*.5,-\len*.45) {$1$};
\node at (-1.7,0) {$P_B^m$};
\node at (-1.2,0) {$=$};
\\
};
\node [matrix] (PiA) at (10.4,0) {
\node (A) at (0,0) {$\coeffs \Z$}; \node (B) at (\len,0) {$\coeffs \Z$};
\draw (A) .. controls (\len*.3,\len*.3) and (\len*.7,\len*.3) .. (B) [->,shorten <=0mm,shorten >=0mm];
\draw (A) .. controls (\len*.3,-\len*.3) and (\len*.7,-\len*.3) .. (B) [<-];
\draw (A) to[out=220,in=-90] ($(A)+(-\len*.6,0)$) to[out=90,in=140] (A) [->];
\draw (B) to[out=-40,in=-90] ($(B)+(\len*.6,0)$) to[out=90,in=40] (B) [->];
\node at ($(A)+(-\len*.35,\len*.35)$) {$t$};
\node at ($(B)+(\len*.35,\len*.35)$) {$t$};
\node at (\len*.5,\len*.45) {$1$};
\node at (\len*.5,-\len*.45) {$1-t$};
\node at (-1.7,0) {$\Pi_A$};
\node at (-1.2,0) {$=$};
\\
};
\node [matrix] (PiB) at (10.4,-2) {
\node (A) at (0,0) {$\coeffs \Z$}; \node (B) at (\len,0) {$\coeffs \Z$};
\draw (A) .. controls (\len*.3,\len*.3) and (\len*.7,\len*.3) .. (B) [->,shorten <=0mm,shorten >=0mm];
\draw (A) .. controls (\len*.3,-\len*.3) and (\len*.7,-\len*.3) .. (B) [<-];
\draw (A) to[out=220,in=-90] ($(A)+(-\len*.6,0)$) to[out=90,in=140] (A) [->];
\draw (B) to[out=-40,in=-90] ($(B)+(\len*.6,0)$) to[out=90,in=40] (B) [->];
\node at ($(A)+(-\len*.35,\len*.35)$) {$t$};
\node at ($(B)+(\len*.35,\len*.35)$) {$t$};
\node at (\len*.5,\len*.45) {$1-t$};
\node at (\len*.5,-\len*.45) {$1$};
\node at (-1.7,0) {$\Pi_B$};
\node at (-1.2,0) {$=$};
\\
};
\end{tikzpicture}
\caption{Notation for various objects of $\muloc(\L_0) \subset \coeffs Q_0\dmod$. The objects $P^m_A$ and $P^m_B$ are isomorphic for $m\neq 1$.}\label{fig:quiverreps}
\end{figure}

We now leverage our algebraic understanding of $\muloc(\L_0)$ from  Section \ref{sec:L0}. 
The following proposition is a close analogue of \cite[Prop 4.5]{BezKap}, 
which derives the same formula for the action of the Fourier transform on perverse sheaves. 

\begin{proposition}\label{prop:Mutcharacterization} 
The composition $\muloc(\L_0) \cong \muloc(\L_{(k)}) \cong \muloc({\L_{(k)}}') \cong \muloc(\L_0)$ 
interchanges $S_A^1$ and $S_B^1$, and interchanges $P^m_A \to P^{1/m}_B$. 
\end{proposition}

\begin{proof}
This equivalence must permute simples (characterized in Lemma \ref{lem:simples}); these in turn are characterized by $A \lcirc m_A$ and $B \lcirc m_B$.  But
these monodromies at the boundary are determined by the above discussion of how $\cL_{(k)}$ and ${{\cL}_{(k)}}'$ include into
$\L_0$. 
\end{proof}

\begin{proposition}
Assume $\cE \in loc(\cL) \subset \muloc(\L)$.  Then the mutation
$\cE':=  \Mut_k(\cE) \in \muloc(\L')$ will in fact lie in $loc(\cL')$ if and only if the monodromy around $C_k$
does have 1 as an eigenvalue. 
\end{proposition}
\begin{proof}
It's enough to check the local case, for $\cL_{(k)}$.  An object lies in $loc(\cL_{(k)})$ 
if and only if it's in the subcategory spanned by the $P^m_A$.  Under mutation, these are sent to 
the $P^{1/m}_B$, which are isomorphic to $P^{1/m}_A$, except when $m=1$.  
$P^1_B$ is not isomorphic to any of the $P^m_A$. 
\end{proof}

Recall that we write $Loc(\cL)$ for local systems in the ordinary sense, i.e., those which have cohomology sheaves
only in degree zero. 
To describe the action of $\Mut_k$ on $\Loc(\cL) \subset \muloc(\L)$, we write an object of $\Loc(\cL)$ as a representation of the fundamental groupoid of $\cL$ on a fixed $\coeffs$-module. That is, it is a $\coeffs$-module $\cE$ together with an endomorphism $\cE_\gamma$ for every path $\gamma: [0,1] \to \cL$, compatible with concatenation of paths. We write $\cE_{C_k}$ for the holonomy around $C_k$, oriented so that its co-orientation points rightward.

\begin{theorem}\label{prop:clustertrans}
There exists an identification of $\cL$ with $\cL'$ such that, whenever $\cE$ and $\cE' = Mut_k(\cE)$ are both local systems, 
\begin{itemize}
\item If $\gamma$ does not meet $C'_k$ then $\cE'_\gamma = \cE_\gamma$.
\item Suppose $\gamma$ crosses $C'_k$ exactly once, and that its tangent at the crossing pairs negatively with the co-orientation of $C'_k$. Let $\gamma_{<C'_k}$ be the subpath that starts at $\gamma(0)$ and ends at the crossing, $\gamma_{> C'_k}$ the subpath which starts at the crossing and ends at $\gamma(1)$. Then
\[
\cE'_\gamma = \cE_{\gamma_{>C'_k}}(\Id - \:\cE_{C_k})\cE_{\gamma_{< C'_k}}.
\]
\end{itemize}
The identification $\cL \cong \cL'$ here differs from the identification used to define $\cC'$ by some fixed universal number of Dehn twists
about $C_k$. 
\end{theorem}
\begin{proof}
The first bullet point is obvious: $\cL$ and $\cL'$ are canonically identified away from $C_k$. 

Let $\gamma$ be a path as in the second bullet point -- without loss of generality we take $\gamma(0)$ and $\gamma(1)$ to lie in opposite components of ${\cL_{(k)}}'$. Using the canonical homeomorphism $\cL^{(k)} \cong {\cL^{(k)}}'$ (which gave the isomorphism making the left square in Equation \ref{eq:localequivsubk} commute), we obtain an isomorphism between the stalks of $\cE$ and $\cE'$ at these points.

In Proposition \ref{prop:Mutcharacterization} we determined $\cE'|_{\L_{(k)}}$ up to isomorphism. 
If $\cE|_{\partial \L_(k)} \cong P^m_A$, then $\cE'_\gamma = a_\cE \cE_\gamma$ for some $a_\cE \in \coeffs^*$ since both are isomorphisms of 1-dimensional spaces. 
By functoriality, $a_\cE$ only depends on $\cE|_{\L_{(k)}}$ up to isomorphism, hence only on $m$; with this in mind we write $a_m$ rather than $a_\cE$.

On the other hand, suppose $\cE|_{\partial \L_{(k)}} \cong \Pi_A$ -- thus in passing we allow $\cE$ to lie in the category $\muloc^\infty(\L) \supset \muloc(\L)$ of Remark \ref{rmk:infrank} (which embeds now into the category of not necessarily finite-dimensional $Q_0$-representations). Then $\cE'|_{\partial \L_{(k)}} \cong \Pi_B$, since $\Pi_A$ and $\Pi_B$ are the projective covers of $S^1_A$ and $S^1_B$, which are exchanged by $\Mut_k$. 

Fix a particular isomorphism $\cE|_{\partial \L_{(k)}} \cong \Pi_A$ to identifys the stalks of $\cE$ at  $\gamma(0)$, $\gamma(1)$ with $\coeffs\Z$, and $\cE_\gamma$ with the identity. Independently, fix a particular isomorphism $\cE'|_{\partial \L_{(k)}} \cong \Pi_B$, identifying the stalks of $\cE'$ at $\gamma(0)$, $\gamma(1)$ with $\coeffs\Z$, and $\cE'_\gamma$ with multiplication by $1-t$. 

As noted above, the homeomorphism $\cL^{(k)} \cong {\cL^{(k)}}'$ identifies the stalks of $\cE$ and $\cE'$ at $\gamma(0)$ with each other, and moreover does so in a way that intertwines the monodromy around $\partial \L_{(k)}$. With respect to the above trivializations of each by $\coeffs\Z$, this identification is multiplication by a unit in $\coeffs \Z$ (i.e. a monomial). The same holds for the trivializations at $\gamma(1)$.  

As these units commute with multiplication by $1-t$, we combine them and conclude that with respect to the trivialization of the stalks of $\cE$, we have $\cE'_\gamma = at^n(1-t) \cE_\gamma$ for some $a \in \coeffs^*$ and some $n \in \Z$. Again, by functoriality these numbers depend only on $\cE$ up to isomorphism. More importantly, we have a morphism $\Pi_A \to P^m_A$ for all $m$. Now functoriality tells us that $a_m = am^n(1-m)$, independently of $m$.

Note that for a given $\L$ it is possible there are no objects $\cE$ whose restriction is isomorphic to $\Pi_A$. But the calculation is entirely about a gluing that takes place in a neighborhood of $\L_{(k)}$ hence we may take $\L$ to small enough to have such $\cE$.

The ambiguity left in $\Mut_k$ comes from the choice of natural transformation making the right square in Equation \ref{eq:localequivsubk} commute, and the fact that the homeomorphism $\cL \cong \cL'$ is only canonical up to Dehn twisting around $C_k$. There is an action of automorphisms of the identity functor of $\muloc(\partial \L_{(k)})$ on the choice of natural transformation -- with this we can fix $a$ to be $1$ (note that a priori $a = \pm 1$ since everything is defined over $\Z$, even if we have assumed $\coeffs$ to be a field in our analysis). Changing the homeomorphism $\cL \cong \cL'$ by a Dehn twist changes $n$, so likewise we may take $n = 1$. 

The reduction from the general case to the previous calculation follows by decomposing into eigenspaces of $\cE_{C_k}$. If it is semisimple there is nothing to check, otherwise functoriality tells us the calculation is compatible with extensions, giving us the case when there are nontrivial Jordan blocks.
\end{proof}

\begin{remark} 
It is possible to determine, rather than absorb, the ambiguities in the above proof,
by a microlocal calculation involving the GKS kernel. 
\end{remark}

\section{Relation to Cluster Theory} \label{sec:cluster}

The language of cluster algebra \cite{FZ} provides a natural setting in which to organize the results of the preceding sections. We begin by reviewing the general theory with certain extensions necessitated by the scope of our discussion. These are related to the implicit sign choices in the definition of the Kashiwara-Schapira sheaf, as well as the fact that a general curve configuration is not nondegenerate. We also discuss certain noncommutative analogues of cluster structures possessed by spaces of rank $n > 1$ Lagrangian branes.

We begin with the notion of a seed, the defining data of a cluster structure. Our notation follows that of \cite{FG,GHK}. 

\begin{definition}
A \textbf{seed} $s = (N,\{e_i\})$ is a lattice $N$ with skew-symmetric integral form $\{,\}$ and a finite collection $\{e_i\}_{i \in I} \subset N$ of distinct primitive elements indexed by a set $I$. 
\end{definition}

We write $[a]_+$ for $\max(a, 0)$. 

\begin{definition}
The mutation of $s$ at $k \in I$ is the seed $\mu_ks = (N,\{\mu_ke_i\})$, where 
\begin{equation}\label{eq:seedmutation}
\mu_ke_i = \begin{cases} e_i + [\{e_i,e_k\}]_+ e_k & i \neq k \\ -e_k & i = k. \end{cases}
\end{equation}
\end{definition}

To a seed we associate a quiver without oriented 2-cycles and with vertex set $\{v_i\}_{i \in I}$.  The number of arrows from $v_i$ to $v_j$ is $[\{e_i,e_j\}]_+$, and if the $e_i$ are a basis the seed is determined up to isomorphism by the quiver. 
Conversely, given such a quiver $Q$ we have a seed given by taking the vector space whose standard basis
vectors are enumerated by the vertices, and whose skew-symmetric form is given by the arrows.  In the literature one often only considers seeds of this form. One can also consider seeds related to skew-symmetrizable matrices, but these do not arise in our setting. We also suppress any explicit discussion of ``frozen indices''; this notion
is already included by 
allowing the $e_i$ to fail to generate $N$.

Given a seed $s = (N,\{e_i\})$, we write  $M = \Hom(N,\Z)$ and consider the dual algebraic tori
\[
\cX_s = \Spec \Z N, \quad \cA_s = \Spec \Z M,
\]
We let $z^n \in \Z N$ denote the monomial associated to $n \in N$, likewise $z^m \in \Z M$ for $m \in M$. 

\begin{definition}
For $k \in I$, the \textbf{cluster $\cX$- and $\cA$-transformations} $\mu_k: \cX_s \dasharrow \cX_{\mu_k s}$, $\mu_k: \cA_s \dasharrow \cA_{\mu_k s}$ are the rational maps defined by
\begin{equation}\label{eq:clustertrans}
\mu_k^* z^n = z^n(1+z^{e_k})^{\langle e_k,n \rangle}, \quad \mu_k^* z^m = z^m(1+z^{\{e_k,-\}})^{-\langle e_k, m\rangle},
\end{equation}
where $\langle e_k,n \rangle$ denotes the skew-symmetric pairing on $N$ and $\langle e_k, m \rangle$ the evaluation pairing. \end{definition}

Let $T$ be an infinite $|I|$-ary tree with edges labeled by $I$ so that the edges incident to a given vertex have distinct labels.  Fix a root $t_0 \in T_0$ and label it by the seed $s$.  Label the remaining $t\in T_0$ by seeds $s_t$ such that if $t$ and $t'$ are connected by an edge labeled $k$, and $t'$ is farther from $t_0$ than $t$, then $s_{t'} = \mu_k s_t$. 

\begin{definition}\label{def:clusterstructure}
A \textbf{cluster $\cX$-structure} on $Y$ is a collection $\{\cX_{s_t} \into Y\}_{t \in T_0}$ of open maps such that the images of $\cX_{s_t}$ and $\cX_{\mu_k s_t}$ are related by a cluster $\cX$-transformation for all $t$, $k$. A partial cluster $\cX$-structure is the same but with maps only for a subset of $T_0$, and a cluster $\cA$-structure the same but with $\cA$-tori and $\cA$-transformations. \end{definition}

\begin{remark}
Though we allow the $e_i$ to not be a basis of $N$, there is a map of seeds $\overline{s} := (\Z^n = \Z\{e_i\},\{e_i\}) \to (N,\{e_i\})$, where $\Z^n$ carries the skew-symmetric form pulled back from $N$. The associated tori are $\cA$- and $\cX$-tori in the standard sense of \cite{FG2}, and are related to those of $s$ by the commutative square
\[
\begin{CD}
\cA_s   @<<< \cA_{\overline{s}} \\
@VVV @VVV \\
\cX_s @<<<  \cX_{\overline{s}}
\end{CD}
\]
\end{remark}

We will require a slightly more general notion, in which signs may appear in the cluster transformations. 

\begin{definition}
A signing on a seed is a function $\sigma: \{e_i\} \to \Z/2\Z$.  Signings transform under mutation
by  $\sigma(e_i') = \sigma(e_i)$.  A $\sigma$-signed cluster structure is like a cluster structure, 
but with the transformation law 
$$ \mu_k^* z^n = z^n(1+ (-1)^{\sigma(e_k)} z^{e_k})^{\{e_k,n\}}, 
\quad \mu_k^* z^m = z^m(1+ (-1)^{\sigma(e_k)}  z^{\{e_k,-\}})^{-\langle e_k, m\rangle}
$$
\end{definition}

The set of signings on $(N,\{e_i\})$ can be identified with $\Hom(\Z\{e_i\},\Z/2\Z)$, hence it carries an action by $\Hom(N,\Z/2\Z)$ via pullback along $\Z\{e_i\} \to N$. Acting by $\Hom(N,\Z/2\Z)$ on the tori $\cX_{s_t}$ identifies cluster structures with signings differing by this action. If the $e_i$ are a basis of $N$ all signings are equivalent in this sense, so a signed cluster $\cX$-structure is only a nontrivial notion when this is not the case.

We now return to the setting of $\cC$ a curve configuration on a surface $\cL$, $W$ the associated 4-manifold with seed skeleton $\L$. Our terminology for $\L$ is justified by having an obvious seed associated to it, namely $(H_1(\cL,\Z),\{[C_i]\})$, where we take only the classes of the closed curves $C_i \in \cC^\circ$ and $H_1(\cL,\Z)$ carries its intersection pairing. Recall that the definition of the Kashiwara-Schapira sheaf, hence the category $\muloc(\L)$, implicitly depends on a function $\sigma: \cC^\circ \to \Z/2\Z$.

\begin{theorem}\label{thm:clustertheorem}
The moduli space $\cM_1(\L)$ carries a partial, $\sigma$-signed cluster $\cX$-structure with initial seed $(H_1(\cL,\Z),\{[C_i]\})$.
\end{theorem}

\begin{proof}
This is basically just a summary of the previous sections' results in different terminology. The association of seeds to curve configurations intertwines mutations of both objects by Proposition \ref{prop:globalcurvemutation}. Since $N = H_1(\cL,\Z)$ we have $\cX_{s} \cong \Loc_1(\cL)$ (up to stabilizers), and there is an open map $\Loc_1(\cL) \to \cM_1(\L)$ by Definition \ref{def:classical}. Mutation at a simple closed curve $C_k$ induces another map $\Loc_1(\cL') \to \cM_1(\L)$, and the two are intertwined by a (possibly signed) cluster $\cX$-transformation by Proposition \ref{prop:clustertrans}. Iterated mutation of the curve configuration in general only induces a partial cluster structure since it may not be possible to perform arbitrary mutation sequences without creating self-intersections, following the discussion in Section \ref{sec:iteration}. 
\end{proof}

\begin{remark}
The choice of phrasing the previous theorem in terms of $\cM_1(\L)$ rather than $\R \cM_1(\L)$ is basically cosmetic; it would be perfectly natural to say the maps $\R \Loc_1(\cL) \to \R \cM_1(\L)$ assemble into a cluster structure on $\R \cM_1(\L)$.
\end{remark}

While the classes of the closed curves in $\cC$ are the only ones at which we can perform mutations, our definitions allow half-disks attached to open curves ending on $\partial \cL$. This would be necessary to discuss gluing of skeleta, which would be the skeletal version of the amalgamation process of \cite{FG2}.

Just as the rank one moduli space $\cM_1(\L)$ is a recepticle for maps from the torus $\Loc_1(\cL)$ of rank one local systems, the higher rank moduli spaces $\cM_n(\L)$ receive maps from spaces of higher rank local systems. The transition functions between these are determined by Proposition \ref{prop:clustertrans} as they are in the rank one case. Adopting the notation of Section \ref{sec:localsystems}, these are:

\begin{definition}
Let $C \subset \cL$ be an oriented simple closed curve. The rank-$n$ cluster $\cX$-transformation $\mu_C: \Loc_n(\cL) \dasharrow \Loc_n(\cL)$ is the following birational map. It is regular on the local systems $\cE$ for which  $\Id - \:\cE_{C_k}$ is invertible. Given such an $\cE$, its image $\cE' := \mu_C(\cE)$ is determined by the following properties:
\begin{itemize}
\item If $\gamma$ does not meet $C$ then $\cE'_\gamma = \cE_\gamma$.
\item Suppose $\gamma$ crosses $C$ exactly once, with $C$ oriented to the right of $\gamma$. Let $\gamma_{<C}$ be the subpath that starts at $\gamma(0)$ and ends at the crossing, $\gamma_{> C}$ the subpath which starts at the crossing and ends at $\gamma(1)$. Then
\[
\cE'_\gamma = \cE_{\gamma_{>C}}(\Id - \:\cE_{C})\cE_{\gamma_{< C}}.
\]
\end{itemize}
\end{definition}

With this in mind, we say that $\cM_n(\L)$ has a rank-$n$ (partial, signed) cluster $\cX$-structure by analogy with Theorem \ref{thm:clustertheorem}.

\newpage
\section{Examples from Almost Toric Geometry}\label{sec:almosttoric}
We recall the setting of almost toric geometry, a formalism for describing certain
Liouville integrable systems.
An example with no degenerate fibres is the restriction of the moment map of a toric variety
to the interior of the moment polytope.  More general almost-toric fibrations
\cite{LSy, Sym, KoS-aff} describe the situation in which certain degenerate fibres are allowed.  We restrict ourselves here to the case where the total space is 4 real dimensional. 

One begins with the moment map of a toric algebraic surface, $\overline{W} \to \overline{\Delta}$,
and possibly makes non-toric blowups along the boundary divisors.   Thus the total space 
of the boundary divisor has many nodes.  Degenerate fibres can be introduced into the interior
by the so-called ``nodal trade'', in which these singularities are pushed into interior fibers,
and the boundary divisor is correspondingly smoothed.  We then restrict attention to the interior
$W \to \Delta$. 

Thus for us the data of an almost-toric fibration is: the interior of a polygon, a certain number of marked points $d_1, \ldots, d_n$, branch cuts
from these to the boundary of the polytope, and specifications of how the integral affine structure
of the polygon changes along the branch cuts.  From this data, an almost toric fibration can be 
uniquely reconstructed: the $d_i$ sit below the singular fibers, and 
under the identification of the connection on $H_1(\mathrm{fibre})$ with the 
affine structure from the action coordinates, the changes across the branch 
cut specify the monodromy.  

As we are working in the complement of the boundary of the polytope, the total space of the
fibration is an exact symplectic manifold.  There is a unique point $0 \in \Delta$ above 
which the fiber is a smooth exact Lagrangian $\Sigma$.  
The Lagrangian thimbles 
 above the straight lines from $d$ to the degenerate fibers give Lagrangian 
disks $D_i$ which end on curves $C_i$ on $\Sigma$.  The total space of the almost-toric
fibration is the same as the space built from the $\Sigma$ and $C_i$ according to Definition
\ref{def:W}.  The union of $\Sigma$ and the $D_i$ is our Lagrangian skeleton $\L$. 

By appropriate change of presentation of the base as in \cite[Sec. 5.3]{Sym},
it can be arranged that the points $d_i$ and 
the branch cuts are in the complement of a neighborhood of $0$.
Beginning in this situation, we can consider the deformation of the integrable system
which brings some given branch point $d_i$ to $0$ along the line connecting them, and
then past.  Watching from the point of view of the unique exact fibre, one sees the 
curve $C_i$ collapse to a point and then re-expand -- that is, one sees precisely the Lagrangian
disk surgery of \cite{Y}.  Afterward, the configuration no longer satisfies the constraint that 
all branch cuts stay away from $0$; again a manipulation of the integral affine structure can 
restore this situation, at the cost of cutting and regluing the polytope $\Delta$.  For details, see
\cite[Sec. 2.3]{Via-one}.

\begin{example}
[Torus With One Disk]\label{ex:torusdisk}
Begin with the moment map $\P^2 \to \overline{\Delta}$, and make the ``nodal trade'' at one vertex of
the closed triangle $\overline{\Delta}$.  The resulting fibration, restricted to the interior of the triangle, has a single degenerate fibre.  Taking the  
thimble over the line connecting the degenerate fibre to 
the exact fibre $\Sigma$, we find that the skeleton is a torus with a disk attached.  

We are working in the complement of the boundary.  Having made the nodal trade at one vertex
has smoothed it out one of the boundary divisors, so the space $W$ is the complement in $\P^2$
of the union of a quadric and a line.  That is, it is the affine surface $W = \{ (x,y)\,|\, xy-1 \neq 0\} \subset \C^2$.
We refer to \cite[Prop. 11.8]{Sei2} for an account of this space via a Lefschetz pencil
presentation.   
\end{example}

\begin{example}[Vianna's Tori]\label{ex:P2}
Begin again with the moment map $\P^2 \to \overline{\Delta}$, and make the ``nodal trade'' at all
three vertices of the triangle.  Having smoothed out all singularities of the boundary divisor,
the fibration $W \to \Delta$ has total space equal to the complement of a smooth elliptic curve. 

Taking thimbles from the three degenerate fibres to the exact fibre $\Sigma = T^2$
gives three curves.  To determine which curves they are, consider the anticanonical moment
map image triangle with vertices $(-1,-1), (2,-1), (-1, 2)$.  The thimble
in the direction $(a, b)$ of the affine
structure on $\Delta$ is carried by the class $(-b, a)$ in $H_1(T^2, \Z)$, up to some universal
choice of sign conventions.  Thus, our curves are in the classes 
$(1, -1), (1, 2), (-2, -1)$. 

Vianna constructs infinitely many monotone tori in $\P^2$ \cite{Via}.
His construction can be identified with the iterated disk surgery we have given here, 
via the dictionary described above 
between the disk surgery prescription and degenerations of the almost toric picture.  

Assuming $Fuk(W) \cong \mu loc(\L)$, one can give a 
cluster-theoretic reason why there are infinitely many inequivalent tori here.
Indeed, Hamiltonian isotopic tori necessarily give rise to the same objects in $Fuk(W)$, 
hence the same cluster charts on $\cM_1(\L)$.  On the other hand, each cluster chart gives rise to a torus as we have described.  The cluster structure on $\cM_1(\L)$ is determined by the quiver defined by the intersection pairings: an oriented 3-cycle with all of its edges tripled.  As this is not a Dynkin quiver, 
there are infinitely many clusters hence infinitely many distinct tori \cite{FZ2}. 
The above argument would show that these tori are non-isotopic in $W$.  More precisely, one should also argue that distinct tori in the usual $\cX$-variety have distinct intersections with the image of the $\cA$-variety, as this image is what is directly related to $\cM_1(\L)$ by Theorem \ref{thm:clustertheorem}.  We note that 
Vianna proves the stronger statement that the corresponding monotone tori 
are not Hamiltonian isotopic even in $\P^2$. 
\end{example}

\begin{example}[Keating's Tori] 
Consider the algebraic surface singularity 
$x^p + y^q + z^r + axyz = 0$, where $\frac{1}{p} + \frac{1}{q} + \frac{1}{r} \le 1$.  We write $T_{p,q,r}$ for its Milnor fibre; this is 
a Stein space, hence a Weinstein 4-manifold.  This space was studied in \cite{Kea}, where exact
Lagrangian tori are constructed for the purpose of showing that vanishing cycles cannot
split generate the Fukaya category of $T_{p,q,r}$.  

These spaces can be constructed from the almost toric point of view as follows. 
Begin with $\P^2 \to \overline{\Delta}$.  Blow up the three boundary divisors, respectively,
$p$, $q$, and $r$ times at distinct points.  Make the nodal trades, so as to introduce 
$p + q + r$ degenerate fibres.  On the corresponding exact Lagrangian fibre, the thimbles
determine $p$ + $q$ + $r$ curves on $\Sigma$ in the respective classes $(0,-1)$,  $(1,0)$, $(-1,1)$. 
The corresponding quiver has $p + q + r$ vertices; one arrow
from each of the $p$ vertices to each of the $q$ vertices, one arrow from each of the $q$ vertices
to each of the $r$ vertices, and one arrow from each of the $r$ vertices to each of the $p$ vertices,
for a total of $pq + qr + rp$ arrows, which participate in $pqr$ 3-cycles. 
The tori resulting from the present construction 
are those of \cite{Kea}.  
\end{example}

In these almost toric examples, we are working in the complement of a divisor linearly equivalent to 
the total transform under blowup of the toric boundary, hence the spaces $W$ are (log)
Calabi-Yau surfaces.  By construction they come with fibrations by Lagrangian tori,
i.e., are in the setting appropriate to \cite{SYZ} mirror symmetry.  This has been studied
for these surfaces using both tropical \cite{GHK-mir} and to some extent symplectic \cite{Pas}
techniques.  

From the above description, it can be seen that the moduli space $\cM_1(\L)$ is, in the almost
toric case, meant to be the moduli space of objects in the Fukaya category of $W$ in
the class of a torus fibre.  That is, it is the SYZ mirror, 
and our perspective matches that of \cite{GHK,GHKK}. 
\appendix
\section{Appendix: Background on constructible sheaves}\label{app:sheaves}
A sheaf can be thought of as a family of (complexes of) $\coeffs$-modules, parameterized by
$X$; in particular, for a sheaf $\cF$ and a point $x \in X$, there is a $\coeffs$-module 
$\cF_x$ called the stalk of $\cF$ at $x$.   Sheaves are the natural 
coefficients for Cech cohomology; i.e., it makes sense to write $H^*(X, \cF)$.  
Constructible sheaves are those for which there is some stratification 
of $X$ such that the $\cF_x$ remain locally constant along strata.

In this section, we very briefly recall from \cite{KS} some formal manipulations
which can be performed on constructible sheaves: 
functors between sheaf categories and their basic properties; integral
kernels; microsupport; and from \cite{GKS}, the action of contact isotopy on
the sheaf category.  

Given a manifold or stratified space $X$ and a commutative ring $\coeffs$, we write $sh(X;\coeffs)$ (or simply $sh(X)$) for the dg version of the derived category of constructible sheaves of $\coeffs_X$-modules on $X$.
That is, it is the quotient \cite{Kel3, Dri} of the dg category of complexes of sheaves with constructible cohomology by 
the acyclic complexes.  
The formalism of \cite{KS} makes good sense in this context; we refer to \cite{N1} for 
some details on this point, which we ignore for the remainder of the appendix. 

\subsection{Functors between sheaf categories}

\subsubsection{Six functors}

Given a space $X$, and a sheaf $\cF$, there are adjoint functors
$$\cdot \otimes \cF :sh(X) \leftrightarrow sh(X) : \underline{\Hom}(\cF, \cdot) $$

Given a map of spaces $f: Y \to X$, one obtains two pairs of adjoint functors
$$f^*: sh(X) \leftrightarrow sh(Y): f_*$$
$$f_!: sh(Y) \leftrightarrow sh(X): f^!$$

The left adjoints are easier to understand at the level of stalks: for a point $p$,
\begin{eqnarray*}
(\cF \otimes \cF')_p & = &  \cF_p \otimes \cF'_p \\
(f^*\cF)_p & = & f^*(\cF_p) \\ 
f_!(\cG)_p & = & H_c^*(f^{-1}(p), \cG) 
\end{eqnarray*}

The right adjoints are easier to understand at the level of sections: 
for an open set $U$, 
\begin{eqnarray*}
H^*(U, \underline{\Hom}(\cF, \cF')) &  =  &
\Hom(\cF|_U, \cF'|_U) \\
H^*(U, f_* \cG) &  = & H^*(f^{-1}(U), \cG) \\
H^*(U, f^! \cF) & = & \D H^*_c(U, f^* \D \cF) 
\end{eqnarray*}

The last line is written in terms of the Verdier duality operation ---
an
anti-involution $\D: sh(M;  \coeffs) \to sh(M;  \coeffs)$.  
It interchanges shrieks and stars ---
$\D f_* = f_! \D$ and $\D f^* = f^! \D$ --- so can be used to calculate the 
shriek pullback.  

The shrieks and stars are directly related in two cases: 
when $f$ is proper, we have $f_! = f_*$; when $f$ is a smooth fibration, the sheafification
of Poincar\'e duality asserts
$f^! = f^*[\dim Z - \dim Y]$.  It follows by considering the map to a point that 
$\D \coeffs_M = \coeffs_M [\dim M]$; one recovers the usual Poincar\'e duality 
from this as
$H^*(M, \coeffs_M[\dim M]) = \D H_c^*(U, \coeffs_M)$, where now the operation $\D$ is
just the linear duality of complexes of vector spaces. 

If $\overline{M}$ is a manifold with boundary, and $j: M \to \overline{M}$ is the inclusion of its
interior, then note that $j_* \coeffs_{M} = \coeffs_{\overline{M}}$.  Taking Verdier duals, we see
$j_! \coeffs_{M}[\dim M] = \D \coeffs_{\overline{M}}$.  We use this in the following
form: if $f: \overline{M} \to N$ is any map to a manifold (without boundary),  then 
$$f^! \coeffs_N = f^! \D \coeffs_N[\dim N] = \D f^* \coeffs_N[\dim N] = \D \coeffs_M[\dim N]
= \coeffs_M[\dim M - \dim N]$$ 

The contortions above to compute $f^!$ are inevitable 
-- this operation has a certain irreducible
complexity (which can be hidden inside the Poincar\'e-Verdier duality, but this in turn is nontrivial
to compute).  However, when $f$ is
the inclusion of a closed subset, $f^!$ extracts the sections supported on that
subset.

\subsubsection{Base change}

Given another map $g: Z \to Y$, we write also $g: Z \times_Y X \to X$ and
$f: Z \times_Y X \to Z$ for the maps induced on the fibre product.  The base change theorems
assert the following relations: $f_! g^* = f_! g^*: Sh(X) \to Sh(Z)$ and similarly the other three
$f_* g^! = g^! f_*$, and $g_! f^* = f^* g_!$, and $g_* f^! = f^! g_*$.  

E.g., if $g: U \to Y$ is an open inclusion, then
$$g^* f_* \cF = g^! f_* \cF = f_* g^! \cF = f_* g^* \cF = f_*(\cF|_{f^{-1}(U)})$$
Taking global sections (i.e. cohomology, i.e. pushing forward to a point), one has
$H^*(U, f_* \cF) = H^*(f^{-1}(U), \cF)$.  This is usually given as the definition of $f_*$. 
Taking $U = Y$ and $\cF = \coeffs$, we have $H^*(Y, f_* \coeffs) = H^*(X, \coeffs)$. 
Expanding out $f_* \coeffs$ into its cohomology sheaves gives rise to the Leray spectral
sequence; this ability to factor cohomological calculations is one of the main 
classical uses of sheaf theory.

\subsubsection{Recollement}

Consider the inclusion of an open subset $U$ and its closed complement $V$ into $Y$, 
$$j: U \rightarrow Y \leftarrow V: i$$
Here, $j^* = j^!$ and $i_* = i_!$.  Because $U \cap V = \emptyset$, all compositions 
involved in the base change formula vanish.  Moreover (because $U$ and $V$ cover), we
have exact triangles 
$$ i_! i^! \to \mathbf{1} \to j_* j^* \xrightarrow{[1]}  \qquad \qquad \qquad
j_! j^! \to \mathbf{1} \to i_* i^* \xrightarrow{[1]}$$

These sequences are the sheaf-theoretic incarnations of excision: applied to the
constant sheaf on $Y$ and pushed forward to a point, one recovers
$$ H^*(V, i^! \coeffs) \to H^*(Y, \coeffs) \to H^*(U, \coeffs) \xrightarrow{[1]}
\qquad \qquad \qquad H^*_c(U, \coeffs) \to H^*_c(Y, \coeffs) \to 
H^*_c(V, \coeffs) \xrightarrow{[1]}$$

These allow us to understand the category $sh(Y)$ in terms of the categories $sh(U)$, 
$sh(V)$, and the data prescribing the connecting morphism in one of 
the above exact triangles; see e.g. \cite[1.4.3]{BBD}.
Nadler has suggested of how this formalism may be applied to
the Fukaya category in \cite{N2}.  

Recollement has historically been used for gluing sheaf categories together from local pieces.
Note however that our approach in the present article has been at least grammatically different: we view the 
global category as the global sections of a sheaf of categories, hence work with an open
cover rather than with a decomposition as above. 

\subsubsection{Functors from kernels}
Additional functors between sheaf categories can be constructed via the formalism
of kernels \cite[Sec. 3.6]{KS}.  This works as follows.  Consider the product
 $$Y \xleftarrow{\pi_Y} Y \times X \xrightarrow{\pi_X} X$$ 
 Given any sheaf $\cK$ on $Y \times X$, one gets
two pairs of adjoint functors: 

$$\cK^*: sh(X) \leftrightarrow sh(Y): \cK_*$$
$$\cK_!: sh(Y) \leftrightarrow sh(X): \cK^!$$

Their definitions are as follows.  Let $\cG$ be a sheaf on $Y$ and $\cF$ a sheaf on $X$. 

\begin{eqnarray*}
\cK^* : \cF & \mapsto & \pi_{Y!} (\cK \otimes \pi_X^* \cF) \\
\cK_! : \cG & \mapsto & \pi_{X!} (\cK \otimes \pi_Y^* \cG) \\
\cK_* : \cG & \mapsto & \pi_{X*} \underline{Hom}( \cK, \pi_Y^! \cG) \\
\cK^! : \cF & \mapsto & \pi_{Y*} \underline{Hom}( \cK, \pi_X^! \cF) \\
\end{eqnarray*}

\begin{example}
Let $f: Y \to X$ be a map.  Let $\coeffs(f) \in sh(Y \times X)$ be the constant sheaf on the graph
of $f$.  Then $\coeffs(f)^* = f^*$, $\coeffs(f)_* = f_*$, 
$\coeffs(f)^! = f^!$, $\coeffs(f)_! = f_!$
\end{example}

\begin{example} (Fourier-Sato transform)
Let $\Phi$ be the constant sheaf on the locus $\{\bf{x} \cdot \bf{y} \ge 0\,|\, (\bf{x}, \bf{y})
\in \R^n \times \R^n\}$.  This defines the so-called Fourier-Sato transforms
$$\Phi^* = \Phi_!: sh(\R^n) \to sh(\R^n)$$
$$\Phi_* = \Phi^!: sh(\R^n) \to sh(\R^n)$$
These transforms are generally considered restricted to the subcategory of conic sheaves, 
i.e., sheaves which are constant along any open ray emanating from the origin.  Here, 
the Fourier transform squares to pull-back by the antipodal map.  Its inverse is given by 
the kernel $\{\bf{x} \cdot \bf{y} \le 0\,|\, (\bf{x}, \bf{y})
\in \R^n \times \R^n\}$
See \cite[Sec. 3.7]{KS} for more. 
\end{example}

The functors induced by the kernel can be composed by convolving the kernels.  That
is, if one has $\cK' \in sh(Z \times Y)$ and $\cK \in sh(Y \times X)$, then with the evident notation
for projection to the factors, one defines
$$\cK' \circ \cK := \pi_{ZX!}(\pi_{ZY}^* \cK' \otimes \pi_{YX}^* \cK) \in sh(Z \times X)$$
This has the properties 
$$(\cK' \circ \cK)_* = \cK'_* \circ \cK_*, \qquad (\cK' \circ \cK)_! = \cK'_! \circ \cK_!, 
\qquad (\cK' \circ \cK)^* = \cK^* \circ {\cK'}^*, \qquad (\cK' \circ \cK)^! = \cK^! \circ {\cK'}^!.$$

\subsubsection{Cutoff functors} \label{sec:cutoff} 
Let $j: U \subset X$ be the inclusion of an open subset.  
Then $j_!$ and $j_*$ are fully faithful.  

Assume in addition that $X \setminus \overline{U}$ is cylindrical, and equipped with the structure
$X \setminus \overline{U} = \R_+ \times \partial U$.  We say a sheaf is 
{\em cylindrical past the boundary} if its restriction to any ray $\R_+ \times u \subset
\R_+ \times \partial U = X \setminus \overline{U}$ is constant.  

\begin{lemma}
The restriction of $j^! = j^*$ to the subcategory of cylindrical sheaves 
is an equivalence of categories.  
\end{lemma}

Note however that $j_! j^!$ and $j_* j^*$ are not in this
case the identity.  We think of them as cutoff functors, and call the first the ``soft cutoff''
and the second the ``hard cutoff''. 
In case $U$ has many boundary components, we can decide independently on 
each whether to use  $j_!$ or $j_*$.  

Informally: when 
we are describing microlocal sheaf categories on some manifold $X$ with cylindrical end,
it is equivalent to allow the sheaves to go to infinity, or to cut them off at some point in
the cylinder, and moreover each cutoff can be co-oriented arbitrarily. 

\subsection{Microsupport}

Let $\cF$ be a sheaf on a manifold $M$.   The microsupport  $ss(\cF)$, introduced in \cite{KS}, 
is meant to capture the locus in $T^*M$ of obstructions to the propagation of sections of $\cF$. 
For instance, if $f: M \to \R$ is a function such that the graph of $df$ avoids $ss(\cF)$ over the 
locus $f^{-1}((a, b])$, then the restriction of sections is an isomorphism \cite[Prop. 5.2.1]{KS}:
 $$H^*(f^{-1}(-\infty, b], \cF) \xrightarrow{\sim} H^*(f^{-1}(-\infty, a], \cF)$$
The formal definition is a local version of the above criterion: 

\begin{definition} \label{def:microsupport} \cite[Chap. 5]{KS}
A point $p = (x, \xi) \in T^*M$ is in the microsupport of a sheaf $\cF$ if there are 
points $(x', \xi')$ arbitrarily close to $(x, \xi)$ and functions $f: M \to \R$
with $f(x') = 0, df(x') = \xi'$, such that:
if $c_f: \{x\,|\, f(x) \ge 0\} \to M$ is the inclusion, then $(c_f^! \cF)_{x'} \ne 0$. 
\end{definition}

Shriek pullback to a closed subset gives the local sections  
supported on that subset.  Thus the statement $(c_f^! \cF)_{x'} \ne 0$ is informally
read as: ``there is a section of $\cF$ beginning at $x'$ and propagating in 
the direction along which $f$ increases.''  Note that, taking the zero function, the support
of $\cF$ is contained in its microsupport. 

For us, microsupports are used as a way to specify certain categories of sheaves.  For a 
subset (usually conical Lagrangian) $L$ 
in $T^*M$, we write $sh_L(M; \coeffs)$ for the category of sheaves
on $M$ with coefficients in $\coeffs$ and microsupport in $L$.  For instance, the category
of local systems on $M$ is $sh_{0_M}(M; \coeffs)$. 

\subsubsection{Properties of the microsupport.} 

For sheaves constructible with respect to a given stratification, 
it is straightforward to show that the microsupport is contained within the union of the
conormals of the strata.  Since the microsupport is co-isotropic, it is in this case Lagrangian,
and necessarily a full dimensional subset of the union of conormals.

Finally, note that, 
per the definition, to show that $(x, \xi)$ is not in the microsupport, one needs to check a 
property of every function vanishing near $x$ with derivative near $\xi$, 
at every point near $x$. 
In fact, it is enough to check a function $f$ which is
stratified Morse at $x$.  Such functions need not exist for all $(x, \xi)$ with respect
to a given stratification, but because they will exist for general points in each
component of the microsupport.  Since microsupports are closed, they 
can be computed with stratified Morse functions \cite{GM}. 

Many properties of the microsupport are developed in \cite{KS}.  In particular: 
writing $\D \cF$ for the Verdier dual of $\cF$, the microsupports $ss(\cF)$ 
and $ss(\D \cF)$ are related by the antipodal map on cotangent fibres, and given 
an exact triangle, $A \to B \to C \xrightarrow{[1]}$, one has
$$ (ss(A) \setminus ss(B)) \cup (ss(B) \setminus ss(A)) \subset  ss(C) \subset ss(A) \cup ss(B)$$

Microsupport interacts well with integral kernels.  Given a conical Lagrangian $M \subset 
T^*(Y \times X)$, we have the convolution 
\begin{eqnarray*}
M_\star: ConLag(T^*Y) & \to & ConLag(T^*X) \\
L & \mapsto & \pi_X( \pi_Y^{-1} L \cap M) 
\end{eqnarray*}
and given a kernel $\cK \in sh(Y \times X)$ satisfying certain properness and non-characteristic
hypotheses (see \cite[Eq. 1.1]{GKS}) one has 
$$ss(\cK_! \cF) \subset ss(\cK)_\star ss(\cF)$$
The effect on microsupport of the other functors can be determined from
the above using Verdier duality and transposition.  In particular, since the
functors associated to a map $Y \to X$ are convolution with the conormal
to the graph, the above formula determines their effect on microsupports. 

We record the effect on microsupport of the hard and soft cutoff
functors. 

\begin{lemma} \label{lem:cutsupport}
Let $u: U \hookrightarrow M$ be the inclusion of an open set into a manifold
extending to an inclusion of a manifold with boundary $\overline{u}: \overline{U} 
\hookrightarrow M$.  Give $\partial U$ the boundary coorientation -- its positive conormal direction is out.  Let $\cF \in sh(M)$ be conical past the boundary.  Then:
$$ss(u_* u^* \cF) \subset ss(\cF)|_{\overline{U}} \cup T^+_{\partial U}M$$
$$ss(u_! u^! \cF) \subset ss(\cF)|_{\overline{U}} \cup T^-_{\partial U}M$$
\end{lemma}

\subsection{Contact isotopies}

Consistent with the expectation that constructible sheaves model the Fukaya category,
contact isotopies of $T^\infty M$ act on $sh(M)$. 
We recall the result as formulated in \cite{GKS}.  

\begin{theorem}\label{thm:GKS} \cite{GKS} 
Let $M$ be a manifold, $I$ an interval,
$T^\circ M$ the cotangent bundle minus the zero section, and $\Phi: T^\circ M \times I 
\to T^\circ M$ a smooth map.  Assume $\Phi( \cdot, 0)$ is the identity, and 
$\Phi(\cdot, t)$ is a homogenous (i.e. commutes with the scaling) symplectomorphism
for each $t$.  

Then there exists a unique closed conic Lagrangian in 
$\Lambda \subset T^\circ (M\times M \times I)$ such that 
$\Lambda_\star T_t^* I  \subset T^\circ (M\times M)$ is the graph of $\Phi(\cdot, t)$. 

Moreover, there exists a unique locally bounded sheaf 
$K_{\Phi}  \in sh(M \times M \times I)$ such that $ss(K_\Phi) = \Lambda$
and $K_\Phi|_{M \times M \times 0}$ is the constant sheaf on the diagonal. 
\end{theorem}

\begin{remark} A homogenous symplectomorphism can be given (up to rescaling)
by a contactomorphism at infinity.
\end{remark}

\begin{corollary}\label{cor:GKSequivalence} \cite{GKS} 
Convolution with the kernel $K_\Phi|_t$ induces an equivalence of categories 
$sh(M) \to sh(M)$.  Away from the zero section, the microsupport of the image of a sheaf under convolution is the image of its microsupport under $\Phi(\cdot, t)$.
\end{corollary}
\begin{proof}
We indicate how to derive this from Theorem \ref{thm:GKS}.
The inverse is given by the kernel coming from the inverse family of symplectomorphisms:
convolving the two families of kernels gives a kernel whose microsupport must lie in 
the conormal to the diagonal and is the constant sheaf there at time zero. 
\end{proof}

This is an extremely powerful
tool, and was used in \cite{GKS} to prove various non-displaceability theorems. 
In the present paper we use it to define mutation functors in Section \ref{sec:mutationfunctors}: given
a neighborhood of a Lagrangian skeleton small enough to be embeddable into a 
cotangent bundle, the result of \cite{GKS} amounts to an assertion that the microlocal sheaf category
depends only on the (singular) Legendrian skeleton at the boundary of this neighborhood. 
Thus we may isotope around this Legendrian at will, changing the 
topology of its Lagrangian cone in the process.  

\begin{example} (Reeb flow)
Let $\Phi(\cdot, t)$ be the flow of the Hamiltonian 
$H(\mathbf{q}, \mathbf{p}) = \mathbf{p}^2$ on $T^* \R^n$. 
Then for $t \ge 0$, the kernel $K_\Phi|_t$ is given by the constant sheaf
on the locus $|\bf{x} - \bf{y}| \le t$ in $\R^n \times \R^n$.  Convolution with this kernel acts as
an averaging operator: the stalk after convolution at $\bf{x}$ is the global sections
of the sheaf over the radius $t$ ball around $\bf{x}$.
\end{example}

\newpage

\bibliographystyle{plain}

\end{document}